\documentclass[10pt, a4paper]{article} 
\author{
	Francisco Gancedo
	\\{\footnotesize Departamento de An\'alisis Matemático \& IMUS}
	\\{\footnotesize Universidad de Sevilla}
	\\{\footnotesize Sevilla, Espa\~na}
	\\{\footnotesize email: {\it fgancedo@us.es}}
	\and 
	Rafael Granero-Belinch\'on
	\\{\footnotesize Departamento de Matem\'aticas, Estad\'istica y Computaci\'on}
	\\{\footnotesize Universidad de Cantabria}
	\\{\footnotesize Santander, Espa\~na}
	\\{\footnotesize email: {\it rafael.granero@unican.es}}
	\and 
	Stefano Scrobogna
	\\{\footnotesize Basque Center for Applied Mathematics}
	\\{\footnotesize Alameda de Mazarredo, 14}
	\\{\footnotesize Bilbao, Espa\~na}
	\\{\footnotesize email: {\it sscrobogna@bcamath.org}}
}
\usepackage{amsfonts}

\usepackage{color}
\usepackage{tikz} 
\usepackage{graphicx}  
\usepackage[all]{xy} \xyoption{arc} \xyoption{color}
\usepackage{epsfig}
\usepackage{amsbsy}

\usepackage[english]{babel}
\usepackage{pst-grad} 
\usepackage{pst-plot} 
\usepackage{pstricks}
\usepackage{amsmath,amssymb,mathrsfs,amsthm, mathtools}
\usepackage{stmaryrd}
\usepackage{tikz} 
\usepackage{mathcomp,wasysym}  
\usepackage{graphicx}  
\usepackage[all]{xy} \xyoption{arc} \xyoption{color}
\usepackage{epsfig}
\usepackage{cite}
\usepackage{enumerate}
\usepackage{geometry}
\usepackage{empheq}
\usepackage{times}
\DeclareMathAlphabet{\mathcal}{OMS}{cmsy}{m}{n}

\newcommand{\dx}{\textnormal{d}{x}}
\newcommand{\dy}{\textnormal{d}{y}}
\newcommand{\dt}{\textnormal{d}{t}}
\newcommand{\ddt}{\frac{\textnormal{d}}{ \textnormal{d}{t}}}

\renewcommand{\div}{\textnormal{div}}

\newcommand{\sgn}{\textnormal{sgn}}
\newcommand{\pare}[1]{\left( #1 \right)}
\newcommand{\norm}[1]{\left\| #1 \right\|}
\newcommand{\av}[1]{\left| #1 \right|}
\newcommand{\bra}[1]{\left[ #1 \right]}

\newcommand{\set}[1]{\left\{ #1 \right\}}

\newcommand{\nd}{\nabla_\delta}
\newcommand{\Dd}{\Delta_\delta}

\DeclareMathOperator*{\esssup}{\textnormal{ess\,sup}}

\newcommand{\cA}{\mathcal{A}}
\newcommand{\cC}{\mathcal{C}}

\newcommand{\cJ}{\mathcal{J}}

\newcommand{\cD}{\mathcal{D}}
\newcommand{\cK}{\mathcal{K}}
\newcommand{\cE}{\mathcal{E}}
\newcommand{\cS}{\mathcal{S}}

\newcommand{\bR}{\mathbb{R}}
\newcommand{\bT}{\mathbb{T}}
\newcommand{\bZ}{\mathbb{Z}}
\newcommand{\bN}{\mathbb{N}}
\newcommand{\bS}{\mathbb{S}}

\newcommand{\cG}{\mathcal{G}}
\newcommand{\cH}{\mathcal{H}}

\newcommand{\cN}{\mathcal{N}}

\newcommand{\ccS}{\mathscr{S}}
\newcommand{\ccK}{\mathscr{K}}

\normalsize
\normalsize
\usepackage[bookmarks,colorlinks]{hyperref}

\newcommand{\hra}{\hookrightarrow}

\newcommand{\loc}{\textnormal{loc}}

\theoremstyle{theorem}
\newtheorem{theorem}{Theorem}[section]
\newtheorem*{theorem*}{Theorem}
\newtheorem{prop}[theorem]{Proposition}
\newtheorem{lemma}[theorem]{Lemma}
\newtheorem{cor}[theorem]{Corollary}

\theoremstyle{definition}
\newtheorem{definition}[theorem]{Definition}

\numberwithin{equation}{section}

\title{Surface tension stabilization of the Rayleigh-Taylor instability for a fluid layer in a porous medium}

\begin{document}

	\maketitle

	\begin{abstract}
		This paper studies the dynamics of an incompressible fluid driven by gravity and capillarity forces in a porous medium. The main interest is the stabilization of the fluid in Rayleigh-Taylor unstable situations where the fluid lays on top of a dry region. An important feature considered here is that the layer of fluid is under an impervious wall. This physical situation have been widely study by mean of thin film approximations in the case of small characteristic high of the fluid considering its strong interaction with the fixed boundary. Here, instead of considering any simplification leading to asymptotic models, we deal with the complete free boundary problem. We prove that, if the fluid interface is smaller than an explicit constant, the solution is global in time and it becomes instantly analytic. In particular, the fluid does not form drops in finite time. Our results are stated in terms of Wiener spaces for the interface together with some non-standard Wiener-Sobolev anisotropic spaces required to describe the regularity of the fluid pressure and velocity. These Wiener-Sobolev spaces are of independent interest as they can be useful in other problems. Finally, let us remark that our techniques do not rely on the irrotational character of the fluid in the bulk and they can be applied to other free boundary problems.

	\end{abstract}

	{\small\tableofcontents}

	\allowdisplaybreaks

	\section{Introduction}
The Rayleigh-Taylor instability, named after Rayleigh \cite{rayleigh1878instability} and Taylor \cite{taylor1950instability}, is an interface fingering instability, which occurs when a denser fluid lies on top of a lighter fluid under the effect of gravity. This acceleration causes the two fluids to mix (see \cite{Sharp1984} and \cite{CCF2016,CFM2019}). The purpose of this paper is to study the stabilization of the Rayleigh-Taylor instability by capillary effects. In particular, in this paper we consider the motion of a underlying incompressible fluid in a porous medium (see Figure 1) under the effect of (downward pointing) gravity and surface tension effects. This setting is known in the literature as the Muskat problem. In particular, we are interested in the study of the possible formation of drops or fingering phenomena in a thin fluid layer. 

	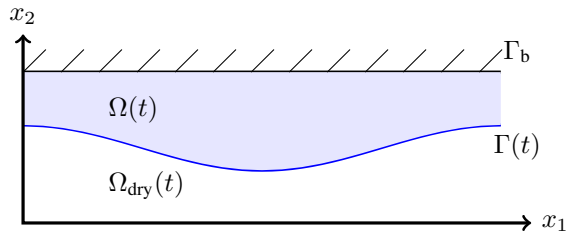
\begin{figure}[h]
	\centering
	\begin{tikzpicture}[domain=0:2*pi, scale=1] 
	\draw[color=black] plot (\x,{0.3*cos(\x r)+1}); 
	\draw[very thick, smooth, variable=\x, blue] plot (\x,{0.3*cos(\x r)+1}); 
	\draw[very thick, smooth, variable=\x, black] plot (\x,2); 
	\fill[blue!10] plot[domain=0:2*pi] (\x,2) -- plot[domain=2*pi:0] (\x,{0.3*cos(\x r)+1});
	\draw[very thick,<->] (2*pi+0.4,0) node[right] {$x_1$} -- (0,0) -- (0,2.5) node[above] {$x_2$};
	\coordinate[label=above:{$\Gamma(t)$}] (A) at (6.5,0.72);
	\coordinate[label=above:{$\Gamma_{\text{b}}$}] (A) at (6.5,2);
	\node[right] at (1,1.5) {$\Omega(t)$};
	\node[right] at (1,0.5) {$\Omega_{\text{dry}}(t)$};
	\draw[-] (0,2) -- (0.3, 2.3);
	\draw[-] (0.5,2) -- +(0.3, 0.3);
	\draw[-] (1,2) -- +(0.3, 0.3);
	\draw[-] (1.5,2) -- +(0.3, 0.3);
	\draw[-] (2,2) -- +(0.3, 0.3);
	\draw[-] (2.5,2) -- +(0.3, 0.3);
	\draw[-] (3,2) -- +(0.3, 0.3);
	\draw[-] (3.5,2) -- +(0.3, 0.3);
	\draw[-] (4,2) -- +(0.3, 0.3);
	\draw[-] (4.5,2) -- +(0.3, 0.3);
	\draw[-] (5,2) -- +(0.3, 0.3);
	\draw[-] (5.5,2) -- +(0.3, 0.3);
	\draw[-] (6,2) -- +(0.3, 0.3);
	\end{tikzpicture} 
	\caption{Physical setting}
\end{figure}	

The incompressible fluid moves in a porous medium where the effect of the solid matrix is modeled by Darcy's law \cite{Darcy1856}
	\begin{equation}
	\label{eq:Darcy}
	\tag{D}
		\frac{\mu}{\kappa} u(x,t) = -\nabla p(x,t) -G \rho(x,t)\,   {\bf e}_2,
	\end{equation}
where $p$ and $\rho$ are the pressure and the density of the fluid, respectively. The constants $\mu$ and $\kappa$ are the viscosity of the fluid and the permeability of the porous media respectively. Both the viscosity and the permeability are assumed to be positive constants. We also have that $G$ is the gravitational constant and the vector ${\bf e}_2=(0,1)$. 

A completely different scenario is the evolution of a fluid in a Hele-Shaw cell. In this case the fluid is confined between two parallel flat plates separated by a small distance. Although the latter is a completely different physical setting, both problems are mathematically equivalent. In particular, the Hele-Shaw cell dynamics is given by the PDE
	\begin{equation}
	\tag{HS}\label{eq:Hele-Shaw}
		\frac{12\mu}{d^2} u(x,t) = -\nabla p(x,t) -G\rho(x,t)\,   {\bf e}_2,
	\end{equation}
	where $d$ is the distance between the plates.

A classical physical phenomenon in porous media, observed also in Hele-Shaw cells, is viscous fingering \cite{SaffmanTaylor1958}. It holds at small scales where capillarity forces play a crucial role in the dynamics of a less viscous fluids penetrating a higher viscous one. This viscous fingering phenomenon consists in the formation of finger-shape patterns at a long space scale compared with initial length of the fragment of the interface. These fingers appear where the less viscous fluid is entering inside the more viscous region. This behavior, well-known in the physics literature, has been widely observed experimentally \cite{BKLST1986,Homsy1987}. From the mathematical point of view it has been extensively studied numerically \cite{HLS1994}, modeling \cite{Otto1997} and theoretically \cite{EscherMatioc2011}. Such a fingering phenomena also occurs when the denser fluid is on top of the lighter one under the effect of gravity. Then, the question that we try to answer in this paper is whether a finger (or a drop) forms in this Rayleigh-Taylor unstable case or, at the contrary, capillary forces depletes the interface. In other words, we want to find explicit conditions on the initial interface and the depth of the fluid layer such that, when surface tension effects are considered, the solution tends to the homogeneous equilibrium.

In that regards, in this paper we prove that for this Rayleigh-Taylor unstable Muskat problem global regularity hold for initial data with medium size slope measured in an appropriate functional space (see below for the precise statements of our results). Remarkably, the result is not based on the irrotational character of the fluid given by Darcy's law and, as a consequence, the techniques that we develop here apply to more general situation and can be applied to systems which do not admit contour dynamic formulation.

\subsection{Some prior results}
There is a large literature on the Muskat and Hele-Shaw problem and, as a consequence, we cannot be exhaustive. Muskat and Hele-Shaw free boundary problems are local in time well-posed when surface tension effects are considered \cite{DuchonRobert1984,Chen1993,EscherSimonett1997}. Surface tension effects get rid of instabilities for short time as it introduces the highest order parabolic term in the evolution equation. However, it is possible to find scenarios where these small scales create a different type of instabilities in the well-posed solutions \cite{Otto1997} proved to exist for very short time giving exponentially growing  modes \cite{GHS2007,EscherMatioc2011}. 
	
On the other hand, if surface tension effects are not considered in the dynamics, then the Muskat problem is ill-posed due to viscosity \cite{SCH2004}, gravity \cite{CordobaGancedo2007} and viscosity-density instabilities \cite{GG-JPS2019}. For this problem this is known as the Saffman-Taylor instabilities, confirmed first experimentally and at the linear level for the equations \cite{SaffmanTaylor1958}. They appear when a more dense fluid is on top of a less dense one, if a less viscous fluid penetrate a higher viscous one, or combining both features. From the mathematical point of view, weak solutions are not unique \cite{MR3014484}. In particular, in this Rayleigh-Taylor unstable regime weak solutions with a mixing zone between the fluids have been shown \cite{CCF2016,CFM2019}. However, in the Rayleigh-Taylor stable regime, the problem is locally well-posed \cite{CordobaGancedo2007,CCG2011}. We refer to the recent results \cite{CGSV2017,CG-BS2016,Matioc2019,AlazardLazar2019,NguenPausader2019,AlazardOneFluid2019} where the regularity of the initial data is reduced to obtain local-in-time well-posedness for any scaling subcritical spaces of different type. Although maximum principles are satisfied for the case with one fluid \cite{Kim2003,AlazardOneFluid2019} and two fluids \cite{CordobaGancedo2009}, singularities may form in finite time. There are families of solutions which start in the stable regime with large slope, turn to the unstable regime \cite{CCFGL-F2012} and later produce singularities in the moving domain \cite{CCFG2013}. The required geometry of the initial data in order to produce singularities in finite time is far from trivial. Some large slope numerical solutions exhibit a regularity effect \cite{CGO2008}. Furthermore, the free boundary can enter into unstable regimes and return to a stable one \cite{CG-SZ2017}. In particular, adding fixed boundaries with slip condition, the geometry for the wave to become unstable is more complicated \cite{G-BG-S2014, BCG-B2014}. See also \cite{CCFG2016} where singularities are produced in the case of one fluid with large slope in the stable regime due to finite time particles' collisions hold. Conversely, if the slope is small, global solutions exist \cite{CCGS2013,Granero-Belinchon2014} starting from different critical spaces \cite{CCGR-PS2016} with sharp decay rates \cite{PatelStrain2017} and becoming analytic instantly \cite{GG-JPS2019}. There are recent results where large slope solutions exists globally with small conditions in critical settings \cite{Cameron2019,CordobaLazar2018} (see also \cite{granero2019growth}). We refer to \cite{granero2019growth} for a detailed review where other cases such as different viscosities, the case of regions with different permeabilities in the porous medium or the effect of impervious boundaries are considered.
		
In general, adding a fixed boundary to the system increases the difficulty to treat the problem, adding the possibility of having different singularity formation. For example, new numerical finite time singularities have been found for the Euler equations \cite{LuoHou2014}. See \cite{KRYZ2016,GancedoPatel2018} for analytic proofs of finite time singularities for related models adding fixed boundary with slip condition. 

Let us remark that the case of a thin film moving in a porous medium has been studiend also using the lubrication approximation \cite{escher2012modelling} . With this approach the free boundary problem reduces to a single degenerate parabolic equation. This PDE has been extensively studied. In particular, global solution near the steady states have been obtained \cite{escher2012modelling,escher2013existence,BGB19}. Furthermore, global weak solution have also been obtained \cite{matioc2012non,ELM11}. Remarkably, when this thin film Muskat problem is considered on $I=\mathbb{R}$, the resulting pde is a gradient flow \cite{laurenccot2013gradient,laurencot2014thin}. The intermediate asymtptotics for this degenerate parabolid pde has been also studied together with the finite speed of propagation of a certain family of weak solutions \cite{laurenccot2017finite,laurencot2017self}.

	\subsection{The Eulerian formulation}
In what follows, we consider a coordinate frame where the gravity points upward in the $x_2$ direction. Then, the physical setting is as in Figure 2
	
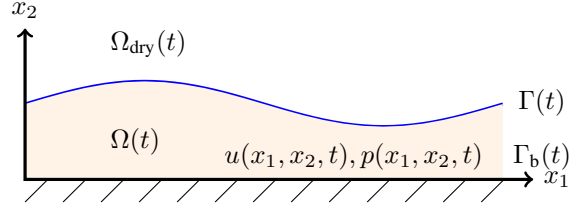
\begin{figure}[h]
	\centering
	\begin{tikzpicture}[domain=0:2*pi, scale=1] 
	\draw[color=black] plot (\x,{0.3*sin(\x r)+1}); 
	\draw[very thick, smooth, variable=\x, blue] plot (\x,{0.3*sin(\x r)+1}); 
	\fill[orange!10] plot[domain=0:2*pi] (\x,0) -- plot[domain=2*pi:0] (\x,{0.3*sin(\x r)+1});
	\draw[very thick,<->] (2*pi+0.4,0) node[right] {$x_1$} -- (0,0) -- (0,2) node[above] {$x_2$};
	\coordinate[label=above:{$\Gamma_{\text{b}}(t)$}] (A) at (6.8,0);
	\coordinate[label=above:{$\Gamma(t)$}] (A) at (6.8,0.72);
	\node[right] at (1,1.8) {$\Omega_{\text{dry}}(t)$};
	\node[right] at (1,0.5) {$\Omega(t)$};
	\node[right] at (2.5,0.3) {$u(x_1,x_2,t),p(x_1,x_2,t)$};
	\draw[-] (0,-0.3) -- (0.3, 0);
	\draw[-] (0.5,-0.3) -- +(0.3, 0.3);
	\draw[-] (1,-0.3) -- +(0.3, 0.3);
	\draw[-] (1.5,-0.3) -- +(0.3, 0.3);
	\draw[-] (2,-0.3) -- +(0.3, 0.3);
	\draw[-] (2.5,-0.3) -- +(0.3, 0.3);
	\draw[-] (3,-0.3) -- +(0.3, 0.3);
	\draw[-] (3.5,-0.3) -- +(0.3, 0.3);
	\draw[-] (4,-0.3) -- +(0.3, 0.3);
	\draw[-] (4.5,-0.3) -- +(0.3, 0.3);
	\draw[-] (5,-0.3) -- +(0.3, 0.3);
	\draw[-] (5.5,-0.3) -- +(0.3, 0.3);
	\draw[-] (6,-0.3) -- +(0.3, 0.3);
	\end{tikzpicture} 
	\caption{Physical setting after the change of coordinates}
\end{figure}

We introduce the notation for our time dependent fluid domain, its boundaries and the \textit{dry region}
\begin{align*}
	\Gamma\pare{t} & = \set{\pare{x_1, h\pare{x_1, t}}\ \left| \ x_1 \in  [-L\pi,L\pi] \right.  }, \\
	\Gamma_{\textnormal{b}} & = \set{\pare{x_1, -H}\ \left| \ x_1 \in [-L\pi,L\pi] \right.  }, \\
	\Omega\pare{t} & = \set{ \pare{x_1, x_2}\ \left|x_1\in[-L\pi,L\pi], \ x_2\in \pare{ -H , h\pare{x_1, t}}   \right.  },\\
		\Omega_{\text{dry}}\pare{t} &=\set{x = \pare{x_1, x_2} \in [-L\pi,L\pi] \times \bR \ \left| \big. \ x_2 > h\pare{x_1, t} \right. },
	\end{align*}
where $H$ is the location of the impervious ceiling and $L$ is the typical wavelength of the free boundary. We assume that the unknowns are periodic in the $x_1$ direction.

For any $ t>0 $ the one-phase Muskat problem with upward pointing gravity consists on the equations
	\begin{equation}
	\label{eq:Muskat}
	\left\lbrace
	\begin{aligned}
	& \frac{\mu}{\kappa} u = -\nabla p +G \rho \  {\bf e}_2, && \text{ in }\Omega\pare{t} \\
	& \div \ u  = 0,  && \text{ in }\Omega\pare{t} \\
	& \partial_t h = u\cdot {\bf n},&& \text{ on }\Gamma\pare{t} \\
	& p = -\gamma \ccK ,&& \text{ on }\Gamma\pare{t} \\
	& u\cdot {\bf e}_2 =0 , && \text{ on }\Gamma_{\textnormal{b}}, 
	\end{aligned}
	\right. 
	\end{equation}
	where $0<\mu$, $0<\rho$, $p(x_1,x_2,t)$ and $u(x_1,x_2,t)$ are the viscosity, density, pressure and velocity of the incompressible fluid under consideration. Similarly,  
	\begin{align*}
	{\bf n} = \pare{\begin{array}{c}
		- \partial_1 h \\ 1
		\end{array}}, && \ccK= \frac{\partial_1^2 h}{\pare{1+\pare{\partial_1 h}^2}^{3/2}}, 
	\end{align*}
are the outward pointing normal and the curvature of the interface.
	
We want to highlight that the equations in \eqref{eq:Muskat} describe the motion of a fluid in a porous medium underlying an impermeable ceiling with flat bottom topography. Although in the present work we restrict ourselves to flat topographies, non-flat bottom topographies whose topography variation is small with respect to the fluid layer depth and such that there is no intersection between the bottom and the moving interface can be studied using the techniques that we are going to outline in the present manuscript.

Denoting 
$$ 
\Phi=-\frac{\kappa}{\mu}\left(p-G\rho y\right) 
$$ 
we find that $u$ is a potential flow
$$
u=\nabla\Phi
$$
and \eqref{eq:Muskat} becomes

	\begin{equation}
	\label{eq:Muskat2}
	\left\lbrace
	\begin{aligned}
	& \Delta \Phi = 0,  && \text{ in }\Omega\pare{t} \\
	& \partial_t h = -\partial_1\Phi \ \partial_1 h + \partial_2 \Phi,&& \text{ on }\Gamma\pare{t} \\
	& \Phi = \frac{\kappa\gamma}{\mu} \ccK+ \frac{G \rho \kappa}{\mu} h ,&& \text{ on }\Gamma\pare{t} \\
	& \partial_2\Phi =0 , && \text{ on }\Gamma_{\textnormal{b}}, 
	\end{aligned}
	\right. 
	\end{equation}
	
	We want to study \eqref{eq:Muskat2} in the particular regime
	\begin{align*}
	\gamma > 0, && L > 0 , && 0< H \ll 1. 
	\end{align*}

	\subsection{Nondimensionalization in Eulerian coordinates}

	Let us introduce the following physical parameters (cf. \cite[Section 1.3]{Lannes2013}):
	\begin{enumerate}
		\item The characteristic fluid depth $ H $, 
		
		\item The characteristic horizontal scale $ L $, 
		
		\item The characteristic free surface amplitude $ a $. 
		
	\end{enumerate}

	Our aim is to study a configuration in which there is a thin layer of fluid, whose interface does not touch the bottom. Such configuration, translated in mathematical terms, imposes the following constraints
	\begin{align}\label{eq:relations_physical parameters}
	a\leqslant H, && H \leqslant L. 
	\end{align}

	Let us now define the dimensionless parameters
	\begin{align} \label{eq:dimensionless_parameters}
	\varepsilon = \frac{a}{H}, && \delta = \frac{H^2}{L^2}, && \nu = \frac{\gamma}{HL\rho \ G}, && \alpha = \frac{a}{L}, 
	\end{align}
	generally $ \varepsilon $ is known as \textit{nonlinearity parameter}, $ \delta $ is the \textit{shallowness parameter}, we refer the reader to \cite[Section 1.3.1]{Lannes2013}. The dimensionless parameter $ \nu $ is the \textit{Bond number}, which quantifies the ratio between capillarity and gravitational forces. Finally, $\alpha$ is the \textit{steepness parameter}.

	We can immediately draw some conclusion from \eqref{eq:relations_physical parameters} and \eqref{eq:dimensionless_parameters}, such as
	\begin{align}\label{eq:relations_dimensionless parameters}
	\delta \leqslant 1, && \varepsilon \leqslant 1, && \alpha \leqslant \varepsilon, && \alpha \leqslant \sqrt{\delta}. 
	\end{align}

	We now nondimensionalize the equations of \eqref{eq:Muskat2} defining suitable dimensionless quantities
	\begin{align*}
	x_1 & = L \tilde{x}_1, & x_2 & = H \tilde{ x_2}, & t & = \frac{\mu L}{\rho \kappa G} \ \tilde{t}, \\
	\Phi \pare{x, t} & = \frac{H\kappa\rho G}{\mu} \tilde{\Phi}\pare{ \tilde{x}, \tilde{t}}, & h \pare{x, t} & = a \tilde{h}\pare{ \tilde{x}, \tilde{t}}. 	 
	\end{align*}
	
	With such nondimensionalization \eqref{eq:Muskat2} becomes (here we drop the tilde notation)

	\begin{equation}
	\label{eq:Muskat3}
	\left\lbrace
	\begin{aligned}
	& \pare{ \partial_1 ^2 + \frac{1}{\delta} \partial_2^2} \Phi = 0,  && \text{ in }\Omega\pare{t} \\
	&  \partial_t h =  - \sqrt{\delta} \ \partial_1\Phi \ \partial_1 h + \frac{1}{\alpha}\partial_2 \Phi,&& \text{ on }\Gamma\pare{t} \\
	& \Phi = \nu \alpha \ccK_{\alpha}+ \varepsilon h  ,&& \text{ on }\Gamma\pare{t} \\
	& \partial_2\Phi =0 , && \text{ on }\Gamma_{\textnormal{b}}, 
	\end{aligned}
	\right. 
	\end{equation}
	where the modified mean-curvature function $ \ccK_{\alpha} $ is
	\begin{equation}\label{eq:Kalpha}
	\ccK_{\alpha} = \frac{\partial_1^2 h}{\pare{1+  \pare{ \alpha \partial_1 h}^2}^{3/2}}. 
	\end{equation}
	and the respective nondimensional domain is
	\begin{align*}
	\Gamma\pare{t} & = \set{\pare{x_1, \varepsilon h\pare{x_1, t}}\ \left| \ x_1 \in  \bT \right.  }, \\
	\Gamma_{\textnormal{b}} & = \set{\pare{x_1, - 1 }\ \left| \ x_1 \in  \bT \right.  }, \\
	\Omega\pare{t} & = \set{ \pare{x_1, x_2}\ \left| \ x_2\in \pare{ - 1 , \varepsilon h\pare{x_1, t}}   \right.  } . 
	\end{align*}

		\subsection{Methodology}
		
	Let us briefly explain the methodology that will be developed along the lines of the present manuscript. At first, as it is well known \cite{CordobaGancedo2007, CG-BS2016}, we exploit the irrotationality of the velocity flow $ u $ in order to express the Muskat problem in terms of the velocity potential $ \Phi $ and elevation $ h $ (see \eqref{eq:Muskat2}). We consider the dimensionless system \eqref{eq:Muskat3} in order to highlight the regularizing properties induced by the surface tension effects together with the role played by the different physical parameters. We observe that in this new (equivalent) system, the capillary forces are proportional to $ H^{-2} $, while the gravity forces are proportional to $ H^{-1} $, being $ H $ the average depth of the fluid layer in the horizontally periodic cell. In such setting if $ H $ is sufficiently small we expect the capillarity forces to dominate the gravity ones \textit{independently of the direction toward which the gravity points}, thus, considering the most unfavorable case, stabilizing a (thin) fluid layer which lies below a periodic ceiling. 
	
	On a mathematical point of view such a goal is obtained by expressing explicitly the capillarity/gravity forces respectively as forward/backward parabolic linear contributions in the evolution equation for the fluid interface $ h $. After the use of a \textit{regularizing diffeomorphism} (see \cite[Definition 2.17, p. 45]{Lannes2013}) which fixes the domain $ \Omega\pare{t} $ into $ \cS = \bT\times\pare{-1, 0} $, we recast the (nondimensional) Muskat problem as an evolution equation for the free elevation $ h $ in which the nonlinear terms are composed by traces of functions which satisfy certain elliptic equation with non-constant coefficients in the bulk of the fluid $ \cS $. Such methodology is also known in the literature as Arbitrary Lagrangian Eulerian (ALE) formulation (see \cite{CG-BS2016, granero2019well, GS18, GS18_2, GS19, GS_DDZ_2}). The elliptic problem is hence endowed with some mixed Dirichlet-Neumann boundary conditions in which the Dirichlet part is a nonlinear relation involving only the elevation $ h $ and its derivatives. At this stage the main problem is to provide suitable bounds in terms of $ h $ only, for the aforementioned traces of functions defined in the bulk $ \cS $. These bounds are proved studying small perturbations of the explicit solution $ \varphi $ of the elliptic problem. These elliptic estimates combined with the parabolic estimates for the evolution of the interface $ h $ allows us to prove that regular solutions stemming from small initial data decay and become instantaneously analytic. Once the uniform global bounds are proved the construction of weak global solutions follows a standard approximation argument.

	\subsection{Notation}\label{sec:notationa}
	
	Given $ M \in \bR^{n\times m} $ we denote with $ M_j^i $ the entry of $ M $ at row $ i $ and column $ j $, and we use Einstein convention for repeated indexes all along the manuscript. 
	
	Let us define, accordingly to the notation of \cite{Lannes2013}
	\begin{equation*}
	I_\delta = \pare{
		\begin{array}{cc}
		\sqrt{\delta} & 0 \\
		0 & 1
		\end{array}
	}. 
	\end{equation*}
	Once we have defined the matrix $ I_\delta $ we can define the following anisotropic differential operators:
	\begin{align*}
	\nd v & = I_\delta \nabla v=\pare{\sqrt{\delta}\partial_1 v , \partial_2 v}^\intercal, & \div_\delta V & = \pare{\nd}^\intercal \cdot V= \sqrt{\delta}\partial_1 V^1 + \partial_2 V^2 \ , & \Delta_\delta v & = \div_\delta \nd v = \delta \partial_1^2 v + \partial_2^2 v. 
	\end{align*}
	Given a unitary vector $ V\in \bS^1\subset \bR^2 $ and $ F\in \cC^1 \pare{\bR^2} $ we denote the directional derivate as $ \partial_V F = \nabla F \cdot V $. 
	
%

	All along the manuscript $ \bT = \bR / 2\pi\bZ $, which alternatively can be thought as the interval $ \bra{-\pi, \pi} $ endowed with periodic boundary conditions. For any function $ v\in L^1\pare{\bT} $ and $ n\in\bZ $ we recall that
	\begin{equation*}
	\hat{v}\pare{n} = \frac{1}{2\pi}\int _{\bT} v\pare{x} e^{-in\ x}\dx, 
	\end{equation*}
	denotes the expression of the $ n $-th Fourier coefficient of $ v $. If $ v\in L^2\pare{\bT} $ we can provide the Fourier series representation of $ v $:
	\begin{equation*}
	v\pare{x} =\sum_{n\in\bZ} \hat{v}\pare{n} \ e^{in\ x}. 
	\end{equation*}

	\subsection{Functional spaces}\label{sec:notation}	
	Let $ X = X\pare{\bT} $ be any functional space defined on $ \bT $, then we use the notation 
	\begin{equation*}
	\av{v}_{X} = \norm{v}_{X\pare{\bT}}. 
	\end{equation*}
	Let us define the $x_1$-periodic two-dimensional open strip $ \cS = \bT \times\pare{ -1, 0} $. Let $ Y=Y\pare{\cS} $ be any functional space defined of $ \cS $, then we use the notation
	\begin{equation*}
	\norm{v}_{Y} = \norm{v}_{Y\pare{\cS}}. 
	\end{equation*}

	We denote with $ \ccS = \ccS\pare{\bT} $ the space of Schwartz functions on the torus, while $ \ccS_0 = \ccS_0\pare{\bT} $ is the subspace of functions of $ \ccS $ with zero average. Respectively $ \ccS' $ and $ \ccS_0' $ are the dual spaces of $ \ccS $ and $ \ccS_0 $.

We define the set
	\begin{equation*}
	D\pare{f \pare{\Lambda}} = \set{v\in L^1 \cap \ccS' \ \left| \ \pare{ \big.  f \pare{\av{n}}\hat{v}\pare{n}}_{n\in \bZ} \in \ell^2 \pare{\bZ} \right. }. 
	\end{equation*}
Then, for any $ f\in \cC^\infty\pare{\bR\setminus \set{0}}  $  and  $ v\in D\pare{f\pare{\Lambda}} $, we define the operator $ f\pare{\Lambda} $ as 
	\begin{equation*}
	\widehat{f\pare{\Lambda} v}\pare{n} = f\pare{\av{n}} \hat{v}\pare{n}, 
	\end{equation*}
	
	
		\subsubsection{Wiener spaces on $\bT$}	
	Let us define the \textit{Wiener space} $ \mathbb{A}^s_\lambda \pare{\bT}, \ s, \lambda \geqslant 0 $ as the closure of $ \ccS \pare{\bT} $ w.r.t. the norm
	\begin{equation*}
	\av{v}_{\mathbb{A}^s_\lambda \pare{\bT}}=\av{v}_{\mathbb{A}^s_\lambda} = \av{v}_{s, \lambda}  = \sum _{n\in\bZ} \pare{ 1 + \av{n} }^s e^{\lambda\av{n}} \av{\hat{v}\pare{n}}. 
	\end{equation*}
We observe that if $\lambda>0$, then this space contains analytic functions. If $ \lambda = \lambda\pare{t}, \ t\in\bra{0, T} $ and $ p\in\bra{1, \infty}, \ T\in\left( 0, \infty\right], \ s\geqslant 0 $ we will adopt the following notation
	\begin{equation*}
	L^p\pare{\bra{0, T}; \mathbb{A}^s_{\lambda\pare{t}} } = L^p\pare{\bra{0, T}; \mathbb{A}^s_{\lambda\pare{\cdot}} }. 
	\end{equation*}
	Let us remark that if $ v $ is of zero average (i.e. $ \hat{v}\pare{0}=0 $) then 
	\begin{equation}\label{eq:hom-nonhom}
	\frac{\av{v}_{s, \lambda}}{2} \leqslant \sum_n \av{n}^s e^{\lambda \av{n}} \av{\hat{v}\pare{n}} \leqslant \av{v}_{s,\lambda}. 
	\end{equation}
	We will use the notation $	\mathbb{A}^s = \mathbb{A}^s_0, \av{v}_s  =\av{v}_{s, 0}.$
	From its definition it is immediate to conclude that
	\begin{align*}
	\mathbb{A}^{s_2}_{\lambda}& \hra \mathbb{A}^{s_1}_\lambda, & s_1 & \leqslant s_2, \\
	\mathbb{A}^{s_2}_{\lambda_2}& \hra \mathbb{A}^{s_1}_{\lambda_1} & \lambda_1 & < \lambda_2, \quad s_1, s_2 \geqslant 0. 
	\end{align*}

	Moreover the Wiener spaces $ \mathbb{A}^s_\lambda, s, \lambda \geqslant 0 $ satisfy the following properties:
	\begin{lemma}\label{lem:interpolation_inequality}
		Let $ 0\leqslant s_1 \leqslant s_2 $, $ \lambda \geqslant 0 $ and let $ v\in \mathbb{A}^{s_2}_\lambda $. Let us define, for any $ \theta \in \bra{0, 1} $, $ s_{\theta} = \theta s_1 + \pare{1-\theta}s_2 $, then
		\begin{equation}\label{eq:interpolation_inequality}
		\av{v}_{s_\theta, \lambda} \leqslant \av{v}_{s_1, \lambda}^{\theta} \av{v}_{s_2, \lambda}^{1-\theta}.
		\end{equation}
		Moreover, for any $ 0\leqslant s_1 < s_2 $ we have
		\begin{equation*}
		\mathbb{A}^{s_2}_{\lambda} \Subset \mathbb{A}^{s_1}_{\lambda}. 
		\end{equation*}
		Let $ f, g \in \mathbb{A}^s_\lambda, \ s, \lambda \geqslant 0 $, then
		\begin{equation}\label{eq:product_rule_Wiener}
		\begin{aligned}
		\av{fg}_{s, \lambda} & \leqslant \cK_s \pare{ \Big. \av{f}_{0, \lambda} \av{g}_{s, \lambda} + \av{f}_{s, \lambda} \av{g}_{0, \lambda}}\leq 2^{s+1} \av{f}_{s, \lambda} \av{g}_{s, \lambda} , 
		\end{aligned}
		\end{equation}
		where 
		$$
		\cK_s = \left\lbrace
	\begin{aligned}
	& 1 && \text{ if } s\in\bra{0, 1}, \\
	& 2^s && \text{ if } s > 1.
	\end{aligned}
	\right. .
		$$ 
	\end{lemma}
	\begin{proof}
		For a proof of \eqref{eq:interpolation_inequality} and \eqref{eq:product_rule_Wiener} we refer the reader to \cite[Lemma 2.1]{BGB19}. The proof of \eqref{eq:product_rule_Wiener} is based on the following elementary inequality
		\begin{align*}
		\pare{1+\av{n}}^s e^{\lambda \av{n}} & \leqslant 2^s \max \set{\pare{1+\av{n-k}}^s, \pare{1+\av{k}}^s } e^{\lambda \pare{ \av{n-k} + \av{k} }}, \\
		& \leqslant 2^s \bra{\pare{1+\av{n-k}}^s + \pare{1+\av{k}}^s}e^{\lambda\av{n-k}} e^{\lambda\av{k}}, 
		\end{align*}
		from which \eqref{eq:product_rule_Wiener} is an immediate consequence. The fact that we can choose $ \cK_s=1 $ for $ s\in\bra{0, 1} $ is a consequence of the concavity and monotonicity of the function $ r\mapsto \pare{1+r}^s $ for $ r\in\bR_+ $ and $ s\in\bra{0, 1}  $. 
		We prove now that the continuous inclusion $ \mathbb{A}^{s_2}_{\lambda}\hra \mathbb{A}^{s_1}_{\lambda} $ is compact. Let us consider a sequence $ \pare{v_n}_n \subset B_{\mathbb{A}^{s_2}_{\lambda}}\pare{0, 1} $, this means that
		\begin{align*}
		\sum_k \pare{ 1+ \av{k} }^{s_2} e^{\lambda\av{k}} \av{\hat{v}_n \pare{k}}  \leqslant 1, && \forall \ n. 
		\end{align*}
		We use the following compactness criterion for $ \ell^p\pare{\bZ^d}, \ p\in [1, \infty) $ whose proof is an application of Kolmogorov-Riesz compactness criterion in $ L^p\pare{\bR^d}, \ p\in [1, \infty ) $ applied to stepwise constant functions:\\
		
		{\it Let $ B $ be a bounded set of  $ \ell^p\pare{\bZ^d}, \ p\in [1, \infty) $, then $ B $ is compact if and only if
			\begin{equation*}
			\lim _{N\to \infty} \sum_{\substack{j\in\bZ^d \\ \av{j}\geqslant N}} \av{b_j}^p =0, 
			\end{equation*}
			uniformly for $ b\in B $}. \\ 
		
		Thus we have
		\begin{align*}
		\sum_{\av{k}\geqslant N} \pare{ 1+ \av{k} }^{s_1} e^{\lambda\av{k}} \av{\hat{v}\pare{k}}  = & \  \sum_{\av{k}\geqslant N}\pare{ 1+  \av{k} }^{s_1-s_2} \ \pare{ 1+ \av{k} }^{s_2} e^{\lambda\av{k}} \av{\hat{v}\pare{k}}, \\
		\leqslant & \ \pare{ 1+ N }^{s_1-s_2} \sum_{\av{k}\geqslant N}  \pare{ 1+ \av{k} }^{s_2} e^{\lambda\av{k}} \av{\hat{v}\pare{k}}, \\
		\leqslant & \ \pare{ 1+  N }^{s_1-s_2}  \xrightarrow{N\to \infty}  0,
		\end{align*}
		since $ s_2 > s_1 \geqslant 0 $. 
	\end{proof}

	As we will see the quantification of the constant $ \cK_s$ characterizing the product rule of Lemma \ref{lem:interpolation_inequality} when $ s\in\bra{0, 1} $ is very important for the rest of the manuscript. We need as well a suitable product rule for an integer power of an element in $ \mathbb{A}^s_{\lambda}, \ s, \lambda \geqslant 0 $. Such result can be deduced iterating Lemma \ref{lem:interpolation_inequality}:
	\begin{cor}
		\label{cor:product_rule_Wiener}
		For any $ s \geqslant 0, \ n\in \bN, \ n\geqslant 2 $ let us define
		\begin{equation}
		\label{eq:Ksn}
		K_{s, n} = \left\lbrace
		\begin{aligned}
		& n & s & \in \bra{0, 1},  \\
		&  \frac{\cK_s \pare{ \cK_s^{n-1} -1}}{\cK_s - 1} & s & > 1,
		\end{aligned}
		\right. 
		\end{equation}
		where $ \cK_s$ is defined in Lemma \ref{lem:interpolation_inequality}, then for any $ v\in \mathbb{A}^s_\lambda $
		\begin{equation*}
		\av{v^n}_{s, \lambda} \leqslant K_{s, n} \av{v}_{0, \lambda}^{n-1}\av{v}_{s, \lambda}. 
		\end{equation*}
	\end{cor}
	

	%
	%
	%
	%
	The next lemma describes the action of a smooth function vanishing at zero applied to functions in the Wiener space. The statement of the lemma is restricted to the setting that will be used in the present work, but generalizations can be derived in several ways (see for instance \cite[Theorem 2.61, p. 94]{BCD} for a Besov setting counterpart).  
	
	\begin{lemma}
		\label{lem:composition_Wiener}
		Let $ f $ be a real analytic function around zero with radius of convergence $ R_f \in \left( 0, \infty \right] $ and such that $ f\pare{0}=0 $. Suppose that $ v\in \mathbb{A}^s_\lambda, \ s, \lambda \geqslant 0 $, let $ K_{s, n} $ be as Corollary \ref{cor:product_rule_Wiener} and let us suppose moreover that
		\begin{align}\label{eq:bound_u_radius_analitycity}
		\bra{ \limsup_{n} \sqrt[n-1]{\frac{\av{f^{\pare{n}}\pare{0}}}{n!}} \  K_{s, n}^{\frac{1}{ n - 1 }}  } \  \av{v}_{0, \lambda}  = q_f <1, 
		\end{align}
		then
		\begin{equation*}
		\av{f\circ v}_{s, \lambda} \leqslant \frac{1}{1-q_f} \ \av{v}_{s, \lambda}. 
		\end{equation*}
	\end{lemma}
	
	\begin{proof}
		Since $ f $ is real analytic and vanishing in zero, we know that for each $ z\in\bR, \ \av{z} < R_f $
		\begin{equation*}
		f\pare{z} = \sum_{n\geqslant 1} \frac{f^{\pare{n}}\pare{0}}{n!} \ z^n, 
		\end{equation*}
		whence, at least formally
		\begin{equation}\label{eq:fu_as_series}
		f\pare{v\pare{x}}= \sum_{n\geqslant 1} \frac{f^{\pare{n}}\pare{0}}{n!} \ v\pare{x}^n. 
		\end{equation}
Thus,
		\begin{equation*}
		\av{ f\pare{v\pare{x}}} \leqslant \sum_{n\geqslant 1} \frac{\av{f^{\pare{n}}\pare{0}}}{n!} \ \av{v\pare{x}}^n. 
		\end{equation*}
Since $ s\geqslant 0 $ we have that
		\begin{equation*}
		\av{v\pare{x}}\leqslant \av{v}_{L^\infty} \leqslant \av{v}_{0, \lambda} \leqslant \av{v}_{s, \lambda}, 
		\end{equation*}
		which in turn implies, using the product rule given in Lemma \ref{lem:interpolation_inequality},  that
		\begin{equation*}
		\av{ f\pare{v\pare{x}}} \leqslant \av{v}_{s, \lambda} \sum_{n\geqslant 1} \frac{\av{f^{\pare{n}}\pare{0}}}{n!} \ K_{s, n} \av{v}^{n-1}_{0, \lambda}. 
		\end{equation*}
		We use now \eqref{eq:bound_u_radius_analitycity} in order to deduce that
		\begin{equation*}
		\av{f\pare{v\pare{x}}}\leqslant \frac{1}{1-q_f} \ \av{v}_{s, \lambda}. 
		\end{equation*}
		In a rather similar fashion we deduce from \eqref{eq:fu_as_series} that
		\begin{equation*}
		\av{f\circ v}_{s, \lambda}  \leqslant \av{v}_{s, \lambda} \sum_{n\geqslant 1} \frac{\av{f^{\pare{n}}\pare{0}}}{n!} \ K_{s, n} \av{v}^{n-1}_{0, \lambda} \leqslant \frac{1}{1-q_f} \ \av{v}_{s, \lambda} .
		\end{equation*}
	\end{proof}

	Despite the fact that Lemma \ref{lem:composition_Wiener} holds true for any $ f $ analytic in zero we will apply such result for a couple of specific functions which are analytic and vanishing at zero. To this end let us define
	\begin{align}\label{eq:def_F}
	\mathsf{G}\pare{x} = \frac{x}{1+x},  &&
	\mathsf{F} \pare{x} = \frac{1}{\pare{1+x^2}^{3/2}} - 1, 
	\end{align}
	and 
	\begin{equation}\label{eq:Js}
	\cK_s = \lim_{n\rightarrow\infty} K_{s, n}^{\frac{1}{n-1}} = \left\lbrace
	\begin{aligned}
	& 1 && \text{ if } s\in\bra{0, 1}, \\
	& 2^s && \text{ if } s > 1.
	\end{aligned}
	\right. 
	\end{equation}Applying Lemma \ref{lem:composition_Wiener} to the functions defined in \eqref{eq:def_F} we obtain the following result
	
	\begin{lemma} \label{lem:composition_Wiener_FG}
		Let $ \mathsf{F, G} $ be defined as in \eqref{eq:def_F}, let $ v \in \mathbb{A}^s_{\lambda}, \ s, \lambda \geqslant 0 $ be such that $ \av{v}_{0, \lambda} < \min \set{ 1 \ , \ \frac{1}{\cK_s}} $, then
		\begin{align*}
		\av{\mathsf{F} \circ v}_{s, \lambda} \leqslant &  \av{v}_{s, \lambda}, \\
		\av{\mathsf{G} \circ v}_{s, \lambda} \leqslant &  \frac{1}{1-\cK_s \av{v}_{0, \lambda} } \ \av{v}_{s, \lambda} .
		\end{align*}
	\end{lemma}
%
%

	\subsubsection{Wiener-Sobolev spaces in $ \cS $}
	
	We need to define functional spaces in the horizontally periodic open strip $ \cS $. Let $ s, \lambda \geqslant 0, \ p\in\bra{1, \infty} $ and $ k\in\bN $
	\begin{enumerate}
		\item We write $ L^p_2 \mathbb{A}^{s, \lambda}_1 = L^p\pare{\pare{-1, 0}; \mathbb{A}^s _\lambda \pare{\bT}} $ endowed with the canonical norm
		\begin{equation*}
		\begin{aligned}
		\norm{v}_{L^p_2 \mathbb{A}^{s, \lambda}_1} = & \pare{\int_{-1}^0 \av{v\pare{\cdot, x_2}}_{s, \lambda}^p \dx_2}^{1/p}, \qquad p\in \left[1, \infty \right), \\
		\norm{v}_{L^\infty_2 \mathbb{A}^{s, \lambda}_1} = & \esssup _{x_2\in\pare{-1, 0}} \av{v\pare{\cdot, x_2}}_{s, \lambda} , 
		\end{aligned}
		\end{equation*}
		
		\item We define the Wiener-Sobolev space $ \cA^{s, k}_\lambda = \cA^{s, k}_\lambda\pare{\cS} $ as the closure of $ \cD\pare{\bT \times \left( -1 , 0 \right]} $ w.r.t. the norm
		\begin{equation}\label{eq:norms_cA}
		\norm{v}_{\cA^{s, k}_\lambda} = \sum_n \pare{ 1+ \av{n} }^s e^{\lambda\av{n}} \int_{-1}^0 \av{\partial_2^k \hat{v}\pare{n, x_2}} \dx_2 .
		\end{equation}
These spaces are Banach spaces. We denote $	\cA^{s, k} = \cA^{s, k}_0 .$
			\end{enumerate}

	The following lemma define some basic properties of the Wiener-Sobolev spaces $ \cA^{s, k}_\lambda $:
	
	\begin{lemma}\label{prop:trace-embedding}
		\begin{enumerate}
			
			\item {\bf Trace estimate:} The map $ v\mapsto v_{x_2 =0} $ is continuous from $ \cA^{s, 1}_\lambda $ onto $ \mathbb{A}^s_\lambda  $ and the following estimate holds
			\begin{equation*}
			\av{\left. v\right|_{x_2=0}}_{s, \lambda} \leqslant  \norm{v}_{\cA^{s, 1}_\lambda}, 
			\end{equation*}
			
			\item {\bf Embedding in $ L^\infty_2 \mathbb{A}^{s, \lambda}_1 $:} The space $ \cA^{s, 1}_\lambda $ is  continuously embedded in $ L^\infty_2 \mathbb{A}^{s, \lambda}_1 $, i.e.
			\begin{equation*}
			\norm{v}_{L^\infty_2 \mathbb{A}^{s, \lambda}_1} \leqslant  \norm{v}_{\cA^{s, 1}_\lambda}. 
			\end{equation*}
		\end{enumerate}
		
	\end{lemma}
	\begin{proof}
		The proof of the above lemma is very similar to the proof of  \cite[Proposition 2.12, p. 42]{Lannes2013}. We prove the point 1, we let $ v\in \cD\pare{\bT\times(-1, 0]} $ hence
		\begin{equation}\label{eq:TRACE}
		\av{\hat{v}\pare{n, 0}} \leqslant \int_{-1}^{0 } \av{\partial_{y_2}\hat{v}\pare{n, y_2}} \dy_2, 
		\end{equation} 
		the claim of the point 1 follows. We prove now the point 2. 
		We let $ v\in \cD\pare{\bT\times(-1, 0]} $, it is true that
		\begin{equation*}
		\av{\hat{v}\pare{n, x_2}} \leqslant \int_{-1}^{x_2} \av{\partial_{y_2}\hat{v}\pare{n, y_2}} \dy_2. 
		\end{equation*}
As a consequence,
		\begin{align}\label{eq:embedding_ineq}
		\av{v\pare{\cdot, x_2}}_s \leqslant \norm{v}_{\cA^{s, 1}_\lambda}, &&
		\sum_n \pare{ 1+ \av{n} }^s e^{\lambda\av{n}} \sup _{x_2\in\pare{-1, 0}} \av{\hat{v}\pare{n, x_2 }} \leqslant \norm{v}_{\cA^{s, 1}_\lambda}. 
		\end{align}
		and since the right hand side of the above inequality is independent of $ x_2\in \pare{-1, 0} $ we obtain the result for test functions. A density argument lifts the result to the functional setting desired. 
	\end{proof}

	%
	%

	\begin{lemma}\label{lem:interpolation_inequality_strip}
		If $ f, g \in \cA^{s, 1}_\lambda, \ s > 0 $ the following product rules hold
		\begin{equation}
		\label{eq:product_rule_Wiener_strip}
		\norm{fg}_{\cA^{s, 1}_\lambda} \leqslant 2 \cK_s \pare{\norm{f}_{\cA^{s, 1}_\lambda}\norm{g}_{\cA^{0, 1}_\lambda} + \norm{f}_{\cA^{0, 1}_\lambda}\norm{g}_{\cA^{s, 1}_\lambda}},  
		\end{equation}
		and
		\begin{equation*}
		\norm{f^n}_{\cA^{s, 1}_\lambda}\leqslant 2^n K_{s, n} \norm{f}^{n-1}_{\cA^{0, 1}_\lambda} \norm{f}_{\cA^{s, 1}_\lambda}. 
		\end{equation*}
		If $ s=0 $
		\begin{equation*}
		\norm{fg}_{\cA^{0, 1}_\lambda} \leqslant 2 \norm{f}_{\cA^{0, 1}_\lambda}\norm{g}_{\cA^{0, 1}_\lambda} .  
		\end{equation*}
	\end{lemma}
	
	\begin{proof}
		The proof is straightforward, we have
		\begin{align*}
		\norm{fg}_{\cA^{s, 1}_\lambda} & \leqslant \norm{\partial_2 f \ g}_{\cA^{s, 0}_\lambda} +  \norm{ f \ \partial_2 g}_{\cA^{s, 0}_\lambda} .
		\end{align*}
Moreover using the triangular inequality and Fatou's lemma, we deduce 
		\begin{equation*}
		\begin{aligned}
		\norm{\partial_2 f \ g}_{\cA^{s, 0}_\lambda} = & \ \sum_n \pare{1+\av{n}}^s e^{\lambda\av{n}} \int _{-1}^0 \av{\sum_m \widehat{\partial_2 f}\pare{n-m, x_2} \ \hat{g}\pare{m, x_2}} \dx_2, \\
		\leqslant & \ \sum_{n, m} \pare{ 1+ \av{n} }^s e^{\lambda\av{n}}  \int _{-1}^{0}\av{\widehat{\partial_2 f}\pare{n-m, x_2}}  \av{\hat{g}\pare{m , x_2}} \dx_2.
		\end{aligned}
		\end{equation*}
We use now H\"older inequality, Young convolution inequality, the second inequality of \eqref{eq:embedding_ineq} and the product rules for $ \mathbb{A}^s $ proved in Lemma \ref{lem:interpolation_inequality} in order to obtain the following chain of inequalities
		\begin{align*}
		\norm{\partial_2 f \ g}_{\cA^{s, 0}_\lambda} & \leqslant \sum_{n, m} \pare{ 1+ \av{n} }^s e^{\lambda\av{n}}  \int _{-1}^{0}\av{\widehat{\partial_2 f}\pare{n-m, x_2}}  \av{\hat{g}\pare{m , x_2}} \dx_2, \\
		& \leqslant \sum_{n, m} \pare{ 1+ \av{n} }^s e^{\lambda\av{n-m}} \norm{\widehat{\partial_2 f}\pare{n-m, \cdot}}_{L^1_2}   e^{\lambda\av{m}} \norm{\hat{g}\pare{m, \cdot}}_{L^\infty_2} , \\
		& \leqslant \cK_s \pare{  \norm{f}_{\cA^{s, 1}_\lambda} \norm{g}_{\cA^{0, 1}_\lambda} + \norm{f}_{\cA^{0, 1}_\lambda} \norm{g}_{\cA^{s, 1}_\lambda} }. 
		\end{align*}
		A similar estimate can be performed for $ \norm{f\ \partial_2 g}_{\cA^{s, 0}_\lambda} $ and this concludes the proof. 
	\end{proof}

	\subsection{Results}

	Before stating the main result of the present manuscript we need to provide a definition of weak solution for the system \eqref{eq:Muskat3}. The definitions that we present here are adapted from \cite[Chapter 2]{Lannes2013}. Let $ T> 0 $ and let us consider an $ h\in L^\infty\pare{\bra{0, T}; W^{1, \infty}\pare{\bT}} $ s.t. there exists a $ h_{\min} \in \pare{0, 1} $ such that
	\begin{align}\label{eq:no_pintch-off}
	1+\varepsilon h\pare{x_1, t} \geqslant h_{\min} > 0, && \forall \ \pare{x_1, t}\in \bT\times \bra{0, T}.
	\end{align}
	Let us define 
	\begin{equation*}
	\Omega_{\textnormal{b}}\pare{t} = \set{\pare{x_1, x_2}\in\bT\times\bR \ \left| \ -1\leqslant x_2 < \varepsilon h\pare{x_1, t} \Big. \right. }. 
	\end{equation*}
	Next, for any $ s > 1/2 $ we define the space
	\begin{equation*}
	H^s_{0, \textnormal{surf}}\pare{\Omega\pare{t}} = \overline{\cD\pare{\Omega_{\textnormal{b}}\pare{t}}}^{H^s\pare{\Omega \pare{t}}}. 
	\end{equation*}
	Let us  consider the elliptic problem
	\begin{equation}
	\label{eq:elliptic}
	\left\lbrace
	\begin{aligned}
	& \Dd \Xi \pare{t} =0, && \text{ in } \Omega\pare{t}, \\
	& \Xi \pare{t} = \xi \pare{t} , && \text{ on }\Gamma\pare{t}, \\
	& \partial_2 \Xi \pare{t} =0, &&  \text{ on }\Gamma_{\textnormal{b}}.
	\end{aligned}
	\right.
	\end{equation}
Abusing notation we denote
$$
\xi(x_1,x_2,t)=\xi(x_1,t).
$$
\begin{definition}
		\label{def:variational_solution}
		Given a fixed $ t\in \bra{0, T} $ such that $ \xi\pare{t}\in H^1\pare{\bT} $ and $\Omega(t)$ is a $C^2$ domain, we say that $ \Xi\pare{t}\in H^1(\Omega(t)) $ is the \textit{variational solution} of \eqref{eq:elliptic} if there exists a unique $ \widetilde{\Xi}\in H^1_{0, \textnormal{surf}}\pare{\Omega\pare{t}} $ such that  $ \Xi = \xi + \widetilde{\Xi} $ and 
		\begin{equation*}
		\int _{\Omega\pare{t}} \nd \widetilde{\Xi}\pare{t} \cdot \nd \Theta\pare{t} \dx = \int _{\Omega\pare{t}} \nd \xi \pare{t} \cdot \nd \Theta\pare{t} \dx, 
		\end{equation*}
		for any $ \Theta\pare{t} \in  H^1_{0, \textnormal{surf}}\pare{\Omega\pare{t}} $. 
	\end{definition}

	We can now provide a proper notion of weak solution for \eqref{eq:Muskat3}:
	\begin{definition}
		\label{def:weak_solution}
		We say that $ h\in L^\infty\pare{\bra{0, T};  W^{1, \infty} \pare{\bT}}  \cap  L^1\pare{\bra{0, T};  W^{3, \infty} \pare{\bT}}   $ is a \textit{weak solution} of \eqref{eq:Muskat3} if \eqref{eq:no_pintch-off} is satisfied and for any $ \theta \in \cD \pare{\bT\times \bra{0, T}} $ and $ t \in \bra{0, T} $ the following equality holds true
		\begin{multline}
		\label{eq:equality_WS}
		\int _{\bT} h\pare{x_1, t}\theta\pare{x_1, t}\dx_1 - \int _{\bT} h_0\pare{x_1}\theta\pare{x_1, 0}\dx_1 - \int_0^t \int _{\bT} h\pare{x_1, t'}\partial_t \theta \pare{x_1, t'} \dx_1 \ \dt' \\
		=   \int_0^t \int _{\bT}\pare{- \sqrt{\delta} \  \left.\partial_1\Phi\pare{x_1,t'}\right|_{\Gamma\pare{t'}} \ \partial_1 h\pare{x_1,  t'} + \frac{1}{\alpha}\left.\partial_2\Phi\pare{x_1, t'}\right|_{\Gamma\pare{t'}} }\theta \pare{x_1, t'} \dx_1 \ \dt' , 
		\end{multline}
		where $ \Phi \in L^1 \pare{\bra{0, T}; H^1\pare{\Omega\pare{\cdot}}},$ is the variational solution in the sense of Definition \ref{def:variational_solution} of
		\begin{equation}
		\label{eq:elliptic2}
		\left\lbrace
		\begin{aligned}
		& \Dd \Phi \pare{t} =0, && \text{ in }\Omega\pare{t}, \\
		& \Phi\pare{t} = \nu \alpha \ccK_{\alpha} \pare{t} +\varepsilon h \pare{t} , && \text{ on }\Gamma\pare{t}\\
		& \partial_2 \Phi \pare{t} =0, && \text{ on }\Gamma_{\textnormal{b}}, 
		\end{aligned}
		\right. 
		\end{equation}
		for a.e. $ t\in\bra{0, T} $. 
	\end{definition}
	We observe that the term
	$$
	- \sqrt{\delta} \  \left.\partial_1\Phi\pare{x_1,t'}\right|_{\Gamma\pare{t'}} \ \partial_1 h\pare{x_1,  t'} + \frac{1}{\alpha}\left.\partial_2\Phi\pare{x_1, t'}\right|_{\Gamma\pare{t'}} \in H^{-1/2}(\Gamma(t')),
	$$
	due to the Normal Trace Theorem (see \cite{CG-BS2016} and the references therein).

	At first we state the well-posedness result for the dimensionless system \eqref{eq:Muskat3}:

	\begin{theorem} \label{thm:main_nondim}
		Let us define  the value
		\begin{equation*}
		\mathsf{C_0} = 3120, 
		\end{equation*}
		let $ \nu  > 1, \  \sqrt{\delta} \in \pare{\frac{1}{\nu},1}, \ T > 0 $ and let $ h_0 \in \mathbb{A}^1 $ be a zero mean initial data so that
		\begin{equation*}
		\av{h_0}_1 <  \min\{ (\mathsf{C_0} \epsilon)^{-1}, 1\} \frac{ \nu \sqrt{\delta}-1}{\nu}  , 
		\end{equation*}
		then there exists a global strong solution of \eqref{eq:Muskat3} in the sense of Definition \ref{def:weak_solution} stemming from $ h_0 $ which moreover becomes instantaneously analytic in a growing strip in the complex plane. In particular, for any
		\begin{equation*}
		\mu \in \left[0, \frac{\sqrt{\delta}}{16}\pare{\nu\sqrt{\delta} - 1}\right) ,
		\end{equation*}
		the solution lies in the energy space
		\begin{align*}
		h \in \cC\pare{\bra{0, T}; \mathbb{A}^1_{\mu t}}\cap L^1\pare{\bra{0, T}; \mathbb{A}^4_{\mu t}},\forall\,0<T<\infty. 
		\end{align*}
		In addition the following energy inequality holds true for any $ t \in\bra{0, T} $ 
		\begin{equation*}
		\av{h \pare{t}}_{1, \mu t} + \frac{\sqrt{\delta}}{16}  \pare{\nu\sqrt{\delta} - 1 } \int_0^t \av{h  \pare{t'} }_{4, \mu t'} \dt' \leqslant \av{h_0}_1,
		\end{equation*} 
		and the following decay holds
		\begin{equation*}
	\av{h\pare{t}}_{1, \mu t}\leqslant \av{h_0}_1 \exp\set{-\frac{\sqrt{\delta}}{16}\pare{\nu \sqrt{\delta}- 1} t}. 
\end{equation*}	
	\end{theorem}

	Since Theorem \ref{thm:main_nondim} is a global-well posedness result for the nondimensional system \eqref{eq:Muskat3}, we can state this result for the dimensional system \eqref{eq:Muskat2}. We observe that if $ \sqrt{\delta} > 1/\nu $ then by \eqref{eq:dimensionless_parameters}
	\begin{align*}
	\frac{H}{L} > \frac{HL\rho G}{\gamma}, && \Longleftrightarrow && L^2 < \frac{\gamma}{\rho G}. 
	\end{align*}
	Next we consider the inequality $ \sqrt{\delta} < 1 $, which is equivalent to $ H < L $. At last we consider the following chain of relations (here $ h_0 $ represents the initial datum for the Muskat problem in dimensional formulation \eqref{eq:Muskat2} and $ \tilde{h}_0 $ for the dimensionless formulation \eqref{eq:Muskat3})
	\begin{align*}
	\av{h_0}_1 & = \frac{a}{L} \av{\tilde{h}_0}_1, \\
	& \leqslant \frac{a}{L} \frac{\nu \sqrt{\delta} - 1}{\mathsf{C}_0 \ \nu \varepsilon}, \\
	& = \frac{a}{L} \frac{ \frac{\gamma}{HL\rho G} \frac{H}{L} - 1 }{\mathsf{C}_0 \ \frac{\gamma}{HL\rho G} \frac{a}{H}}, \\
	& = \frac{H^2}{L^2} \frac{\gamma -L^2 \rho G}{\mathsf{C}_0 \ \gamma }. 
	\end{align*}
	
	Then we obtain

	\begin{theorem}\label{thm:main}
		Let $ \gamma, \rho, G >0 $ be respectively the surface tension, density of the fluid, and gravitational acceleration. Define $\mathsf{C_0} = 3120$. Let $ L $ the length of the periodic cell, and $ H $ the depth of the fluid layer be such that
		\begin{align*}
		L^2 < \frac{\gamma}{\rho G}, && H < L ,
		\end{align*}
		then for any zero mean initial data $ h_0\in \mathbb{A}^1\pare{L\bT} $ such that
		\begin{equation*}
		\av{h_0}_{\mathbb{A}^1\pare{L\bT}} < \min \set{ \frac{H^2}{L^2} \ \frac{\gamma-L^2\rho G}{\mathsf{C}_0\gamma} \ , \ H }, 
		\end{equation*}
		there exists a global strong solution of \eqref{eq:Muskat} stemming from $ h_0 $ which moreover becomes instantaneously analytic in a growing strip in the complex plane.  In particular, this weak solution lies in
		\begin{align*}
		h \in \cC\pare{\bra{0, T}; \mathbb{A}^1_{\mu t}\pare{L\bT}}\cap L^1\pare{\bra{0, T}; \mathbb{A}^4_{\mu t}\pare{L\bT}},
		\end{align*}
and decays exponentially towards the equilibrium.	\end{theorem}

	The paper is divided as follows 
	
	\begin{itemize}
	
	\item In Section \ref{sec:fixed_domain} we reformulate \eqref{eq:Muskat3} in a fixed domain, i.e. 	we define a suitable family of maps $ \set{\Sigma\pare{t}, \ t> 0} $ such that $ \Sigma\pare{t} $ maps $ \cS = \bT\times\pare{-1 , 0} $ onto $ \Omega \pare{t} $. Every function $ v $ defined onto $ \Omega\pare{t} $ can hence be naturally pulled back onto $ \cS $ defining $ V=v\circ \Sigma $, and partial derivatives $ \partial_j $ onto $ \Omega\pare{t} $ are transformed in the varying coefficients differential operator $ A_j^k\partial_k, \ A = \pare{\nabla \Sigma}^{-1} $ onto $ \cS $. The map $ \Sigma $ is uniquely determined by the interface and is not chosen randomly, but in a way that the interior regularity is optimal with respect to the regularity provided by the interface. Following the terminology of \cite{Lannes2013} we will say that $ \Sigma $ is a \textit{regularizing diffeomorfism}. There are several ways to construct regularizing diffeomorphisms. Since we are dealing with small perturbations of flat interfaces we will define $ \Sigma $ as a suitable harmonic extension, but the same result (in terms of regularity) can be achieved via pseudo-differential vertical localizations as in \cite{Lannes2013} which work for arbitrary-size interface perturbations. 
	\item  The major difficulty to overcome in order to prove such a result is to deduce appropriate elliptic estimates for non-constant coefficients operators in the functional framework of the non-standard Wiener-Sobolev spaces $ \cA^{s, 1}\pare{\cS}, \ s\geqslant 1 $ (see \eqref{eq:norms_cA}). As this may be of independent interest, these elliptic estimates are proved as an autonomous abstract result in Section \ref{sec:elliptic_estimates_appendix}.
	
	\item Section \ref{sec:apriori_estimates} is the main part of the present manuscript and provides uniform a priori estimates for regular solutions of the Muskat problem. The result proved is the following: any global regular solution stemming from an initial elevation $ h_0 $ small with respect to the viscosity parameter in the critical space $ A^1\pare{\bT} $ decays toward zero and becomes instantaneously analytic. The elliptic estimates in Section \ref{sec:elliptic_estimates_appendix} are then applied to the pull back onto $ \cS $ of the velocity potential equation in Section \ref{sec:elliptic_phi2}. In Section \ref{sec:parabolic_estimates} the elliptic estimates of Section \ref{sec:elliptic_estimates} and \ref{sec:elliptic_phi2} are applied to the evolution equation for the interface. 
	
	\item In Section \ref{sec:pf_main} we prove Theorem \ref{thm:main_nondim}. This result is based on an approximationprocedure and a compactness argument.

	\end{itemize}

	\section{Reformulation of \eqref{eq:Muskat3} in a fixed domain}
	\label{sec:fixed_domain}

	Let us define the two-dimensional reference domain
	$$ 
	\cS = \bT\times \pare{-1 , 0},
	$$
	and its boundaries
	$$ \Gamma_{\textnormal{b}} = \bT\times\set{-1}, $$
	$$ \Gamma  = \bT\times\set{0}.$$
	We consider the Laplace equation with Dirichlet boundary conditions
	\begin{equation}\label{eq:Laplace_sigma}
	\left\lbrace
	\begin{aligned}
	& \Delta \sigma =0, && \text{ in }\cS, \\
	& \sigma  =  \ h ,&& \text{ on }\Gamma,\\
	& \sigma  = 0, && \text{ on }\Gamma_{\textnormal{b}}. 
	\end{aligned}
	\right. 
	\end{equation}
	
	If we assume $ h $ to be sufficiently regular we can explicitly compute the  solution of \eqref{eq:Laplace_sigma} in terms of its boundary values; 
	\begin{equation}\label{eq:sigma}
	\begin{aligned}
	\sigma\pare{x_1, x_2, t} & =  \frac{ e^{ \pare{1+x_2}\Lambda} - e^{- \pare{1+x_2}\Lambda}  }{ e^{ \Lambda} - e^{- \Lambda}  } \  h\pare{x_1, t} , \\
	& = \frac{\sinh\pare{\pare{1+x_2}\Lambda}}{\sinh\pare{\Lambda}} \  h\pare{x_1, t} , 
	&& \pare{x_1, x_2}\in \cS. 
	\end{aligned}
	\end{equation}

	We define now the family of transformations parametrized by $ t>0 $
	\begin{equation}\label{eq:def_Psi}
	\Sigma\pare{t} : \quad 
	\begin{aligned}
	\cS & \to \Omega\pare{t} \\
	\pare{x_1, x_2} & \mapsto \pare{ \big.  x_1 \ , \  x_2 + \varepsilon  \sigma\pare{x_1, x_2, t}}
	\end{aligned} \ . 
	\end{equation}

$\Sigma$ has nonnegative determinant 
$$\det \nabla \Sigma = 1+\varepsilon\partial_2 \sigma, $$	
where the Dirichlet-Neumann map in this domain takes the following explicit form
	\begin{equation}
	\label{eq:trace_pa2sigma}
	\left. \partial_2 \sigma \right|_{x_2=0} = \frac{\Lambda}{\tanh\pare{\Lambda}}h . 
	\end{equation}
However, in order we can ensure that $\Sigma$ is a diffeomorphism we need to ensure that its distance from the identity map, $\text{I}(x)$, is small. This argument has been used previously by many other authors (see \cite{hadzic2017local}). Furthermore,
$$
\|\Sigma-\text{I}\|_{\cC^1}\leq \varepsilon\|\sigma\|_{\cC^1}\leq \varepsilon |h|_1.
$$
Thus, we have that
$$
\|\Sigma(x)-\Sigma(y)\|_{\cC^0}\geq\|x-y\|_{\cC^0}-\varepsilon\|\sigma(x)-\sigma(y)\|_{\cC^0}\geq (1-2\varepsilon|h|_1)\|x-y\|_{\cC^0}.
$$
From the previous inequality we find that $\Sigma$ is injective. The surjectiveness of $\Sigma$ is easy to obtain. Then, as long as $ \av{h}_1\leq (2\varepsilon)^{-1} $, $ \Sigma $ is a (global) diffeomorphism. 	
	
	We can now define the pull-back of the potential function $ \Phi $ as 
	\begin{equation*}
	\phi = \Phi \circ \Sigma .
	\end{equation*}
	Let us denote $ A = \pare{\nabla \Sigma}^{-1} $, whose explicit expression is
	\begin{equation}\label{eq:matrixA}
	A = \frac{1}{1 + \varepsilon \partial_2 \sigma} \pare{
		\begin{array}{cc}
		1 + \varepsilon \partial_2 \sigma & 0 \\[2mm]
		-\varepsilon \partial_1 \sigma & 1
		\end{array}
	}.
	\end{equation}
	We can  deduce the following derivation rules:
	\begin{align*}
	\partial_j \Phi = A_j^k \partial_k \phi , 
	\end{align*}
	which transform \eqref{eq:Muskat3} in
	\begin{equation}
	\label{eq:Muskat4}
	\left\lbrace
	\begin{aligned}
	& A_1^j \partial_j \pare{A_1^k\partial_k \phi \big. } + \frac{1}{\delta } A_2^2 \partial_2 \pare{A_2^2\partial_2 \phi \big. } = 0,  && \text{ in }\cS,  \\
	& \partial_t h = - \sqrt{\delta}\pare{\big. A_1^k\partial_k \phi} \partial_1 h + \frac{1}{\alpha}\pare{\big. \big. A_2^2 \partial_2 \phi} ,&& \text{ on }\Gamma , \\
	& \phi = \nu \alpha  \ccK_{\alpha}+ \varepsilon h  ,&& \text{ on }\Gamma , \\
	& \partial_2\phi =0 , && \text{ on }\Gamma_{\textnormal{b}} . 
	\end{aligned}
	\right. 
	\end{equation}
	
	Let us now set $ \phi = \phi_1 +\phi_2 $, where $ \phi_j $ solves
	\begin{align}
	\nonumber
	& \left \lbrace 
	\begin{aligned}
	& \partial_1^2 \phi_1 + \frac{1}{\delta} \partial_2^2 \phi_1 =0 && \text{ in }\cS,  \\
	& \phi_1  =  \nu \alpha \ccK_{\alpha}+ \varepsilon h  ,&& \text{ on }\Gamma , \\
	& \partial_2\phi_1 =0 , && \text{ on }\Gamma_{\textnormal{b}} ,  
	\end{aligned}
	\right. \\
	\label{eq:phi2}
	& \left \lbrace 
	\begin{aligned}
	& A_1^j \partial_j \pare{A_1^k\partial_k \phi_2 \big. } + \frac{1}{\delta} A_2^2 \partial_2 \pare{A_2^2\partial_2 \phi_2 \big. } =- A_1^j \partial_j \pare{A_1^k\partial_k \phi_1 \big. } - \frac{1}{\delta} A_2^2 \partial_2 \pare{A_2^2\partial_2 \phi_1 \big. } + \partial_1^2 \phi_1 + \frac{1}{\delta} \partial_2^2 \phi_1 && \text{ in }\cS,  \\
	& \phi_2  =  0 ,&& \text{ on }\Gamma , \\
	& \partial_2\phi_2 =0 , && \text{ on }\Gamma_{\textnormal{b}} ,  
	\end{aligned}
	\right. 
	\end{align}
	
	We can compute  $ \phi_1 $ explicitly in terms of its boundary values as follows 
	\begin{equation}\label{eq:phi1}
	\begin{aligned}
	\phi_1 \pare{x_1, x_2, t} & = \frac{e^{-\sqrt{\delta}\pare{1+x_2}\Lambda} + e^{\sqrt{\delta}\pare{1+x_2}\Lambda}}{e^{\sqrt{\delta}\Lambda} + e^{-\sqrt{\delta}\Lambda}} \ \pare{ \big. \nu \alpha \ccK_{\alpha}+ \varepsilon h\pare{x_1, t}}, \\
	& = \frac{\cosh\pare{\sqrt{\delta}\pare{1+x_2}\Lambda}}{\cosh\pare{\sqrt{\delta}\Lambda}}\ \pare{ \big. \nu \alpha \ccK_{\alpha}+ \varepsilon h\pare{x_1, t}}, && \pare{x_1, x_2} \in \bT \times \bra{0, 1}, t\in[0,T]. 
	\end{aligned}
	\end{equation}
	Differentiating \eqref{eq:phi1} we derive the following useful formulas
	\begin{equation*}
	\begin{aligned}
	\partial_2 \phi_1 \pare{x_1, x_2, t} & = \sqrt{\delta} \ \frac{\sinh\pare{\sqrt{\delta}\pare{1+x_2}\Lambda}}{\cosh\pare{\sqrt{\delta}\Lambda}}\ \Lambda\pare{ \big. \nu \alpha \ccK_{\alpha}+ \varepsilon h\pare{x_1, t}}, \\
	\partial_2 \phi_1 \pare{x_1, 0, t} & = \sqrt{\delta} \tanh\pare{\sqrt{\delta} \Lambda} \ \Lambda \pare{ \big. \nu \alpha \ccK_{\alpha}+ \varepsilon h\pare{x_1, t}}. 
	\end{aligned}
	\end{equation*}

	We can now use the decomposition $ \phi=\phi_1 +\phi_2 $ in the evolution equation for the elevation $ h $; 
	\begin{equation*}
	\partial_t h = - \sqrt{\delta} \pare{\big. A_1^k\partial_k \pare{\phi_1 + \phi_2}} \partial_1 h + \frac{1}{\alpha}\pare{\big. \big. A_2^2 \partial_2 \pare{\phi_1 + \phi_2}}. 
	\end{equation*}
	We now want to make explicitly appear the parabolic smoothing effects induced by the surface tension and the backward parabolic behavior due to the gravity unstable configuration in which we are working. In order to do so let us consider the term
	\begin{equation*}
	\begin{aligned}
	A_2^2 \partial_2 \phi_1 & =  \frac{\sqrt{\delta}}{1+ \varepsilon\partial_2 \sigma} \tanh\pare{\sqrt{\delta} \Lambda}  \Lambda \pare{\nu\alpha \ccK_{\alpha}+ \varepsilon h}, \\
	& =  -\sqrt{\delta} \tanh\pare{\sqrt{\delta} \Lambda} \pare{\nu \alpha \Lambda^3 h - \varepsilon \Lambda h}\\
	& \quad +\sqrt{\delta} \bra{ \frac{\varepsilon \partial_2 \sigma}{1+ \varepsilon \partial_2 \sigma} \tanh\pare{\sqrt{\delta}\Lambda} \pare{\nu \alpha \Lambda^3 h - \varepsilon \Lambda h} + \frac{\nu \alpha}{1+\varepsilon\partial_2 \sigma} \tanh\pare{\sqrt{\delta}\Lambda} \Lambda \pare{\big. \ccK_{\alpha} - \partial_1 ^2 h}}. 
	\end{aligned}
	\end{equation*}

After using the identity $ \sqrt{\delta} / \alpha = 1/ \varepsilon $ (see \eqref{eq:dimensionless_parameters}), the equation for the elevation $ h $ becomes 
	\begin{multline*}
	\partial_t h + \frac{1}{\varepsilon} \tanh\pare{\sqrt{\delta} \Lambda} \pare{\nu \alpha \Lambda^3 h - \varepsilon \Lambda h} = - \sqrt{\delta}\pare{\big. A_1^k\partial_k \pare{\phi_1 + \phi_2}} \partial_1 h + \frac{1}{\alpha}\pare{\big. \big. A_2^2 \partial_2  \phi_2} \\ 
	+\sqrt{\delta} \bra{ \frac{\varepsilon \partial_2 \sigma}{1+ \varepsilon \partial_2 \sigma} \tanh\pare{\sqrt{\delta}\Lambda} \pare{\nu \alpha \Lambda^3 h - \varepsilon \Lambda h} + \frac{\nu \alpha}{1+\varepsilon\partial_2 \sigma} \tanh\pare{\sqrt{\delta}\Lambda} \Lambda \pare{\big. \ccK_{\alpha} - \partial_1 ^2 h}}.
	\end{multline*}

	We use now the explicit expression of the functions $ \mathsf{F} $ and $ \mathsf{G} $ provided in \eqref{eq:def_F} in order to simplify the nonlinear term as
	\begin{multline*}
	\frac{\varepsilon \partial_2 \sigma}{1+ \varepsilon \partial_2 \sigma} \tanh\pare{\sqrt{\delta}\Lambda} \pare{\nu \alpha \Lambda^3 h - \varepsilon \Lambda h} + \frac{\nu \alpha}{1+\varepsilon\partial_2 \sigma} \tanh\pare{\sqrt{\delta}\Lambda} \Lambda \pare{\big. \ccK_{\alpha} - \partial_1 ^2 h} \\
	= \mathsf{G}\pare{\varepsilon \partial_2 \sigma} \tanh\pare{\sqrt{\delta}\Lambda} \pare{\nu \alpha \Lambda^3 h - \varepsilon \Lambda h} - \nu \alpha \pare{1-\mathsf{G}\pare{\varepsilon\partial_2 \sigma}} \tanh\pare{\sqrt{\delta}\Lambda} \Lambda \pare{\mathsf{F}\pare{\alpha h'} \Lambda^2 h} , 
	\end{multline*}
	and we use \eqref{eq:trace_pa2sigma} in order to deduce
	\begin{multline*}
	\mathsf{G}\pare{\varepsilon \partial_2 \sigma} \tanh\pare{\sqrt{\delta}\Lambda} \pare{\nu \alpha \Lambda^3 h - \varepsilon \Lambda h} - \nu \alpha \pare{1-\mathsf{G}\pare{\varepsilon\partial_2 \sigma}} \tanh\pare{\sqrt{\delta}\Lambda} \Lambda \pare{\mathsf{F}\pare{\alpha h'} \Lambda^2 h} \\
	= \mathsf{G}\pare{ \frac{\varepsilon \Lambda}{\tanh\pare{\Lambda}} h} \tanh\pare{\sqrt{\delta}\Lambda} \pare{\nu \alpha \Lambda^3 h - \varepsilon \Lambda h} - \nu \alpha \pare{1-\mathsf{G}\pare{\frac{\varepsilon \Lambda}{\tanh\pare{\Lambda}} h}} \tanh\pare{\sqrt{\delta}\Lambda} \Lambda \pare{\mathsf{F}\pare{\alpha h'} \Lambda^2 h} .
	\end{multline*}
	Thus, the ALE formulation of the Muskat problem \eqref{eq:Muskat3} is given by the system
	\begin{equation}\label{eq:Muskat5}
	\left\lbrace
	\begin{aligned}
	& \begin{multlined}
	\partial_t h + \frac{1}{\varepsilon} \tanh\pare{\sqrt{\delta} \Lambda} \pare{\nu \alpha \Lambda^3 h - \varepsilon \Lambda h} = - \sqrt{\delta}\pare{\big. A_1^k\partial_k \pare{\phi_1 + \phi_2}} \partial_1 h + \frac{1}{\alpha} \pare{\big. \big. A_2^2 \partial_2  \phi_2} \\ 
	+\sqrt{\delta} \left[
	\mathsf{G}\pare{ \frac{\varepsilon \Lambda}{\tanh\pare{\Lambda}} h} \tanh\pare{\sqrt{\delta}\Lambda} \pare{\nu \alpha \Lambda^3 h - \varepsilon \Lambda h}\right. \\
	\left. - \nu \alpha \pare{1-\mathsf{G}\pare{\frac{\varepsilon \Lambda}{\tanh\pare{\Lambda}} h}} \tanh\pare{\sqrt{\delta}\Lambda} \Lambda \pare{\mathsf{F}\pare{\alpha h'} \Lambda^2 h} 
	\right]    , 
	\end{multlined} && \text{ on } \Gamma\\[4mm]
	& \phi_1 \pare{x_1, x_2, t} =\frac{\cosh\pare{\sqrt{\delta}\pare{1+x_2}\Lambda}}{\cosh\pare{\sqrt{\delta}\Lambda}}\ \pare{ \big. \nu \alpha \ccK_{\alpha}+ \varepsilon h\pare{x_1, t}}&& \text{ in }\cS,  \\[4mm]
	& A_1^j \partial_j \pare{A_1^k\partial_k \phi_2 \big. } + \frac{1}{\delta} A_2^2 \partial_2 \pare{A_2^2\partial_2 \phi_2 \big. } =- A_1^j \partial_j \pare{A_1^k\partial_k \phi_1 \big. } - \frac{1}{\delta} A_2^2 \partial_2 \pare{A_2^2\partial_2 \phi_1 \big. } + \partial_1^2 \phi_1 + \frac{1}{\delta} \partial_2^2 \phi_1 && \text{ in }\cS,  \\
	& \phi_2  =  0 ,&& \text{ on }\Gamma , \\
	& \partial_2\phi_2 =0 , && \text{ on }\Gamma_{\textnormal{b}} . 
	\end{aligned}
	\right. 
	\end{equation}

	\section{Elliptic estimates in Wiener-Sobolev spaces} \label{sec:elliptic_estimates_appendix}

	The present section is dedicated to prove some elliptic estimates in the functional framework of Wiener-Sobolev spaces (see Proposition \ref{prop:elliptic_estimates} below). Such a result is, to the best of our knowledge, new and it is a crucial technical tool required in order to provide optimal bounds for $ \phi_2 $.

	In this section we consider the following Poisson problem for a given sufficiently regular $ {\bf g} = \pare{g_1, g_2}^{\intercal} $: 
	
	\begin{equation}
	\label{eq:Poisson}
	\left\lbrace
	\begin{aligned}
	& \Dd \varphi = \nd \cdot {\bf g}, & \text{ in } &\cS,  \\
	& \varphi  =0,  & \text{ on } & \Gamma \\
	& \partial_2 \varphi =0, & \text{ on } & \Gamma_{\textnormal{b}} . 
	\end{aligned}
	\right.
	\end{equation}

	\begin{theorem} \label{prop:elliptic_estimates}
		Let $ {\bf g}\in \cA^{s_0, 1}_\lambda , \ s_0\geqslant 1, \ \lambda\geqslant 0  $, there exists a unique $ \varphi $ solution of \eqref{eq:Poisson}, moreover such solution satisfies  the estimate
		\begin{equation*}
		\norm{\nd \varphi}_{\cA^{s, 1}_\lambda} \leqslant 13 \norm{{\bf g}}_{\cA^{s, 1}_\lambda}, 
		\end{equation*}
		for any $ s\geqslant 0 $.
	\end{theorem}

	\textit{ Proof of Theorem \ref{prop:elliptic_estimates} :}
	For the sake of readability we divide the proof in several steps. 
	
	\subsubsection*{Step 1: derivation of the explicit solution of \eqref{eq:Poisson} }
	
To simplify the notation, let us for the moment define
	\begin{equation*}
	b = \nd \cdot {\bf g} .
	\end{equation*}
We have that $ b\in \cA^{s_0-1, 0} $, and since $ s_0-1 \geqslant 0 $ for any fixed $ \bar{x}_2\in \pare{-1, 0} $ we can define the Fourier series in the $x_1$ direction 
$$ 
\pare{\hat{b}\pare{k, \bar{x}_2}}_{k\in\bZ} .$$
 The Fourier series in $ x_1 $ transforms \eqref{eq:Poisson} in the sequence of ODEs
	\begin{align*}
	-\delta k^2 \ \hat{\varphi}\pare{k, x_2 } + \partial_2^2 \hat{\varphi}\pare{k, x_2 } = \hat{b}\pare{k, x_2 } , && k \in \bZ, 
	\end{align*}
	These ODEs are explicitly solvable and we find that
	\begin{equation*}
	\hat{\varphi}\pare{k, x_2} = C_1 \pare{k} e^{\sqrt{\delta}\av{k}x_2} + C_2 \pare{k} e^{- \sqrt{\delta} \av{k}x_2} - \int_0^{x_2} \frac{\hat{b}\pare{k, y_2}}{2\sqrt{\delta} \av{k}} \bra{e^{\sqrt{\delta}\av{k}\pare{y_2 - x_2 }} - e^{-\sqrt{\delta}\av{k}\pare{y_2 - x_2 }}} \textnormal{d} y_2. 
	\end{equation*}
	Using now the first boundary condition we deduce that
	\begin{equation*}
	C_2 = -C_1. 
	\end{equation*}
	Next we compute
	\begin{equation*}
	\partial_2 \hat{\varphi}\pare{k, x_2} = C_1 \pare{k} \sqrt{\delta}\av{k} \pare{ e^{\sqrt{\delta}\av{k}x_2} +  e^{-\sqrt{\delta} \av{k}x_2}} \\
	+ \frac{1}{2}\int_0^{x_2} {\hat{b}\pare{k, y_2}} \bra{e^{\sqrt{\delta}\av{k}\pare{y_2 - x_2 }} + e^{\sqrt{\delta}\av{k}\pare{x_2 - y_2 }}} \textnormal{d} y_2 , 
	\end{equation*} 
	which combined with the condition $ \partial_2 \varphi \pare{x_1, -1} =0  $ gives
	\begin{equation*}
	C_1 \pare{k} = \frac{1}{2 \sqrt{\delta}\av{k} }\int _{-1}^0 {\hat{b}\pare{k, y_2}} \bra{\frac{e^{\sqrt{\delta}\av{k}\pare{ 1 + y_2  }} + e^{-\sqrt{\delta}\av{k}\pare{1+ y_2 }}}{e^{\sqrt{\delta}\av{k}} +  e^{-\sqrt{\delta} \av{k}}}} \textnormal{d} y_2 . 
	\end{equation*}
	Thus, we obtained that
	\begin{multline}
	\label{eq:psi_explicit}
	\hat{\varphi} \pare{k, x_2} =  \frac{1}{2 \sqrt{\delta}\av{k} }\int _{-1}^0 {\hat{b}\pare{k, y_2}} \bra{\frac{e^{\sqrt{\delta}\av{k}\pare{ 1 + y_2  }} + e^{-\sqrt{\delta}\av{k}\pare{1+ y_2 }}}{e^{\sqrt{\delta}\av{k}} +  e^{-\sqrt{\delta} \av{k}}}} \textnormal{d} y_2 \  \pare{ e^{\sqrt{\delta}\av{k}x_2} -  e^{-\sqrt{\delta} \av{k}x_2}} \\
	+ \int_{x_2}^0 \frac{\hat{b}\pare{k, y_2}}{2\sqrt{\delta} \av{k}} \bra{e^{\sqrt{\delta}\av{k}\pare{y_2 - x_2 }} - e^{-\sqrt{\delta}\av{k}\pare{y_2 - x_2 }}} \textnormal{d} y_2 .
	\end{multline}
	Since $ C_j, \ j=1, 2 $ are uniquely determined by the boundary conditions the solution derived is unique. \\

	Let us remark that if we define
	\begin{equation}
	\label{eq:Pi_j}
	\begin{small}
	\begin{aligned}
	\Pi_1 \pare{\sqrt{\delta}\av{k}, y_2, x_2} = & \ \frac{\cosh\pare{\sqrt{\delta}\av{k}\pare{1+y_2}} \ \sinh \pare{\sqrt{\delta}\av{k} x_2} }{\cosh\pare{\sqrt{\delta}\av{k}}}, & y_2 \in & \ \bra{-1, x_2},  \\
	\Pi_2 \pare{\sqrt{\delta}\av{k}, y_2, x_2} = & \ \frac{\cosh\pare{\sqrt{\delta}\av{k}\pare{1+y_2}} \ \sinh \pare{\sqrt{\delta}\av{k} x_2} }{\cosh\pare{\sqrt{\delta}\av{k}}} + \sinh \pare{ \sqrt{\delta}\av{k}  \pare{y_2 - x_2} }, & y_2 \in & \ \bra{ x_2, 0},   
	\end{aligned}
	\end{small}
	\end{equation}
	then $ \hat{\varphi} $ given in \eqref{eq:psi_explicit} can be re-written as
	\begin{equation}
	\label{eq:psi_explicit2}
	\hat{\varphi} \pare{k, x_2} =  \frac{1}{ \sqrt{\delta}\av{k} } \bra{  \int _{-1}^{x_2} \Pi_1\pare{\sqrt{\delta}\av{k}, y_2, x_2} \hat{b}\pare{k, y_2} \textnormal{d} y_2   \\
		+\int _{x_2}^0  \Pi_2\pare{\sqrt{\delta}\av{k}, y_2, x_2} \hat{b}\pare{k, y_2} \textnormal{d} y_2} . 
	\end{equation}

	Let us notice now that
	\begin{equation}
	\label{eq:Pi1_as_negative_exp}
	\begin{aligned}
	\av{ 2 \Pi_1\pare{\sqrt{\delta}\av{k}, y_2, x_2}} = & \   \frac{\pare{e^{\sqrt{\delta}\av{k}\pare{ 1 + y_2  }} + e^{-\sqrt{\delta}\av{k}\pare{1+ y_2 }}}\pare{   e^{-\sqrt{\delta} \av{k}x_2} -e^{\sqrt{\delta}\av{k}x_2} }}{e^{\sqrt{\delta}\av{k}} +  e^{-\sqrt{\delta} \av{k}}}\\
	\lesssim & \ \frac{e^{\sqrt{\delta}\av{k}\pare{1+y_2-x_2}}}{e^{\sqrt{\delta}\av{k}}} \\ 
	\leqslant & \  e^{\sqrt{\delta}\av{k}\pare{y_2-x_2}} \leqslant 1, 
	\end{aligned}
	\end{equation}
	since in such term the integration was performed for $ y_2 \in \bra{-1, x_2} $ and hence $ y_2 - x_2 \leqslant 0 $. We now estimate the term $ \Pi_2 $ which is defined for $ y_2\in\pare{x_2, 0} $. Let us rewrite this term as
	\begin{multline*}
	\Pi_2\pare{\sqrt{\delta}\av{k}, y_2, x_2} \\ =
	\frac{1}{2} \bra{\pare{e^{\sqrt{\delta}\av{k}\pare{y_2 - x_2 }} - e^{-\sqrt{\delta}\av{k}\pare{y_2 - x_2 }}}
		- \frac{\pare{e^{\sqrt{\delta}\av{k}\pare{ 1 + y_2  }} + e^{-\sqrt{\delta}\av{k}\pare{1+ y_2 }}}\pare{   e^{-\sqrt{\delta} \av{k}x_2} -e^{\sqrt{\delta}\av{k}x_2} }}{e^{\sqrt{\delta}\av{k}} +  e^{-\sqrt{\delta} \av{k}}}}.
	\end{multline*}
In this way we outline the negativity of the second term on the right hand side of the above equality. 
	Indeed,  as $ y_2-x_2 > 0 $, we have that
	\begin{equation}
	\label{eq:Pi2_as_negative_exp}
	\begin{aligned}
	& 2\Pi_{2}\pare{\sqrt{\delta}\av{k}, y_2, x_2} \\
	& = e^{\sqrt{\delta}\av{k}\pare{y_2 - x_2 }} \bra{1 -  \frac{
			\pare{1+e^{-2\sqrt{\delta}\av{k}\pare{1+y_2}}} \pare{1-e^{2\sqrt{\delta}\av{k}x_2}}
		}{1+e^{-2\sqrt{\delta}\av{k}}}}
	- e^{-\sqrt{\delta}\av{k}\pare{y_2 - x_2 }}.
	\\
	& = e^{\sqrt{\delta}\av{k}\pare{y_2 - x_2 }} \bra{  \frac{ e^{-2\sqrt{\delta}\av{k}} - e^{-2\sqrt{\delta}\av{k}\pare{1+y_2}} + e^{2\sqrt{\delta}\av{k}x_2} + e^{-2\sqrt{\delta}\av{k}\pare{1+y_2 - x_2}}}{1+e^{-2\sqrt{\delta}\av{k}}}}
	- e^{-\sqrt{\delta}\av{k}\pare{y_2 - x_2 }}, \\
	& = \frac{
		e^{\sqrt{\delta}\av{k}\pare{-2 + y_2 -x_2}} 
		-e^{\sqrt{\delta}\av{k}\pare{-2-y_2-x_2}}
		+ e^{\sqrt{\delta}\av{k}\pare{x_2 +y_2}}  
		+ e^{\sqrt{\delta}\av{k}\pare{-2-y_2+x_2}} 
	}{1+e^{-2\sqrt{\delta}\av{k}}} - e^{-\sqrt{\delta}\av{k}\pare{y_2 - x_2 }}. 
	\end{aligned}
	\end{equation}
	Since all the arguments of the exponentials above are non-positive for $ x_2\in\bra{-1, 0} $ and strictly negative for $ x_2\in\pare{- 1, 0} $ we deduce as well that
	\begin{equation*}
	\av{ \Pi_2\pare{\sqrt{\delta}\av{k}, y_2, x_2}} \lesssim  1. 
	\end{equation*}

	\subsubsection*{Step 2: The explicit solution when the forcing is in divergence form }
	
	The present step consists of rather standard computations. Let us use the explicit expression of the $ \Pi_j $ given in \eqref{eq:Pi_j} in order to compute
	\begin{equation}
	\label{eq:Pi_derivatives}
	\begin{small}
	\begin{aligned}
	\partial_{x_2} \Pi_1 \pare{\sqrt{\delta}\av{k}, y_2, x_2} = & \ \sqrt{\delta}\av{k} \frac{\cosh\pare{\sqrt{\delta}\av{k}\pare{1+y_2}} \ \cosh \pare{\sqrt{\delta}\av{k} x_2} }{\cosh\pare{\sqrt{\delta}\av{k}}}, \\
	\partial_{y_2} \Pi_1 \pare{\sqrt{\delta}\av{k}, y_2, x_2} = & \ \sqrt{\delta}\av{k}  \frac{\sinh\pare{\sqrt{\delta}\av{k}\pare{1+y_2}} \ \sinh \pare{\sqrt{\delta}\av{k} x_2} }{\cosh\pare{\sqrt{\delta}\av{k}}}
	, \\
	\partial_{x_2}\Pi_2 \pare{\sqrt{\delta}\av{k}, y_2, x_2} = & \ \sqrt{\delta}\av{k} \bra{ \frac{\cosh\pare{\sqrt{\delta}\av{k}\pare{1+y_2}} \ \cosh \pare{\sqrt{\delta}\av{k} x_2} }{\cosh\pare{\sqrt{\delta}\av{k}}} - \cosh \pare{ \sqrt{\delta}\av{k}  \pare{y_2 - x_2} }}
	, \\
	\partial_{y_2}\Pi_2 \pare{\sqrt{\delta}\av{k}, y_2, x_2} = & \ \sqrt{\delta}\av{k} \bra{ \frac{\sinh\pare{\sqrt{\delta}\av{k}\pare{1+y_2}} \ \sinh \pare{\sqrt{\delta}\av{k} x_2} }{\cosh\pare{\sqrt{\delta}\av{k}}} + \cosh \pare{ \sqrt{\delta}\av{k}  \pare{y_2 - x_2} }}, \\
	\partial_{x_2} \partial_{y_2}\Pi_1 \pare{\sqrt{\delta}\av{k}, y_2, x_2} = & \ \pare{\sqrt{\delta}\av{k}}^2 \frac{\sinh\pare{\sqrt{\delta}\av{k}\pare{1+y_2}} \ \cosh \pare{\sqrt{\delta}\av{k} x_2} }{\cosh\pare{\sqrt{\delta}\av{k}}}, \\
	\partial_{x_2} \partial_{y_2} \Pi_2 \pare{\sqrt{\delta}\av{k}, y_2, x_2} = & \ \pare{\sqrt{\delta}\av{k}}^2 \bra{ \frac{\sinh\pare{\sqrt{\delta}\av{k}\pare{1+y_2}} \ \cosh \pare{\sqrt{\delta}\av{k} x_2} }{\cosh\pare{\sqrt{\delta}\av{k}}} - \sinh \pare{ \sqrt{\delta}\av{k}  \pare{y_2 - x_2} }} .
	\end{aligned}
	\end{small}
	\end{equation}

	We can now use the explicit formulation of $ \varphi $ provided in \eqref{eq:psi_explicit2} and explicit the forcing $ b = \nd \cdot {\bf g} $, obtaining
	\begin{multline*}
	\hat{\varphi}\pare{k, x_2 } = \frac{1}{\sqrt{\delta}\av{k}} \int_{-1}^{x_2} \Pi_1 \pare{\sqrt{\delta}\av{k}, y_2, x_2} \pare{i\sqrt{\delta}k \hat{g}_1 \pare{k, y_2} + \partial_{y_2}\hat{g}_2 \pare{k, y_2}} \dy_2 \\
	+ \frac{1}{\sqrt{\delta}\av{k}} \int_{x_2}^{0} \Pi_2 \pare{\sqrt{\delta}\av{k}, y_2, x_2} \pare{i\sqrt{\delta}k \hat{g}_1 \pare{k, y_2} + \partial_{y_2}\hat{g}_2 \pare{k, y_2}} \dy_2. 
	\end{multline*}
We decompose
	\begin{equation*}
	\varphi = \varphi_1 + \varphi_2,
	\end{equation*}
	with
	\begin{align}
	\hat{\varphi_1}\pare{k, x_2 } = & \  
	i\sgn \pare{k} \pare{\int_{-1}^{x_2} \Pi_1 \pare{\sqrt{\delta}\av{k}, y_2, x_2} \hat{g}_1 \pare{k, y_2}  \dy_2 +  \int_{x_2}^{0} \Pi_2 \pare{\sqrt{\delta}\av{k}, y_2, x_2} \hat{g}_1 \pare{k, y_2}  \dy_2} , \label{eq:psi1} \\
	\hat{\varphi}_2\pare{k, x_2 } = & \ \frac{1}{\sqrt{\delta}\av{k}} \pare{ \int_{-1}^{x_2} \Pi_1 \pare{\sqrt{\delta}\av{k}, y_2, x_2}  \partial_{y_2}\hat{g}_2 \pare{k, y_2} \dy_2 +\int_{x_2}^{0} \Pi_2 \pare{\sqrt{\delta}\av{k}, y_2, x_2}  \partial_{y_2}\hat{g}_2 \pare{k, y_2} \dy_2}. \nonumber 
	\end{align}
	We can integrate by parts in order to deduce that
	\begin{align*}
	\hat{\varphi}_2\pare{k, x_2} = & \ \frac{1}{\sqrt{\delta}\av{k}} \bra{ \Big.
		\Pi_1\pare{\sqrt{\delta}\av{k}, x_2, x_2}\hat{g}_2 \pare{k, x_2} - \Pi_1\pare{\sqrt{\delta}\av{k}, -1, x_2}\hat{g}_2 \pare{k, -1}
	}
	\\
	& \ - \frac{1}{\sqrt{\delta}\av{k}} \bra{ \Big.
		\Pi_2 \pare{\sqrt{\delta}\av{k}, x_2, x_2}\hat{g}_2 \pare{k, x_2} - \Pi_2 \pare{\sqrt{\delta}\av{k}, 0, x_2}\hat{g}_2 \pare{k, 0}
	} \\
	& \  + \frac{1}{\sqrt{\delta}\av{k}} \bra{ \int _{-1}^{x_2} {\hat{g}_2\pare{k, y_2}} \partial_{y_2} \Pi_1\pare{\sqrt{\delta}\av{k}, y_2, x_2} \textnormal{d} y_2 
		+ \int _{x_2}^0 {\hat{g}_2\pare{k, y_2}}  \partial_{y_2} \Pi_2\pare{\sqrt{\delta}\av{k}, y_2, x_2} \textnormal{d} y_2} . 
	\end{align*}

Using the identity
	\begin{equation*}
	\Pi_1 \pare{\sqrt{\delta}\av{k}, x_2, x_2} - \Pi_2 \pare{\sqrt{\delta}\av{k}, x_2, x_2} =0, 
	\end{equation*}
$ \varphi_2 $ can be written as
	\begin{multline*} 
	\hat{\varphi}_2\pare{k, x_2} = 
	\frac{1}{\sqrt{\delta}\av{k}} \bra{ \Big.
		\Pi_2 \pare{\sqrt{\delta}\av{k}, 0, x_2}\hat{g}_2 \pare{k, 0} - \Pi_1\pare{\sqrt{\delta}\av{k}, -1, x_2}\hat{g}_2 \pare{k, -1}
	}\\
	+ \int _{-1}^{x_2} {\hat{g}_2\pare{k, y_2}} \frac{1}{\sqrt{\delta}\av{k}} \partial_{y_2} \Pi_1\pare{\sqrt{\delta}\av{k}, y_2, x_2} \textnormal{d} y_2 \\
	+ \int _{x_2}^0 {\hat{g}_2 \pare{k, y_2}} \frac{1}{\sqrt{\delta}\av{k}} \partial_{y_2} \Pi_2\pare{\sqrt{\delta}\av{k}, y_2, x_2} \textnormal{d} y_2 .
	\end{multline*}
	
	Moreover we observe that
	\begin{align}\label{eq:id_integration_kernels1}
	\Pi_2 \pare{\sqrt{\delta}\av{k}, 0, x_2} =0, && \Pi_1 \pare{\sqrt{\delta}\av{k}, -1, x_2} = \frac{\sinh\pare{\sqrt{\delta}\av{k}x_2}}{\cosh\pare{\sqrt{\delta}\av{k}}},
	\end{align}
	from where
	\begin{multline} \label{eq:psi2}
	\hat{\varphi}_2\pare{k, x_2} = 
	-\frac{1}{\sqrt{\delta}\av{k}} \frac{\sinh\pare{\sqrt{\delta}\av{k}x_2}}{\cosh\pare{\sqrt{\delta}\av{k}}}\hat{g}_2 \pare{k, -1}
	+ \int _{-1}^{x_2} {\hat{g}_2\pare{k, y_2}} \frac{1}{\sqrt{\delta}\av{k}} \partial_{y_2} \Pi_1\pare{\sqrt{\delta}\av{k}, y_2, x_2} \textnormal{d} y_2 \\
	+ \int _{x_2}^0 {\hat{g}_2 \pare{k, y_2}} \frac{1}{\sqrt{\delta}\av{k}} \partial_{y_2} \Pi_2\pare{\sqrt{\delta}\av{k}, y_2, x_2} \textnormal{d} y_2 .
	\end{multline}

	The above expression is important because it allows us to compute $ \partial_2 \varphi_2 $ without differentiating $ g_2 $.

	\subsubsection*{Step 3: Computation of $ \nd \varphi $ and $ \partial_2 \nd \varphi  $  }
	
	We can use the explicit definition of $ \varphi_1 $ and $ \varphi_2 $, computed in the previous step and provided respectively in \eqref{eq:psi1} and \eqref{eq:psi2} together with \eqref{eq:id_integration_kernels1} and the identity 
	\begin{equation*}
	\partial_{y_2} \Pi_1 \pare{\sqrt{\delta}\av{k}, x_2, x_2} - \partial_{y_2} \Pi_2 \pare{\sqrt{\delta}\av{k}, x_2, x_2} =\sqrt{\delta}\av{k},
	\end{equation*} in order to obtain
	\begin{equation}
	\label{eq:pa2_psi}
	\begin{aligned}
	\partial_2 \hat{\varphi}_1  \pare{k, x_2} = &  \  i\sgn \pare{k}  \left[ \int_{-1}^{x_2} \partial_{x_2}\Pi_1 \pare{\sqrt{\delta}\av{k}, y_2, x_2} \hat{g}_1 \pare{k, y_2}  \dy_2 \right. \\
	& \qquad  \qquad \left. +  \int_{x_2}^{0} \partial_{x_2} \Pi_2 \pare{\sqrt{\delta}\av{k}, y_2, x_2} \hat{g}_1 \pare{k, y_2}  \dy_2 \right] , \\
	\partial_2 \hat{\varphi}_2  \pare{k, x_2} = & \ - \frac{\cosh\pare{\sqrt{\delta}\av{k}x_2}}{\cosh\pare{\sqrt{\delta}\av{k}}}\hat{g}_2 \pare{k, -1}
	+ \hat{g}_2\pare{k, x_2}\\
	& \  + \int _{-1}^{x_2} {\hat{g}_2\pare{k, y_2}} \frac{1}{\sqrt{\delta}\av{k}} \partial_{x_2}\partial_{y_2} \Pi_1\pare{\sqrt{\delta}\av{k}, y_2, x_2} \textnormal{d} y_2 \\
	& \  + \int _{x_2}^0 {\hat{g}_2 \pare{k, y_2}} \frac{1}{\sqrt{\delta}\av{k}} \partial_{x_2}\partial_{y_2} \Pi_2\pare{\sqrt{\delta}\av{k}, y_2, x_2} \textnormal{d} y_2 .
	\end{aligned}
	\end{equation}

We also obtain that
	\begin{equation}
	\label{eq:pa1_psi}
	\begin{small}
	\begin{aligned}
	\sqrt{\delta} \ \widehat{\partial_1 \varphi_1} \pare{k, x_2}  = & \  
	\sqrt{\delta} i \av{k} \bra{ \int_{-1}^{x_2} \Pi_1 \pare{\sqrt{\delta}\av{k}, y_2, x_2} \hat{g}_1 \pare{k, y_2}  \dy_2 +  \int_{x_2}^{0} \Pi_2 \pare{\sqrt{\delta}\av{k}, y_2, x_2} \hat{g}_1 \pare{k, y_2}  \dy_2} ,  \\
	\sqrt{\delta} \ \widehat{\partial_1 \varphi_2}\pare{k, x_2 } = & \ i\sgn\pare{k} \Bigg[ -\frac{\sinh\pare{\sqrt{\delta}\av{k}x_2}}{\cosh\pare{\sqrt{\delta}\av{k}}}\hat{g}_2 \pare{k, -1}
	\\
	&  + \int _{-1}^{x_2} {\hat{g}_2\pare{k, y_2}}  \partial_{y_2} \Pi_1\pare{\sqrt{\delta}\av{k}, y_2, x_2} \textnormal{d} y_2 + \int _{x_2}^0 {\hat{g}_2 \pare{k, y_2}}  \partial_{y_2} \Pi_2\pare{\sqrt{\delta}\av{k}, y_2, x_2} \textnormal{d} y_2 \Bigg] .
	\end{aligned}
	\end{small}
	\end{equation}

	Our final goal is to deduce an estimate for the $ \cA^{s, 1} $ norm of $ \nd \varphi $ in terms of the $ \cA^{s, 1} $ norm of $ {\bf g} $. In order for this to be possible we must be able to express $ \partial_2 \nd \varphi $ as a function of $ \partial_2{\bf g} $. In order to do so we remark that the following identities hold true (cf. \eqref{eq:Pi_j} and \eqref{eq:Pi_derivatives})
	\begin{align*}
	\Pi_1 \pare{\sqrt{\delta}\av{k}, y_2, x_2} & = \frac{1}{\pare{\sqrt{\delta}\av{k}}^2} \  \partial_{y_2}^2  \Pi_1 \pare{\sqrt{\delta}\av{k}, y_2, x_2} , \\
	\Pi_2 \pare{\sqrt{\delta}\av{k}, y_2, x_2} & = \frac{1}{\pare{\sqrt{\delta}\av{k}}^2} \  \partial_{y_2}^2 \Pi_2 \pare{\sqrt{\delta}\av{k}, y_2, x_2} ,\\
	\partial_{x_2} \Pi_1 \pare{\sqrt{\delta}\av{k}, y_2, x_2} & = \frac{1}{\pare{\sqrt{\delta}\av{k}}^2} \  \partial_{y_2}^2 \partial_{x_2} \Pi_1 \pare{\sqrt{\delta}\av{k}, y_2, x_2} , \\
	\partial_{x_2} \Pi_2 \pare{\sqrt{\delta}\av{k}, y_2, x_2} & = \frac{1}{\pare{\sqrt{\delta}\av{k}}^2} \  \partial_{y_2}^2 \partial_{x_2} \Pi_2 \pare{\sqrt{\delta}\av{k}, y_2, x_2} ,  
	\end{align*}
	which we use in order to transform \eqref{eq:pa2_psi} and \eqref{eq:pa1_psi} into
	\begin{equation}
	\label{eq:nd_psi}
	\begin{aligned}
	\partial_2 \hat{\varphi}_1  \pare{k, x_2} = &  \  i\sgn \pare{k}  \left[ \int_{-1}^{x_2} \frac{1}{\pare{\sqrt{\delta}\av{k}}^2} \  \partial_{y_2}^2 \partial_{x_2} \Pi_1 \pare{\sqrt{\delta}\av{k}, y_2, x_2} \hat{g}_1 \pare{k, y_2}  \dy_2 \right. \\
	& \qquad  \qquad \left. +  \int_{x_2}^{0} \frac{1}{\pare{\sqrt{\delta}\av{k}}^2} \  \partial_{y_2}^2 \partial_{x_2} \Pi_2 \pare{\sqrt{\delta}\av{k}, y_2, x_2} \hat{g}_1 \pare{k, y_2}  \dy_2 \right] , \\
	\partial_2 \hat{\varphi}_2  \pare{k, x_2} = & \ - \frac{\cosh\pare{\sqrt{\delta}\av{k}x_2}}{\cosh\pare{\sqrt{\delta}\av{k}}}\hat{g}_2 \pare{k, -1}
	+ \hat{g}_2\pare{k, x_2}\\
	& \  + \int _{-1}^{x_2} {\hat{g}_2\pare{k, y_2}} \frac{1}{\sqrt{\delta}\av{k}} \partial_{x_2}\partial_{y_2} \Pi_1\pare{\sqrt{\delta}\av{k}, y_2, x_2} \textnormal{d} y_2 \\
	& \  + \int _{x_2}^0 {\hat{g}_2 \pare{k, y_2}} \frac{1}{\sqrt{\delta}\av{k}} \partial_{x_2}\partial_{y_2} \Pi_2\pare{\sqrt{\delta}\av{k}, y_2, x_2} \textnormal{d} y_2 , 
	\end{aligned}
	\end{equation}
	\begin{equation}
	\label{eq:nd_ps}
	\begin{aligned}
	\sqrt{\delta} \ \widehat{\partial_1 \varphi_1} \pare{k, x_2}  = & \  
	\sqrt{\delta} i \av{k} \Bigg[\int_{-1}^{x_2} \frac{1}{\pare{\sqrt{\delta}\av{k}}^2} \  \partial_{y_2}^2  \Pi_1 \pare{\sqrt{\delta}\av{k}, y_2, x_2} \hat{g}_1 \pare{k, y_2}  \dy_2\\
	& \ \qquad\qquad +  \int_{x_2}^{0} \frac{1}{\pare{\sqrt{\delta}\av{k}}^2} \  \partial_{y_2}^2 \Pi_2 \pare{\sqrt{\delta}\av{k}, y_2, x_2} \hat{g}_1 \pare{k, y_2}  \dy_2\Bigg] ,  \\
	\sqrt{\delta} \ \widehat{\partial_1 \varphi_2}\pare{k, x_2 } = & \ i\sgn\pare{k} \Bigg[ -\frac{\sinh\pare{\sqrt{\delta}\av{k}x_2}}{\cosh\pare{\sqrt{\delta}\av{k}}}\hat{g}_2 \pare{k, -1}
	\\
	& \qquad\qquad  + \int _{-1}^{x_2} {\hat{g}_2\pare{k, y_2}}  \partial_{y_2} \Pi_1\pare{\sqrt{\delta}\av{k}, y_2, x_2} \textnormal{d} y_2\\
	& \qquad\qquad  + \int _{x_2}^0 {\hat{g}_2 \pare{k, y_2}}  \partial_{y_2} \Pi_2\pare{\sqrt{\delta}\av{k}, y_2, x_2} \textnormal{d} y_2 \Bigg] .
	\end{aligned}
	\end{equation}
	We see now that every integral term in \eqref{eq:nd_psi} and \eqref{eq:nd_ps} presents a kernel of the form $ \partial_{y_2}^{1+l}\partial_{x_2}^i \Pi_j, \ j=1, 2, \ l, i=0, 1  $. We can then integrate by parts in order to commute the operator $ \partial_{y_2} $ onto $ {\bf g} $. Following this integration by parts and using the relations (see \eqref{eq:Pi_derivatives})
	\begin{align*}
	\partial_{y_2}\partial_{x_2}\Pi_1 \pare{\sqrt{\delta}\av{k}, x_2, x_2 } - \partial_{y_2}\partial_{x_2}\Pi_2 \pare{\sqrt{\delta}\av{k}, x_2, x_2 } & =0, \\
	\partial_{x_2}\Pi_1 \pare{\sqrt{\delta}\av{k}, x_2, x_2 } - \partial_{x_2}\Pi_2 \pare{\sqrt{\delta}\av{k}, x_2, x_2 } & =\sqrt{\delta}\av{k},\\
	\frac{1}{\pare{\sqrt{\delta}\av{k}}^2} \partial_{y_2}\partial_{x_2}\Pi_2 \pare{\sqrt{\delta}\av{k}, 0, x_2 } & = \frac{\sinh \pare{ \sqrt{\delta}\av{k}\pare{1+x_2} }}{\cosh\pare{\sqrt{\delta}\av{k}}} ,\\
	\frac{1}{\pare{\sqrt{\delta}\av{k}}^2} \partial_{y_2}\partial_{x_2}\Pi_1 \pare{\sqrt{\delta}\av{k}, -1 , x_2 } & = 0 ,\\
	\frac{1}{\sqrt{\delta}\av{k}} \partial_{x_2}\Pi_2 \pare{\sqrt{\delta}\av{k}, 0, x_2 } & =0,\\
	\frac{1}{\sqrt{\delta}\av{k}} \partial_{x_2}\Pi_1 \pare{\sqrt{\delta}\av{k}, -1 , x_2 } & = \frac{\cosh\pare{\sqrt{\delta}\av{k}x_2}}{\cosh\pare{\sqrt{\delta}\av{k}}}.
	\end{align*}
	we arrive at
	\begin{equation}
	\label{eq:pa2_psi_2}
	\begin{small}
	\begin{aligned}
	\partial_2 \hat{\varphi}_1  \pare{k, x_2} = &  \  i\sgn \pare{k}  \left[ \int_{-1}^{x_2} \frac{1}{\pare{\sqrt{\delta}\av{k}}^2} \  \partial_{y_2} \partial_{x_2} \Pi_1 \pare{\sqrt{\delta}\av{k}, y_2, x_2} \partial_{y_2}\hat{g}_1 \pare{k, y_2}  \dy_2 \right. \\
	&  \left. +  \int_{x_2}^{0} \frac{1}{\pare{\sqrt{\delta}\av{k}}^2} \  \partial_{y_2} \partial_{x_2} \Pi_2 \pare{\sqrt{\delta}\av{k}, y_2, x_2} \partial_{y_2} \hat{g}_1 \pare{k, y_2}  \dy_2
	+ \frac{\sinh \pare{\sqrt{\delta}\av{k}\pare{1+x_2}}}{\cosh\pare{\sqrt{\delta}\av{k}}} \hat{g}_1 \pare{k, 0}
	\right] , \\
	\partial_2 \hat{\varphi}_2  \pare{k, x_2} = & \ - 2 \pare{ \frac{\cosh\pare{\sqrt{\delta}\av{k}x_2}}{\cosh\pare{\sqrt{\delta}\av{k}}}\hat{g}_2 \pare{k, -1}
		+  \hat{g}_2\pare{k, x_2} }\\
	& \  + \int _{-1}^{x_2} \partial_{y_2}\hat{g}_2\pare{k, y_2} \frac{1}{\sqrt{\delta}\av{k}} \partial_{x_2} \Pi_1\pare{\sqrt{\delta}\av{k}, y_2, x_2} \textnormal{d} y_2 \\
	& \  + \int _{x_2}^0 \partial_{y_2}\hat{g}_2 \pare{k, y_2} \frac{1}{\sqrt{\delta}\av{k}} \partial_{x_2} \Pi_2\pare{\sqrt{\delta}\av{k}, y_2, x_2} \textnormal{d} y_2 , 
	\end{aligned}
	\end{small}
	\end{equation}

	Using the formulas (see \eqref{eq:Pi_j} and \eqref{eq:Pi_derivatives})
	\begin{align*}
	\partial_{y_2} \Pi_1 \pare{\sqrt{\delta}\av{k}, x_2, x_2} - \partial_{y_2} \Pi_2 \pare{\sqrt{\delta}\av{k}, x_2, x_2} & = - \sqrt{\delta}\av{k}, \\
	\partial_{y_2} \Pi_1 \pare{\sqrt{\delta}\av{k}, x_2, x_2} - \partial_{y_2} \Pi_2 \pare{\sqrt{\delta}\av{k}, x_2, x_2} & = 0,  \\
	\partial_{y_2} \Pi_1  \pare{\sqrt{\delta}\av{k}, - 1 , x_2} & =0,  \\
	\partial_{y_2} \Pi_2  \pare{\sqrt{\delta}\av{k}, 0 , x_2} & = \sqrt{\delta}\av{k} \frac{\cosh\pare{\sqrt{\delta}\av{k}\pare{1+x_2}}}{\cosh\pare{\sqrt{\delta}\av{k}}},  \\
	\Pi_1  \pare{\sqrt{\delta}\av{k}, - 1 , x_2} & =\frac{\sinh\pare{\sqrt{\delta}\av{k}x_2 }}{\cosh\pare{\sqrt{\delta}\av{k}}},  \\
	\Pi_2  \pare{\sqrt{\delta}\av{k}, 0 , x_2} & =0.
	\end{align*}
	and performing analogous computations on $ \partial_1 \varphi_j, \ j=1, 2 $,  we derive
		\begin{equation}
	\label{eq:pa1_psi_2}
	\begin{aligned}
	\sqrt{\delta} \ \widehat{\partial_1 \varphi_1} \pare{k, x_2}  = & \  
	i  \Bigg[
	\frac{1}{\sqrt{\delta}\av{k}} \ \frac{\cosh\pare{\sqrt{\delta}\av{k}\pare{1+x_2}}}{\cosh\pare{\sqrt{\delta}\av{k}}} \ \hat{g}_1 \pare{k, 0} -   \hat{g}_1\pare{k, x_2}
	\\
	& \ \qquad\qquad +\int_{-1}^{x_2} \frac{1}{\sqrt{\delta}\av{k}} \  \partial_{y_2}  \Pi_1 \pare{\sqrt{\delta}\av{k}, y_2, x_2} \partial_{y_2}\hat{g}_1 \pare{k, y_2}  \dy_2\\
	& \ \qquad\qquad +  \int_{x_2}^{0} \frac{1}{\sqrt{\delta}\av{k}} \  \partial_{y_2} \Pi_2 \pare{\sqrt{\delta}\av{k}, y_2, x_2} \partial_{y_2} \hat{g}_1 \pare{k, y_2}  \dy_2\Bigg] ,  \\
	\sqrt{\delta} \ \widehat{\partial_1 \varphi_2}\pare{k, x_2 } = & \ i\sgn\pare{k} \Bigg[ -2 \frac{\sinh\pare{\sqrt{\delta}\av{k}x_2}}{\cosh\pare{\sqrt{\delta}\av{k}}}\hat{g}_2 \pare{k, -1}
	\\
	& \qquad\qquad  + \int _{-1}^{x_2} \partial_{y_2}\hat{g}_2\pare{k, y_2}   \Pi_1\pare{\sqrt{\delta}\av{k}, y_2, x_2} \textnormal{d} y_2\\
	& \qquad\qquad  + \int _{x_2}^0 \partial_{y_2}\hat{g}_2 \pare{k, y_2}   \Pi_2\pare{\sqrt{\delta}\av{k}, y_2, x_2} \textnormal{d} y_2 \Bigg] , 
	\end{aligned}
	\end{equation}
	
	Up to now we hence rewrote $ \nd \varphi $ in an appropriate form. We now differentiate in $ x_2 $ \eqref{eq:pa2_psi_2} and \eqref{eq:pa1_psi_2} and use the identities (see \eqref{eq:Pi_derivatives})
	\begin{align*}
	\partial_{y_2} \partial_{x_2} \Pi_1 \pare{\sqrt{\delta}\av{k}, x_2, x_2} - \partial_{y_2} \partial_{x_2} \Pi_2 \pare{\sqrt{\delta}\av{k}, x_2, x_2}   & =0, \\
	\partial_{x_2} \Pi_1 \pare{\sqrt{\delta}\av{k}, x_2, x_2} - \partial_{x_2} \Pi_2 \pare{\sqrt{\delta}\av{k}, x_2, x_2}   & = \sqrt{\delta}\av{k}. 
	\end{align*} to obtain that	
	\begin{equation}
	\label{eq:pa22_psi_1}
	\begin{small}
	\begin{aligned}
	\partial_2^2 \hat{\varphi}_1  \pare{k, x_2} = &  \  i\sgn \pare{k}  \left[ \int_{-1}^{x_2} \frac{1}{\pare{\sqrt{\delta}\av{k}}^2} \  \partial_{y_2} \partial_{x_2}^2 \Pi_1 \pare{\sqrt{\delta}\av{k}, y_2, x_2} \partial_{y_2}\hat{g}_1 \pare{k, y_2}  \dy_2 \right. \\
	&  \qquad \qquad +  \int_{x_2}^{0} \frac{1}{\pare{\sqrt{\delta}\av{k}}^2} \  \partial_{y_2} \partial_{x_2}^2 \Pi_2 \pare{\sqrt{\delta}\av{k}, y_2, x_2} \partial_{y_2} \hat{g}_1 \pare{k, y_2}  \dy_2  \\
	& \qquad \qquad  \left.
	+ \sqrt{\delta}\av{k} \ \frac{\cosh \pare{\sqrt{\delta}\av{k}\pare{1+x_2}}}{\cosh\pare{\sqrt{\delta}\av{k}}} \hat{g}_1 \pare{k, 0}
	\right] , \\
	\partial_2^2 \hat{\varphi}_2  \pare{k, x_2} = & \ - 2  \sqrt{\delta}\av{k}\frac{\sinh\pare{\sqrt{\delta}\av{k}x_2}}{\cosh\pare{\sqrt{\delta}\av{k}}}\hat{g}_2 \pare{k, -1}
	-  \partial_2 \hat{g}_2\pare{k, x_2}\\
	& \  + \int _{-1}^{x_2} \partial_{y_2}\hat{g}_2\pare{k, y_2} \frac{1}{\sqrt{\delta}\av{k}} \partial_{x_2}^2 \Pi_1\pare{\sqrt{\delta}\av{k}, y_2, x_2} \textnormal{d} y_2 \\
	& \  + \int _{x_2}^0 \partial_{y_2}\hat{g}_2 \pare{k, y_2} \frac{1}{\sqrt{\delta}\av{k}} \partial_{x_2}^2 \Pi_2\pare{\sqrt{\delta}\av{k}, y_2, x_2} \textnormal{d} y_2 , 
	\end{aligned}
	\end{small}
	\end{equation}

Due to \begin{align*}
	\partial_{y_2}\Pi_1 \pare{\sqrt{\delta}\av{k}, x_2, x_2} - \partial_{y_2}\Pi_2 \pare{\sqrt{\delta}\av{k}, x_2, x_2} & = -\sqrt{\delta}\av{k}, \\
	\Pi_1 \pare{\sqrt{\delta}\av{k}, x_2, x_2} - \Pi_2 \pare{\sqrt{\delta}\av{k}, x_2, x_2} & = 0,
	\end{align*}
	after differentiating \eqref{eq:pa1_psi_2} in $ x_2 $, we find that
	
	\begin{equation}
	\label{eq:pa21_psi}
	\begin{aligned}
	\sqrt{\delta} \ \widehat{\partial_2 \partial_1 \varphi_1} \pare{k, x_2}  = & \  
	i  \Bigg[
	\frac{\sinh\pare{\sqrt{\delta}\av{k}\pare{1+x_2}}}{\cosh\pare{\sqrt{\delta}\av{k}}} \ \hat{g}_1 \pare{k, 0} -  2 \partial_2 \hat{g}_1\pare{k, x_2}
	\\
	& \ \qquad\qquad +\int_{-1}^{x_2} \frac{1}{\sqrt{\delta}\av{k}} \  \partial_{x_2} \partial_{y_2}  \Pi_1 \pare{\sqrt{\delta}\av{k}, y_2, x_2} \partial_{y_2}\hat{g}_1 \pare{k, y_2}  \dy_2\\
	& \ \qquad\qquad +  \int_{x_2}^{0} \frac{1}{\sqrt{\delta}\av{k}} \  \partial_{x_2}\partial_{y_2} \Pi_2 \pare{\sqrt{\delta}\av{k}, y_2, x_2} \partial_{y_2} \hat{g}_1 \pare{k, y_2}  \dy_2\Bigg] ,  \\
	\sqrt{\delta} \ \widehat{\partial_2 \partial_1 \varphi_2}\pare{k, x_2 } = & \ i\sgn\pare{k} \Bigg[ -2 \sqrt{\delta}\av{k} \frac{\cosh \pare{\sqrt{\delta}\av{k}x_2}}{\cosh\pare{\sqrt{\delta}\av{k}}}\hat{g}_2 \pare{k, -1}
	\\
	& \qquad\qquad  + \int _{-1}^{x_2} \partial_{y_2}\hat{g}_2\pare{k, y_2}  \partial_{x_2} \Pi_1\pare{\sqrt{\delta}\av{k}, y_2, x_2} \textnormal{d} y_2\\
	& \qquad\qquad  + \int _{x_2}^0 \partial_{y_2}\hat{g}_2 \pare{k, y_2}  \partial_{x_2} \Pi_2\pare{\sqrt{\delta}\av{k}, y_2, x_2} \textnormal{d} y_2 \Bigg] , 
	\end{aligned}
	\end{equation}

	\subsubsection*{Step 4: Elliptic estimates}
	We have now a comprehensive description of $ \partial_2 \nd \varphi $ given by the identities \eqref{eq:pa22_psi_1} and \eqref{eq:pa21_psi}. We now take the absolute value of \eqref{eq:pa22_psi_1} and \eqref{eq:pa12_psi} and apply the operator $ \int _{-1}^{0} \cdot \  \dx_2 $. This gives

	\begin{equation}
	\label{eq:pa22_psi_2}
	\begin{small}
	\begin{aligned}
	\int _{-1}^{0}\av{\partial_2^2 \hat{\varphi}_1  \pare{k, x_2}} \dx_2  \leqslant &  \ \int _{-1}^{0}  \left\lbrace    \frac{1}{\pare{\sqrt{\delta}\av{k}}^2} \  \norm{ \mathbf{1} _{\bra{-1, x_2}}\pare{\cdot}\partial_{y_2} \partial_{x_2}^2 \Pi_1 \pare{\sqrt{\delta}\av{k}, \cdot , x_2} }_{L^\infty_{y_2}}\norm{\partial_{y_2}\hat{g}_1 \pare{k, \cdot }}_{L^1_{y_2}}  \right.  \\
	&  \qquad \qquad +   \frac{1}{\pare{\sqrt{\delta}\av{k}}^2} \  \norm{ \mathbf{1} _{\bra{ x_2, 0}}\pare{\cdot} \partial_{y_2} \partial_{x_2}^2 \Pi_2 \pare{\sqrt{\delta}\av{k}, \cdot , x_2} }_{L^\infty_{y_2}}\norm{\partial_{y_2}\hat{g}_1 \pare{k, \cdot }}_{L^1_{y_2}}    \\
	& \qquad \qquad  \left.
	+ \sqrt{\delta}\av{k} \ \frac{\cosh \pare{\sqrt{\delta}\av{k}\pare{1+x_2}}}{\cosh\pare{\sqrt{\delta}\av{k}}} \av{\hat{g}_1 \pare{k, 0}} 
	\right\rbrace \dx_2 , \\
	\int _{-1}^{0} \av{\partial_2^2 \hat{\varphi}_2  \pare{k, x_2}} \dx_2 \leqslant & \ \int _{-1}^{0}  \left\lbrace  - 2  \sqrt{\delta}\av{k}\frac{\sinh\pare{\sqrt{\delta}\av{k}x_2}}{\cosh\pare{\sqrt{\delta}\av{k}}}\av{\hat{g}_2 \pare{k, -1}}
	+ \av{ \partial_2 \hat{g}_2\pare{k, x_2} }\right. \\
	& \ \qquad\qquad +   \frac{1}{\sqrt{\delta}\av{k}} \norm{ \mathbf{1} _{\bra{-1, x_2}}\pare{\cdot}\partial_{x_2}^2 \Pi_1\pare{\sqrt{\delta}\av{k}, \cdot , x_2}}_{L^\infty_{y_2}} \norm{\partial_{y_2}\hat{g}_2\pare{k, \cdot }}_{L^1_{y_2}}  \\
	& \ \qquad\qquad +  \left. \frac{1}{\sqrt{\delta}\av{k}} \norm{ \mathbf{1} _{\bra{ x_2, 0}}\pare{\cdot} \partial_{x_2}^2 \Pi_2\pare{\sqrt{\delta}\av{k}, \cdot , x_2}}_{L^\infty_{y_2}} \norm{\partial_{y_2}\hat{g}_2 \pare{k, \cdot}}_{L^1_{y_2}}  \right\rbrace \dx_2 .  
	\end{aligned}
	\end{small}
	\end{equation}
We use now the following estimates which holds for any $ \pare{j, l}\in\bN^2 $
	\begin{equation}
	\label{eq:bounds_integral_Pi_j}
	\begin{aligned}
	\int _{-1}^{0} \norm{ \mathbf{1} _{\bra{-1, x_2}}\pare{\cdot} \partial_{x_2}^j \partial_{y_2}^l \Pi_1 \pare{\sqrt{\delta}\av{k}, \cdot , x_2} }_{L^\infty_{y_2}} \dx_2 & \leqslant  2 \pare{\sqrt{\delta\av{k}}}^{j+l-1}, \\
	\int _{-1}^{0} \norm{ \mathbf{1} _{\bra{ x_2, 0}}\pare{\cdot}  \partial_{x_2}^j \partial_{y_2}^l \Pi_2 \pare{\sqrt{\delta}\av{k}, \cdot , x_2} }_{L^\infty_{y_2}} \dx_2  & \leqslant  \frac{5}{2} \pare{\sqrt{\delta\av{k}}}^{j+l-1}. 
	\end{aligned}
	\end{equation}
Indeed, from \eqref{eq:Pi1_as_negative_exp} and \eqref{eq:Pi2_as_negative_exp} we deduce that for any $ \pare{ j, l }\in \bN^2 $ and $y_2 \in \pare{-1, x_2}$ we have that
	\begin{equation}
	\label{eq:Pi_Derivatives}
	\begin{aligned}
	& 2\ \partial_{x_2}^{j} \partial_{y_2}^{l} \Pi_1 \pare{\sqrt{\delta}\av{k}, y_2, x_2 } \\
	& = \pare{\sqrt{\delta}\av{k}}^{j+l}
	\frac{
		e^{\sqrt{\delta}\av{k}\pare{y_2 + x_2}} - \pare{-1}^j e^{\sqrt{\delta}\av{k}\pare{y_2-x_2}} + \pare{-1}^l e^{\sqrt{\delta}\av{k}\pare{-2-y_2+x_2}} + \pare{-1}^{j+l}e^{\sqrt{\delta}\av{k}\pare{-2-y_2-x_2}}
	}{1-e^{-2\sqrt{\delta}\av{k}}}, \\
	\end{aligned}
	\end{equation}
	Similarly, if $	y_2 \in \pare{x_2, 0}$ then
		\begin{equation}
	\label{eq:Pi_Derivatives2}
	\begin{aligned}
	& 2\ \partial_{x_2}^{j} \partial_{y_2}^{l} \Pi_2 \pare{\sqrt{\delta}\av{k}, y_2, x_2 } \\
	& = \pare{\sqrt{\delta}\av{k}}^{j+l}
	\frac{
		\pare{-1}^j e^{\sqrt{\delta}\av{k}\pare{-2 + y_2 -x_2}} 
		-\pare{-1}^{j+l}e^{\sqrt{\delta}\av{k}\pare{-2-y_2-x_2}}
		+ e^{\sqrt{\delta}\av{k}\pare{x_2 +y_2}}  
		+ \pare{-1}^{l} e^{\sqrt{\delta}\av{k}\pare{-2-y_2+x_2}} 
	}{1+e^{-2\sqrt{\delta}\av{k}}}\\
	& \qquad  - \pare{\sqrt{\delta}\av{k}}^{j+l}\pare{-1}^{l} e^{-\sqrt{\delta}\av{k}\pare{y_2 - x_2 }}.
 	\end{aligned}
	\end{equation}

Recall that \textit{any} argument of every exponential appearing here above is negative, so that every exponential is bounded by one uniformly in $ y_2 $ and $ k $. We now show how the integration of the term $ e^{\sqrt{\delta}\av{k}\pare{y_2 - x_2}} $ appearing in the expression of $ \partial_{x_2}^j \partial_{y_2}^l \Pi_1 $ has to be performed, being the other terms similar. Indeed since we have to respect the constraint $ y_2 \in \pare{-1, x_2} $ we have to compute
	\begin{equation*}
	\int _{y_2 }^{0} e^{\sqrt{\delta}\av{k}\pare{y_2 - x_2}} \dx_2 = -\frac{1}{\sqrt{\delta}\av{k}} e^{\sqrt{\delta}\av{k} y_2} \pare{1- e^{- \sqrt{\delta}\av{k} y_2} } \leqslant \frac{1}{\sqrt{\delta}\av{k}}. 
	\end{equation*}

	We can integrate \eqref{eq:Pi_Derivatives} and \eqref{eq:Pi_Derivatives2} in $ x_2 $ as done above from which we obtain

	\begin{align*}
	2\ \int _{y_2}^{0} \av{\partial_{x_2}^{j} \partial_{y_2}^{l} \Pi_1 \pare{\sqrt{\delta}\av{k}, y_2, x_2 }} \dx_2  &
	\leqslant 4 \pare{\sqrt{\delta}\av{k}}^{j+l-1}, &&
	y_2 \in \pare{-1, x_2}, \\
	2\ \int _{-1}^{y_2} \av{\partial_{x_2}^{j} \partial_{y_2}^{l} \Pi_2 \pare{\sqrt{\delta}\av{k}, y_2, x_2 }} \dx_2  & \leqslant 5 \pare{\sqrt{\delta}\av{k}}^{j+l-1}, &&
	y_2 \in \pare{x_2, 0}. 
	\end{align*}
	Since the bound provided are uniform in $ y_2 $ we concluded the proof of \eqref{eq:bounds_integral_Pi_j} 

	We can use the bounds \eqref{eq:bounds_integral_Pi_j} in \eqref{eq:pa22_psi_2}. Recalling that, by the definition of the space $ \cA^{s, 1} $, the hypothesis $ {\bf g}\in \cA^{s, 1} $ implies that $ \left. {\bf g}\right|_{x_2 =-1} \equiv 0 $, we find that
	$$
	 - 2  \sqrt{\delta}\av{k}\frac{\sinh\pare{\sqrt{\delta}\av{k}x_2}}{\cosh\pare{\sqrt{\delta}\av{k}}}\av{\hat{g}_2 \pare{k, -1}} =0 \ \forall k,
	 $$
	  thus we obtain
	\begin{equation}
	\label{eq:BOUND1}   
	\begin{aligned}
	\int _{-1}^{0}\av{\partial_2^2 \hat{\varphi}_1  \pare{k, x_2}} \dx_2  & \leqslant \frac{9}{2} \int _{-1}^{0} \av{\partial_2 \hat{g}_1 \pare{k, y_2}} \dy_2 + \av{\hat{g}_1\pare{k, 0}}, \\
	& \leqslant \frac{11}{2} \int _{-1}^{0} \av{\partial_2 \hat{g}_1 \pare{k, y_2}} \dy_2, \\
	\int _{-1}^{0}\av{\partial_2^2 \hat{\varphi}_2  \pare{k, x_2}} \dx_2  & \leqslant \frac{11}{2} \int _{-1}^{0} \av{\partial_2 \hat{g}_2 \pare{k, y_2}} \dy_2 . 
	\end{aligned}
	\end{equation}
	
	We can argue similarly as above in order to obtain that  
	\begin{equation}
	\label{eq:pa12_psi}
	\begin{small}
	\begin{aligned}
	\sqrt{\delta} \int _{-1}^{0} \av{\widehat{\partial_2 \partial_1 \varphi_1} \pare{k, x_2}} \dx_2   \leqslant & \  
	\int _{-1}^{0} \Bigg\lbrace 
	\frac{\sinh\pare{\sqrt{\delta}\av{k}\pare{1+x_2}}}{\cosh\pare{\sqrt{\delta}\av{k}}} \ \av{\hat{g}_1 \pare{k, 0}} -  2 \av{\partial_2 \hat{g}_1\pare{k, x_2}}
	\\
	& \ \qquad\qquad +  \frac{1}{\sqrt{\delta}\av{k}} \  \norm{\mathbf{1} _{\bra{-1,  x_2}}\pare{\cdot} \partial_{x_2} \partial_{y_2}  \Pi_1 \pare{\sqrt{\delta}\av{k}, \cdot , x_2}}_{L^\infty_{y_2}}  \norm{\partial_{y_2}\hat{g}_1 \pare{k, \cdot}}_{L^1_{y_2}} \\
	& \ \qquad\qquad +  \frac{1}{\sqrt{\delta}\av{k}} \  \norm{ \mathbf{1} _{\bra{ x_2, 0}}\pare{\cdot}  \partial_{x_2} \partial_{y_2}  \Pi_2 \pare{\sqrt{\delta}\av{k}, \cdot , x_2}}_{L^\infty_{y_2}}  \norm{\partial_{y_2} \hat{g}_1 \pare{k, \cdot}}_{L^1_{y_2}} \Bigg\rbrace \dx_2 ,  \\
	\sqrt{\delta} \int _{-1}^{0}\av{ \widehat{\partial_2 \partial_1 \varphi_2}\pare{k, x_2 }} \dx_2  \leqslant & \ \int _{-1}^{0} \Bigg\lbrace -2 \sqrt{\delta}\av{k} \frac{\cosh \pare{\sqrt{\delta}\av{k}x_2}}{\cosh\pare{\sqrt{\delta}\av{k}}}\av{\hat{g}_2 \pare{k, -1}}
	\\
	& \qquad\qquad  + \norm{ \mathbf{1} _{\bra{-1,  x_2}}\pare{\cdot}  \partial_{x_2} \Pi_1\pare{\sqrt{\delta}\av{k}, \cdot , x_2}}_{L^\infty_{y_2}} \norm{\partial_{y_2}\hat{g}_2\pare{k, \cdot }}_{L^1_{y_2}} \\
	& \qquad\qquad  + \norm{ \mathbf{1} _{\bra{ x_2, 0}}\pare{\cdot}  \partial_{x_2} \Pi_2\pare{\sqrt{\delta}\av{k}, \cdot , x_2}}_{L^\infty_{y_2}} \norm{\partial_{y_2}\hat{g}_2\pare{k, \cdot }}_{L^1_{y_2}} \Bigg\rbrace \dx_2  .
	\end{aligned}
	\end{small}
	\end{equation} 
We use \eqref{eq:bounds_integral_Pi_j} and the identity $  -2 \sqrt{\delta}\av{k} \frac{\cosh \pare{\sqrt{\delta}\av{k}x_2}}{\cosh\pare{\sqrt{\delta}\av{k}}}\av{\hat{g}_2 \pare{k, -1}}=0 $ due to the fact that $ {\bf g}\in \cA^{s, 1} $ to obtain the estimate
	\begin{equation}
	\label{eq:BOUND2}
	\begin{aligned}
	\sqrt{\delta} \int _{-1}^{0} \av{\widehat{\partial_2 \partial_1 \varphi_1} \pare{k, x_2}} \dx_2 & \leqslant \frac{1}{\sqrt{\delta}\av{k}} \frac{\pare{ \cosh\pare{\sqrt{\delta}\av{k}}-1 }}{\cosh\pare{\sqrt{\delta}\av{k}}} \av{\hat{g}_1 \pare{k, 0}} + \frac{13}{2} \int _{-1}^{0}\av{\partial_2\hat{g}_1 \pare{k, x_2}} \dx_2, \\
	& \leqslant \frac{15}{2} \int _{-1}^{0}\av{\partial_2\hat{g}_1 \pare{k, x_2}} \dx_2, \\
	\sqrt{\delta} \int _{-1}^{0} \av{\widehat{\partial_2 \partial_1 \varphi_2} \pare{k, x_2}} \dx_2 &\leqslant  \frac{9}{2} \int _{-1}^{0}\av{\partial_2 \hat{g}_2 \pare{k, x_2}} \dx_2.
	\end{aligned}
	\end{equation}
	We can now combine the results in \eqref{eq:BOUND1} and \eqref{eq:BOUND2} to finally obtain the desired bound 
	\begin{equation*}
	\norm{\nd \varphi}_{\cA^{s, 1}_\lambda} \leqslant 13 \norm{{\bf g}}_{\cA^{s, 1}_\lambda}. 
	\end{equation*}

	\hfill $ \Box $

	\section{A priori estimates} \label{sec:apriori_estimates}

	In this section we assume $ h, \phi_1, \phi_2 $ to be  space-time smooth solutions of \eqref{eq:Muskat5} and we want to obtain appropriate a priori estimates. Furthermore, we assume that the zero mean initial data $h_0$ is small enough with respect to the physical parameters in the problem as stated when the size is measured in appropriate norms.
	

	\subsection{Elliptic estimates for the ALE potential} \label{sec:elliptic_estimates}
	
Recalling the definition of the matrix $ I_\delta $ and the anisotropic differential operators $\nd V,\div_\delta V$ and $\Delta_\delta V,$ and defining the matrix
	\begin{equation} \label{eq:def_Q}
	Q\pare{\nabla \sigma} = \pare{
		\begin{array}{cc}
		\partial_2 \sigma & -\sqrt{\delta} \partial_1 \sigma \\[2mm]
		-\sqrt{\delta} \partial_1 \sigma & \displaystyle \frac{-\partial_2 \sigma + \varepsilon\delta \av{\partial_1 \sigma}^2}{1+\varepsilon\partial_2 \sigma}
		\end{array}
	}
	\end{equation}
	the anisotropic Laplace equation in \eqref{eq:Muskat4} can alternatively be expressed as (cf. \cite[Section 2.2.3]{Lannes2013}) 
	\begin{equation*}
	\nd \cdot \pare{ I+\varepsilon Q\pare{\nabla \sigma}} \nd \phi =0. 
	\end{equation*}
	
	With such notation and using the decomposition $ \phi= \phi_1+\phi_2 $ the static equation \eqref{eq:phi2} becomes

	\begin{equation}
	\label{eq:phi2_3}
	\left \lbrace 
	\begin{aligned}
	& \Delta_\delta \phi_2 =- \varepsilon \nd \cdot \pare{Q\pare{\nabla \sigma}  \nd \pare{\phi_1 + \phi_2} } && \text{ in }\cS,  \\
	& \phi_2  =  0 ,&& \text{ on }\Gamma , \\
	& \partial_2\phi_2  =0 , && \text{ on }\Gamma_{\textnormal{b}} .  
	\end{aligned}
	\right. 
	\end{equation}

	Considering the decomposition $ \phi= \phi_1 + \phi_2 $ it is clear that we must provide some elliptic bound for each $ \phi_j, \ j=1, 2 $. The rest of the present section is accordingly divided in two subsections which address the problem of providing elliptic estimates for $ \phi_j, \ j=1, 2 $.

	\subsubsection{Elliptic estimates for $ \phi_1 $}

	\begin{lemma} \label{lem:elliptic_estimate_via_symbol}
		Let $ s, \lambda \geqslant 0 , \ \delta > 0 \  f: \bR \to \bR $ be a $ \cC^j \pare{\bR}\cap W^{j,  1 }_{\loc}\pare{\bR}, \ j\in \bN, \ j\geqslant 1 $ and $ u\in \mathbb{A}^{s+j-1}_\lambda $. Let $ F $ be the primitive of $ \av{f^{\pare{j}}} $, 
		and suppose there exists  $ C_f $ a strictly positive constant depending on $ f $ only such that
		\begin{align*}
		F\pare{0} - F\pare{-\sqrt{\delta} \av{n}} \leqslant C_f, 
		\end{align*}
		uniformly in $ \delta > 0 $ and $ n\in\bZ $. 
		Then  one has
		\begin{equation*}
		\norm{ f\pare{ \cdot \sqrt{\delta}  \Lambda } u }_{\cA^{s, j}_\lambda} \leqslant C_f \delta^{\frac{j-1}{2}} \av{ u}_{s+j- 1, \lambda} .
		\end{equation*}
	\end{lemma}
	
	\begin{proof}
		The proof is straightforward, let us compute
		\begin{align*}
		\int _{-1}^{0} \av{\partial_y^j \pare{f\pare{ y \sqrt{\delta} \av{n}}\hat{u}\pare{n}}} \dy & \leqslant \pare{ \sqrt{\delta} \av{n} }^j  \av{\hat{u}\pare{n}} \int _{-1}^{0} \av{f^{\pare{j}} \pare{ y \sqrt{\delta} \av{n}}}\dy, \\
		& = \pare{ \sqrt{\delta} \av{n} }^j  \av{\hat{u}\pare{n}}  \ \frac{F\pare{0} - F\pare{-\sqrt{\delta} \av{n}}}{ \sqrt{\delta} \av{n} }, \\
		& \leqslant C_f \pare{ \sqrt{\delta} \av{n} }^{j -1} \av{\hat{u}\pare{n}} , 
		\end{align*}
		a summation in $ n $ concludes the proof. 
	\end{proof}

	Our aim is to use Lemma \ref{lem:elliptic_estimate_via_symbol} on $ \phi_1 $ in order to derive suitable optimal elliptic estimates in the particular functional framework in which we are working. Let us at first denote with $ \psi $ the trace of $ \phi_1 $ onto $ \Gamma $, i.e.
	\begin{equation*}
	\psi = \nu\alpha \ccK_{\alpha}+ \varepsilon h . 
	\end{equation*}
	We use the explicit form of $ \phi_1 $ provided in \eqref{eq:phi1}, 
	\begin{equation*}
	\begin{aligned}
	\phi_1 \pare{x_1, x_2, t} & = \frac{e^{-\pare{1+x_2}\sqrt{\delta}\Lambda} + e^{\pare{1+x_2}\sqrt{\delta}\Lambda}}{e^{\sqrt{\delta}\Lambda} + e^{-\sqrt{\delta}\Lambda}} \ \psi \pare{x_1, t}, 
	\end{aligned}
	\end{equation*}
	%
Applying the operator $ \partial_2^j, \ j\in \bN $ to $ \phi_1 $ we infer that when $ x_2 \in \bra{-1, 0} $
	\begin{equation}\label{eq:pa2j_phi1}
	\partial_2^j \phi_1 \pare{x_1, x_2, t} = \delta^{j/2} \Lambda^j \  \frac{  e^{\pare{1+x_2}\sqrt{\delta}\Lambda} + \pare{-1}^j e^{-\pare{1+x_2}\sqrt{\delta}\Lambda}}{e^{\sqrt{\delta}\Lambda} + e^{-\sqrt{\delta}\Lambda}} 	\  \psi \pare{x_1, t} . 
	\end{equation}
	A computation shows that
	\begin{align*}
	\int_{-1}^{0} \frac{  e^{\pare{1+x_2}\sqrt{\delta}\av{n}} + \pare{-1}^j e^{-\pare{1+x_2}\sqrt{\delta}\av{n}}}{e^{\sqrt{\delta}\av{n}} + e^{-\sqrt{\delta}\av{n}}} \dx_2 & = \frac{1}{\sqrt{\delta}\av{n}} \frac{  e^{\sqrt{\delta}\av{n}} - 1 + \pare{-1}^{j+1} \pare{ e^{-\sqrt{\delta}\av{n}}-1 }}{e^{\sqrt{\delta}\av{n}} + e^{-\sqrt{\delta}\av{n}}}, \\
	& \leqslant \frac{1}{\sqrt{\delta}\av{n}} \pare{1+\frac{1}{\cosh\pare{\sqrt{\delta}\av{n}}}}  \leqslant \frac{2}{\sqrt{\delta}\av{n}}.
	\end{align*}
	We can apply Lemma \ref{lem:elliptic_estimate_via_symbol} in order to deduce that, for any $ s \geqslant 0 $, 
	\begin{equation}
	\label{eq:elliptic_estimate_phi1}
	\begin{aligned}
	\norm{ \phi_1 }_{ \cA^{s, j}_\lambda } & \leqslant2 \delta^{\frac{j-1}{2}}  \av{\psi}_{s+j-1, \lambda}. 
	\end{aligned}
	\end{equation}

	The computations performed above lead naturally to the next lemma:
	
	\begin{lemma}
		\label{lem:elliptic_estimates_phi1}
		Let $ s_0\geqslant 0 , \ j \in\bN$, such that $ j\geqslant 1 $. Let $ \ h \in \mathbb{A}^{s_0+j+1}_\lambda $ and let $ \phi_1 $ be a solution of \eqref{eq:phi1} in the space $ \cA^{s_0, j}_\lambda $, then
		\begin{equation}\label{eq:est_phi1}
		\norm{ \phi_1}_{\cA^{s_0, j}_\lambda}\leqslant 2 \delta^{\frac{j-1}{2}} \cK_{s_0+j-1} \bra{ \nu\alpha  \pare{1+ 2 \alpha \av{h}_{1, \lambda}} \av{h}_{s_0+j+1, \lambda}   + \varepsilon \av{h}_{s_0+j-1, \lambda}} .
		\end{equation}
		Similarly, if $ s, \lambda \geqslant 0$ and $ h  \in \mathbb{A}^{s+3}_\lambda $, the following estimates hold true
		\begin{equation}
		\label{eq:bound_pa12phi1}
		\begin{aligned}
		\av{\partial_1 \phi_1 }_{s, \lambda} & \leqslant 2\cK_{s+1} \bra{\nu\alpha \av{h}_{s+3, \lambda} \pare{1+ 2 \alpha \av{h}_{1, \lambda}}   + \varepsilon \av{h}_{s+1, \lambda}\Big. }, \\
		\av{\partial_2 \phi_1 }_{s, \lambda} & \leqslant 2\cK_{s+1}\delta^{1/2}\bra{ \nu\alpha \av{h}_{s+3, \lambda} \pare{1+ 2 \alpha \av{h}_{1, \lambda}}   + \varepsilon \av{h}_{s+1, \lambda}\Big.  } .
		\end{aligned}
		\end{equation}

	\end{lemma}
	
	\begin{proof}
We apply the estimate \eqref{eq:elliptic_estimate_phi1}  
		\begin{align*}
		\norm{\partial_2^j \phi_1}_{\cA^{s_0, 0}_\lambda}  & \leqslant 2 \delta^{\frac{j-1}{2}} \av{\nu\alpha \ccK_{\alpha} + \varepsilon h}_{s_0+j-1, \lambda} , \\
		& \leqslant 2 \delta^{\frac{j-1}{2}} \pare{\nu\alpha\av{\ccK_{\alpha}}_{s_0+j-1, \lambda } + \varepsilon \av{h}_{s_0+j-1, \lambda}}.
		\end{align*}
		Next we consider that
		\begin{equation*}
		\ccK_{\alpha} = h'' \pare{1+\mathsf{F}\pare{\alpha h'}}, 
		\end{equation*}
		where $ \ccK_{\alpha} $ is defined in \eqref{eq:Kalpha} and $ \mathsf{F} $ in \eqref{eq:def_F}. We can invoke Lemma \ref{lem:interpolation_inequality} and \ref{lem:composition_Wiener_FG} together with \eqref{eq:hom-nonhom}
		\begin{align*}
		\av{\ccK_{\alpha}}_{s_0+j-1, \lambda} & \leqslant \cK_{s_0+j-1} \pare{ \av{h}_{s_0+j+1, \lambda} \pare{1+ \av{\mathsf{F}\pare{\alpha h'}}_{0, \lambda} } + \av{h}_{2, \lambda} \av{\mathsf{F}\pare{\alpha h'}}_{s_0+j-1, \lambda} }, \\
		& \leqslant \cK_{s_0+j-1} \pare{ \av{h}_{s_0+j+1, \lambda} \pare{1+ \alpha \av{h}_{1, \lambda}}  + \alpha \av{h}_{2, \lambda} \av{h}_{s_0+j, \lambda} }. 
		\end{align*}
		The above estimate lead us to
		\begin{align*}
		\norm{ \partial_2^j \phi_1 }_{\cA^{s_0, 0}_\lambda} & \leqslant 2\delta^{\frac{j-1}{2}} \cK_{s_0+j-1} \bra{ \nu\alpha\pare{\av{h}_{s_0+j+1, \lambda} \pare{1+ \alpha \av{h}_{1, \lambda}}  + \alpha \av{h}_{2, \lambda} \av{h}_{s_0+j, \lambda} } + \varepsilon \av{h}_{s_0+j-1, \lambda}},  \\
		& \leqslant 2\delta^{\frac{j-1}{2}}  \cK_{s_0+j-1, \lambda } \bra{ \nu\alpha \av{h}_{s_0+j+1, \lambda} \pare{1+ 2 \alpha \av{h}_{1, \lambda}}   + \varepsilon \av{h}_{s_0+j-1, \lambda}},
		\end{align*}
		where in the second inequality we used \eqref{eq:interpolation_inequality}, concluding the proof. 
To obtain the second part of the statement, we can use the trace estimates of Proposition \ref{prop:trace-embedding} to deduce
		\begin{equation*}
		\av{\partial_i \phi_1 }_{s, \lambda} \leqslant  \norm{\partial_i \phi_1 }_{\cA^{s, 1}_\lambda} = \norm{\partial_2 \partial_i \phi_1}_{\cA^{s, 0}_\lambda}  , 
		\end{equation*}
		we now conclude similarly as before. 
	\end{proof}

Lemma \ref{lem:elliptic_estimates_phi1} provides a suitable bound for the trace of the gradient of $ \phi_1 $ in terms of $ h $. This is the bound that we will use in the parabolic energy estimates for the evolution of $ h $ described by the evolution law \eqref{eq:Muskat5}.

	\subsubsection{Elliptic estimates for $ \phi_2 $}\label{sec:elliptic_phi2}

In this section we need to obtain estimates for the solution of the problem \eqref{eq:phi2_3}. At this point it suffices to apply Theorem \ref{prop:elliptic_estimates} to \eqref{eq:phi2_3}, this gives
	\begin{equation*}
	\norm{\nd \phi_2}_{\cA^{s, 1}_{\lambda}} \leqslant 13 \varepsilon \norm{Q\pare{\nabla \sigma}\nd\pare{\phi_1 + \phi_2}}_{\cA^{s, 1}_{\lambda}}.
	\end{equation*}
This inequality can be combined now with the product rule stated in Lemma \ref{lem:interpolation_inequality_strip} to obtain that 
	\begin{multline}\label{eq:application_elliptic_estimates}
	\norm{\nd \phi_2}_{\cA^{s, 1}_{\lambda}}
	\leqslant 26\varepsilon \ \cK_s \bigg[ \norm{Q\pare{\nabla \sigma}}_{\cA^{0, 1}_{\lambda}} \pare{\norm{\nd \phi_1}_{\cA^{s, 1}_{\lambda}} + \norm{\nd \phi_2}_{\cA^{s, 1}_{\lambda}}} \\
	+ \norm{Q\pare{\nabla \sigma}}_{\cA^{s, 1}_{\lambda}} \pare{\norm{\nd \phi_1}_{\cA^{0, 1}_{\lambda}} + \norm{\nd \phi_2}_{\cA^{0, 1}_{\lambda}}} \bigg]. 
	\end{multline}
Considering the explicit expression of $ Q\pare{\nabla \sigma} $ and $ \phi_1 $ (given in \eqref{eq:def_Q} and \eqref{eq:phi1} respectively) we see that both functions are harmonic extensions of (suitable nonlinear relations involving) the elevation $ h $ on the boundary. We exploit this idea in order to control $ \norm{\nd \phi_2}_{\cA^{s, 1}_{\lambda}} $ in terms of $ h $ only.

	Lemma \ref{lem:elliptic_estimates_phi1} provides already a control for terms such as $ \norm{\nd \phi_1}_{\cA^{s, 1}_{\lambda}}, \ s, \lambda \geqslant 0 $, we need hence to control now
	\begin{equation*}
	\norm{Q\pare{\nabla \sigma}}_{\cA^{s, 1}_{\lambda}}, \qquad s, \lambda \geqslant 0. 
	\end{equation*}
	
	Such result is proved in the next lemma: 
	
	\begin{lemma}\label{lem:bound_Q}
		Let $ Q\pare{\nabla \sigma} $ be as in \eqref{eq:def_Q}, $ 0\leq s\leq 1, \lambda \geqslant 0 $ and $ \av{h}_{1, \lambda} < \pare{4 \cK_s \varepsilon}^{-1} $. The following bound holds true
		\begin{equation}\label{eq:est_Q}
		\norm{Q\pare{\nabla \sigma}}_{\cA^{s, 1}_{\lambda}} \leqslant 10 \av{h}_{s+1, \lambda}.
		\end{equation}
	\end{lemma}
	
	\begin{proof}
		Using the expansion
		\begin{align*}
		\frac{1}{1+x} = \sum_{k\geqslant 0} \pare{-x}^k, && \av{x} < 1, 
		\end{align*}
		we derive that
		\begin{equation*}
		Q\pare{\nabla \sigma} = \pare{
			\begin{array}{cc}
			\partial_2 \sigma & -\sqrt{\delta} \partial_1 \sigma \\[2mm]
			-\sqrt{\delta} \partial_1 \sigma & \displaystyle \frac{-\partial_2 \sigma + \varepsilon\delta \av{\partial_1 \sigma}^2}{1+\varepsilon\partial_2 \sigma}
			\end{array}
		} = \sum_{n\geqslant 0} \varepsilon^n Q_{\pare{n}}\pare{\nabla \sigma}, 
		\end{equation*}
		where
		\begin{align*}
		Q_{\pare{0}} \pare{\nabla \sigma} & = \pare{\begin{array}{cc}
			\partial_2 \sigma & -\sqrt{\delta}\partial_1 \sigma \\[2mm]
			-\sqrt{\delta}\partial_1 \sigma & -\partial_2 \sigma + \varepsilon \delta \av{\partial_1 \sigma}^2
			\end{array}}, \\[3mm]
		Q_{\pare{n}} \pare{\nabla \sigma} & = \pare{\begin{array}{cc}
			0&0 \\
			0 & \pare{-\partial_2 \sigma}^n \pare{-\partial_2 \sigma + \varepsilon\delta \av{\partial_1 \sigma}^2}
			\end{array}}, & n\geqslant 1 . 
		\end{align*}
		We can now use \eqref{eq:sigma} in order to express $ \nabla \sigma $ in terms of $ h $
		\begin{equation}
		\label{eq:nabla_sigma}
		\begin{aligned}
		\partial_1 \sigma\pare{x_1, x_2, t} & =  \frac{ e^{ \pare{1+x_2}\Lambda} - e^{- \pare{1+x_2}\Lambda}  }{ e^{ \Lambda} - e^{- \Lambda}  } \ \partial_1  h\pare{x_1, t} = \frac{\sinh\pare{\pare{1+x_2}\Lambda}}{\sinh\pare{\Lambda}} \ \partial_1 h \pare{x_1, t} , \\
		\partial_2 \sigma\pare{x_1, x_2, t} &=  \frac{ e^{ \pare{1+x_2}\Lambda} + e^{- \pare{1+x_2}\Lambda}  }{ e^{ \Lambda} - e^{- \Lambda}  } \ \Lambda h\pare{x_1, t} = \frac{\cosh \pare{\pare{1+x_2}\Lambda}}{\sinh\pare{\Lambda}} \ \Lambda h\pare{x_1, t}. 
		\end{aligned}
		\end{equation}
		Since $ h $ is a zero-average function, the Fourier multiplier $ \pare{\sinh\pare{\Lambda}}^{-1} $ is bounded in the Fourier side and we can use Lemma \ref{lem:elliptic_estimate_via_symbol}
		\begin{align*}
		\norm{\partial_j \sigma}_{\cA^{s, 1}_{\lambda} } \leqslant \av{h}_{s+1, \lambda }, && j=1, 2, 
		\end{align*}
		which in turn implies, using the product rule of \eqref{eq:product_rule_Wiener_strip}, that
		\begin{align*}
		\norm{Q_{\pare{0}}\pare{\nabla \sigma}}_{\cA^{s, 1}_{\lambda}} & \leqslant 2 \pare{1+\sqrt{\delta} + 2 \varepsilon\delta \ \cK_s \av{h}_{1, \lambda} }\av{h}_{s+1, \lambda}, \\
		\norm{Q_{\pare{n}}\pare{\nabla \sigma}}_{\cA^{s, 1}_{\lambda}} & \leqslant 2^{n+1}\cK_s\pare{ 1 + K_{s, n} } \pare{1+2 \cK_s \varepsilon\delta \av{h}_{1, \lambda }\big. } \av{h}_{1, \lambda}^{n} \av{h}_{s+1, \lambda}, 
		\end{align*}
		where $ K_{s, n} $ is defined in Corollary \ref{cor:product_rule_Wiener}. Recalling the definition \eqref{eq:Ksn} of $ K_{s, n} $ and \eqref{eq:Js} of $\cK_s$, we deduce that
		\begin{equation*}
		\cK_s \pare{1+K_{s, n}} = \cK_s\pare{1+\sum_{j=1}^{n-1} \cK_s^j}= \sum_{j=1}^{n} \cK_s^j = K_{s, n+1}, 
		\end{equation*}
		hence using the monotonicity of the application $ n\mapsto K_{s, n}, \ s\geqslant 0 $
		\begin{equation*}
		\begin{aligned}
		2^{n+1}\cK_s(1+K_{s,n})|h|_{1,\lambda}^n & = 2^{n+1}K_{s,n+1}|h|_{1,\lambda}^n \\
		& \leq 2 \left(2K_{s,n+1}^{1/n}|h|_{1,\lambda}\right)^n\\
		& \leq 2\left(2\cK_s|h|_{1,\lambda}\right)^n.
		\end{aligned}
		\end{equation*}
		 Thus, we proved that if $ \av{h}_{1, \lambda} < \pare{4 \cK_s \varepsilon }^{-1}$, then
\begin{equation*}
\norm{Q_{\pare{n}}\pare{\nabla \sigma}}_{\cA^{s, 1}_{\lambda}} \leqslant \frac{2}{2^{n}\varepsilon^n}\pare{1+2 \cK_s \varepsilon\delta \av{h}_{1, \lambda }\big. }  \av{h}_{s+1, \lambda}. 
\end{equation*}		 
		  This forces the series to be summable and we obtain
		  \begin{equation*}
		  \begin{aligned}
		  \norm{Q\pare{\nabla \sigma}}_{\cA^{s, 1}_{\lambda}}  \leqslant & \ \norm{Q_{\pare{0}}\pare{\nabla \sigma}}_{\cA^{s, 1}_{\lambda}} + \pare{\sum_{n=1}^\infty \varepsilon
		  ^n \norm{Q_{\pare{n}}\pare{\nabla \sigma}}_{\cA^{s, 1}_{\lambda}}}, \\
		   \leqslant & \ 2 \pare{1+\sqrt{\delta} + 2 \varepsilon\delta \ \cK_s \av{h}_{1, \lambda} }\av{h}_{s+1, \lambda} \\
		   & \  + 2 \pare{\sum_{n=1}^\infty\frac{1}{2^{n}}\pare{1+2 \cK_s \varepsilon\delta \av{h}_{1, \lambda }\big. }  \av{h}_{s+1, \lambda}}, \\
		   = & \ 2 \set{ \pare{1+2\varepsilon\delta \cK_s \av{h}_{1, \lambda}}\pare{1+\sum_{n=1}^\infty2^{-n}}  + \sqrt{\delta}}\av{h}_{s+1, \lambda}, \\
		   = & \ 2 \set{ 2\pare{1+2\varepsilon\delta \cK_s \av{h}_{1, \lambda}}  + \sqrt{\delta}}\av{h}_{s+1, \lambda}, 
		  \end{aligned}
\end{equation*}		   
Considering that under the hypothesis $ \av{h}_{1, \lambda}\leqslant\pare{4 \cK_s\varepsilon}^{-1} $ we obtain that
		\begin{align*}
		1+2\varepsilon\delta \cK_s \av{h}_{1, \lambda} \leqslant 1+\frac{\delta}{2} 
		\end{align*}
		thus, 
		\begin{equation*}
		\norm{Q\pare{\nabla \sigma}}_{\cA^{s, 1}_{\lambda}} \leqslant 2  \bra{2+\sqrt{\delta  } + \delta
		} \av{h}_{s+1, \lambda}, 
		\end{equation*}
		and we can now use that $ \delta\leqslant 1 $ to obtain the desired estimate. 	\end{proof}

	We are finally now able to bound $\phi_2$:
	
	\begin{prop}\label{prop:elliptic_est_phi2_strip}
		Let $ \phi_2 $ be the solution of \eqref{eq:phi2_3}, and let us fix $s\geq0$ and $ c_0 \in \pare{0, \frac{1}{2\cK_s\sqrt{\delta}}} $ such that 
		\begin{equation}
		\label{eq:smallness_h_1}
		\av{h}_{1, \lambda}
		\leqslant \frac{ c_0 \sqrt{\delta}}{260\ \varepsilon} , 
		\end{equation}
		then the following bound holds true
		\begin{equation}\label{eq:est_phi2}
		\norm{\nd \phi_2}_{\cA^{s, 1}_{\lambda}}
		\leqslant  \    \  8 c_0 \delta \pare{   \cK_{s+1} +2 }   \bra{ \nu\alpha  \pare{1+ 2 \alpha \av{h}_{1, \lambda}} \av{h}_{s+3, \lambda}   + \varepsilon \av{h}_{s+1, \lambda}\Big. }        . 
		\end{equation}
	\end{prop}

	\begin{proof}
		Recalling Lemma \ref{lem:interpolation_inequality_strip}, we observe that we have that
		$$
		\norm{\nd \phi_2}_{\cA^{0, 1}_{\lambda}}  \leqslant 26 \varepsilon \norm{Q\pare{\nabla \sigma}}_{\cA^{0, 1}_{\lambda}} \pare{\norm{\nd \phi_1}_{\cA^{0, 1}_{\lambda}} + \norm{\nd \phi_2}_{\cA^{0, 1}_{\lambda}}}. 
		$$
		Invoking \eqref{eq:est_Q}, we find that
		$$
		\pare{ 1-260 \varepsilon 
			\av{h}_{1, \lambda}
		}\norm{\nd \phi_2}_{\cA^{0, 1}_{\lambda}}  \leqslant  260 \varepsilon 
		\av{h}_{1, \lambda}\norm{\nd \phi_1}_{\cA^{0, 1}_{\lambda}}. 
		$$
Then, from \eqref{eq:application_elliptic_estimates} and using \eqref{eq:est_Q} we deduce that
		\begin{equation*}
		26\varepsilon \norm{Q\pare{\nabla \sigma}}_{\cA^{0, 1}_{\lambda}} \norm{\nd \phi_2}_{\cA^{s, 1}_{\lambda}} \leqslant 260 \varepsilon 
		\av{h}_{1, \lambda} \norm{\nd \phi_2}_{\cA^{s, 1}_{\lambda}}. 
		\end{equation*}
		We use such estimate in order to move the term $ 26\varepsilon \norm{Q\pare{\nabla \sigma}}_{\cA^{0, 1}_{\lambda}} \norm{\nd \phi_2}_{\cA^{s, 1}_{\lambda}} $ on the left hand side of the inequality \eqref{eq:application_elliptic_estimates}. This gives
			\begin{multline}\label{eq:bound_elliptic_est_phi2_2}
		\pare{ 1-260\cK_s \varepsilon 
			\av{h}_{1, \lambda}
		}
	\norm{\nd \phi_2}_{\cA^{s, 1}_{\lambda}}
	\leqslant 26\varepsilon \ \cK_s \bigg[ \norm{Q\pare{\nabla \sigma}}_{\cA^{0, 1}_{\lambda}} \norm{\nd \phi_1}_{\cA^{s, 1}_{\lambda}} \\
	+ \norm{Q\pare{\nabla \sigma}}_{\cA^{s, 1}_{\lambda}} \norm{\nd \phi_1}_{\cA^{0, 1}_{\lambda}} + \norm{Q\pare{\nabla \sigma}}_{\cA^{s, 1}_{\lambda}} \frac{260 \varepsilon 
		\av{h}_{1, \lambda}}{\pare{ 1-260 \varepsilon 
			\av{h}_{1, \lambda}
		}}\norm{\nd \phi_1}_{\cA^{0, 1}_{\lambda}} \bigg]. 
	\end{multline}
		We apply now Lemma 	Lemma \ref{lem:elliptic_estimates_phi1}, \eqref{eq:est_phi1} and \eqref{eq:est_Q} to \eqref{eq:bound_elliptic_est_phi2_2} in order to obtain an inequality which involves norms of $ \phi_2 $ and $ h $ only
		
		\begin{equation*}
		\begin{aligned}
		\pare{1-
			260 \cK_s \; \varepsilon \av{h}_{1, \lambda}
		}
		\norm{\nd \phi_2}_{\cA^{s, 1}_{\lambda}} \leqslant & \ 4 \cK_s\sqrt{\delta} \Bigg\lbrace \left[  \cK_{s+1} 260 \; \varepsilon  \av{h}_{1, \lambda}    \pare{ \nu\alpha  \pare{1+ 2 \alpha \av{h}_{1, \lambda}} \av{h}_{s+3, \lambda}   + \varepsilon \av{h}_{s+1, \lambda}}\right.  \\
		& \  +\left.  260 \; \varepsilon  \av{h}_{s+1, \lambda}   \pare{ \nu\alpha  \pare{1+ 2 \alpha \av{h}_{1, \lambda}} \av{h}_{3, \lambda}   + \varepsilon \av{h}_{1, \lambda}} \right] \\
		& \left. +   \pare{ \nu\alpha  \pare{1+ 2 \alpha \av{h}_{1, \lambda}} \av{h}_{3, \lambda}   + \varepsilon \av{h}_{1, \lambda}}  \frac{260 \; \varepsilon  \av{h}_{s+1, \lambda}}{1- 260 \; \varepsilon  \av{h}_{1, \lambda}  }\right\rbrace .
		\end{aligned}
		\end{equation*}
		We imposed the smallness hypothesis \eqref{eq:smallness_h_1} in order to deduce that
		\begin{equation*}
		260 \; \varepsilon  \av{h}_{1, \lambda} < c_0 \sqrt{\delta} < \frac{1}{2} , 
		\end{equation*}
		so that we can derive the inequality
		\begin{multline*}
		\pare{1- c_0 \cK_s
			\sqrt{\delta}
		}
		\norm{\nd \phi_2}_{\cA^{s, 1}_{\lambda}}
		\leqslant  \ 4 \cK_s c_0 \delta \Bigg\lbrace \bigg{[}  \cK_{s+1}    \pare{ \nu\alpha  \pare{1+ 2 \alpha \av{h}_{1, \lambda}} \av{h}_{s+3, \lambda}   + \varepsilon \av{h}_{s+1, \lambda}}   \\
		+ \frac{ \av{h}_{s+1, \lambda}}{\av{h}_{1, \lambda}}   \pare{ \nu\alpha  \pare{1+ 2 \alpha \av{h}_{1, \lambda}} \av{h}_{3, \lambda}   + \varepsilon \av{h}_{1, \lambda}} \bigg{]} \\
		+  \frac{c_0 \delta}{\pare{ 1-c_0 \sqrt{ \delta} }} \frac{ \av{h}_{s+1, \lambda}}{\av{h}_{1, \lambda }} \pare{ \nu\alpha  \pare{1+ 2 \alpha \av{h}_{1, \lambda}} \av{h}_{3, \lambda }   + \varepsilon \av{h}_{1, \lambda }}    \Bigg\rbrace .
		\end{multline*}
		
		We use now \eqref{eq:interpolation_inequality} in order to deduce that
		\begin{equation*}
		\av{h}_{3, \lambda}\av{h}_{s+1, \lambda}\leqslant \av{h}_{1, \lambda} \av{h}_{s+3, \lambda}, 
		\end{equation*}
		from which we obtain
		\begin{equation*}
		\norm{\nd \phi_2}_{\cA^{s, 1}_{\lambda}}
		\leqslant  \  \frac{4 \cK_s c_0 \delta}{1-c_0\cK_s\sqrt{ \delta}} \pare{   \cK_{s+1} +1+ \frac{c_0\delta}{\pare{ 1-c_0\sqrt{ \delta} }} }   \bra{ \nu\alpha  \pare{1+ 2 \alpha \av{h}_{1, \lambda}} \av{h}_{s+3, \lambda}   + \varepsilon \av{h}_{s+1, \lambda}\Big. }  . 
		\end{equation*}
Due to the estimates 
$$
 1+ \frac{c_0\delta}{\pare{ 1-c_0\sqrt{ \delta} }} \leqslant 2 
 $$ 
 and 
 $$ \frac{1}{1-c_0\cK_s\sqrt{ \delta}} \leqslant 2, 
 $$ we conclude the proof. 
	\end{proof}

	\begin{cor}\label{cor:bounds_nabla_phi2}
		Let $ s, \lambda \geqslant 0 $, the following estimate holds true
		\begin{align}
		\av{\partial_2 \phi_2}_{s, \lambda}  \leqslant  \    \ 8 c_0 \delta \pare{  \cK_{s+1} +2 }   \bra{ \nu\alpha  \pare{1+ 2 \alpha \av{h}_{1, \lambda }} \av{h}_{s+3, \lambda}   + \varepsilon \av{h}_{s+1, \lambda }\Big. } . 
		\label{eq:bound_pa2phi2} 
		\end{align}
	\end{cor}

	\begin{proof}
		We use the trace estimates of Proposition \ref{prop:trace-embedding} to argue that
		\begin{align*}
		\av{\partial_2 \phi_2}_{s, \lambda }\leqslant \norm{\partial_2 \phi_2}_{\cA^{s, 1}_{\lambda}},  
		\end{align*}
		and combine the above inequality with \eqref{eq:est_phi2}. 
	\end{proof}

	\subsection{Parabolic estimates}\label{sec:parabolic_estimates}
We observe that, due to the definition of $\phi_2$, we have that
$$ 
\partial_1\phi_2 = 0 
$$ 
on $ \set{x_2=0} $. Similarly, we notice that we are interested in the particular case $s=1$ (so $\cK_1=1$ which will be used throughtout this section) of the elliptic estimates in the previous section. As a consequence, we can rewrite the evolution equation for $ h $ in \eqref{eq:Muskat5} as
	\begin{equation} \label{eq:Muskat6}
	\partial_t h + \frac{1}{\varepsilon} \tanh\pare{\sqrt{\delta} \Lambda} \pare{\nu \alpha \Lambda^3 h - \varepsilon \Lambda h} = N_{\phi_1} + N_{\phi_2} + N_{h}, 
	\end{equation}
	where
	\begin{align*}
	N_{\phi_1} & = - \sqrt{\delta} A_1^k \partial_k \phi_1 \ \partial_1 h, \\
	N_{\phi_2} & = - \sqrt{\delta} A_1^2 \partial_2 \phi_2 \ \partial_1 h + \frac{1}{\alpha} \ A_2^2 \partial_2 \phi_2 , 
	\end{align*}
	and 
	\begin{multline}\label{eq:def_Nh}
	N_h = \sqrt{\delta} \left[
	\mathsf{G}\pare{ \frac{\varepsilon \Lambda}{\tanh\pare{\Lambda}} h} \tanh\pare{\sqrt{\delta}\Lambda} \pare{\nu \alpha \Lambda^3 h - \varepsilon \Lambda h}\right. \\
	\left. - \nu \alpha \pare{1-\mathsf{G}\pare{\frac{\varepsilon \Lambda}{\tanh\pare{\Lambda}} h}} \tanh\pare{\sqrt{\delta}\Lambda} \Lambda \pare{\mathsf{F}\pare{\alpha h'} \Lambda^2 h} 
	\right]. 
	\end{multline}
	
	Indeed the nonlinear bounds will be of different kind for the $ N_j $'s:
	\begin{itemize}
		
		\item Since $ N_h $ is an explicit nonlinear function of $ h $ , the nonlinear bounds for $ N_h $ can be computed explicitly using an estimate as the one stated in Lemma \ref{lem:interpolation_inequality}.
		
		\item Recalling \eqref{eq:phi1}, we can express $ \phi_1 $ explicitly as a function of $ h $. Thus, the nonlinear bounds for $ N_{\phi_1} $ are also relatively easy  to obtain. 
		
		\item The nonlinear bounds for $ \phi_2 $ are more involved. It is in this setting that the elliptic estimates performed in Section \ref{sec:elliptic_estimates} will be used. 
		
	\end{itemize}
	
	Let us recall that for any $ \tilde{\mu} \geqslant 0 $ we have
	\begin{equation*}
	\ddt \av{h \pare{t}}_{s, \tilde{\mu} t } = \pare{ \sum_n \pare{1+\av{n}}^s e^{\tilde{\mu} t \av{n} }\partial_t \av{\hat{h}\pare{n}} } + \tilde{\mu} \av{\Lambda h }_{s, \tilde{\mu} t}, 
	\end{equation*}
	hence if $ \tilde{\mu} $ is sufficiency small we can absorb a negative energy contribution of the form $ -\tilde{\mu} \av{\Lambda h }_{s, \tilde{\mu} t} $ using \eqref{eq:hom-nonhom} in the parabolic term in the energy estimates of the equation \eqref{eq:Muskat6}. Motivated by such considerations let us consider any 
	\begin{equation}\label{eq:mu}
	\mu \in \left[0, \frac{\sqrt{\delta}}{4}\pare{\nu\sqrt{\delta} - 1}\right) ,
	\end{equation}
	and let us perform an $ A^1_{\mu t} $ energy estimate on \eqref{eq:Muskat6}. 
For mean-free function $ h $, we compute
\begin{align*}
 \av{\tanh\pare{\sqrt{\delta} \Lambda} h}_{s, \lambda} \geqslant \frac{\sqrt{\delta}}{2}\av{h}_{s, \lambda}, && s, \lambda \geqslant 0. 
\end{align*} 
Then we deduce the following inequality 
	\begin{equation}\label{eq:Muskat7}
	\ddt \av{h}_{1, \mu t} + \frac{\sqrt{\delta}}{8} \pare{\nu\sqrt{\delta} - 1 } \av{h}_{4, \mu t} \leqslant \av{N_{\phi_1}}_{1, \mu t} + \av{N_{\phi_2}}_{1, \mu t} + \av{N_{h}}_{1, \mu t} .
	\end{equation}

	Before starting to provide the suitable parabolic estimates we prove the following technical lemma, which will be continuously use in  the rest of the manuscript:
	
	\begin{lemma}\label{lem:pullback_transport_estimates}
		Let $ A_i^j, \ i, j =1, 2 $ be the coefficients of the matrix $ A $ defined in \eqref{eq:matrixA}, and let $ h, w \in \mathbb{A}^{s+1}_\lambda, \ s, \lambda \geqslant 0 $ then
		\begin{align*}
		\av{A_1^k \partial_k w}_{s, \lambda}  \leqslant &  \ \av{\partial_1 w}_{s, \lambda} + \pare{1+ \varepsilon \bra{
				1+ 
				\frac{2 \varepsilon \av{h}_{1, \lambda} }{1- \varepsilon \av{h}_{1, \lambda}}
			}  \av{h}_{1, \lambda}} \av{\partial_2 w}_{s, \lambda} \\
		& \ + \cK_s\varepsilon \bra{
			1+ \varepsilon \av{h}_{1, \lambda} \pare{
				\frac{1}{1- \varepsilon \av{h}_{1, \lambda}} + \frac{1}{1-\cK_s \varepsilon \av{h}_{1, \lambda}}
			}
		}  \av{h}_{s+1, \lambda} \av{\partial_2 w}_{0, \lambda} , \\
		\av{A_2^2 \partial_2 w}_{s, \lambda}   \leqslant & \  \pare{
			1+ \frac{\varepsilon}{1-\varepsilon\av{h}_{1, \lambda}} \ \av{h}_{1, \lambda}
		} \av{\partial_2 w}_{s, \lambda} + \frac{\varepsilon}{1-\cK_s\varepsilon\av{h}_{1, \lambda}} \ \av{h}_{s+1, \lambda} \ \av{\partial_2 w}_{0, \lambda}. 
		\end{align*}
	\end{lemma}

	\begin{proof}
		From the explicit definition of $ A $ given in \eqref{eq:matrixA} and the definition of the function $ \mathsf{G} $ given in \eqref{eq:def_F} we can rewrite the trace of $ A $ in $ x_2 =0 $ using \eqref{eq:nabla_sigma} as
		\begin{align*}
		\left. A\right|_{x_2 =0} & = I + \widetilde{A} , \\
		\widetilde{A} & = \pare{
			\begin{array}{cc}
			0 & 0 \\[3mm]
			\varepsilon \partial_1 h \pare{\mathsf{G}\pare{\frac{\varepsilon\Lambda}{\tanh\pare{\Lambda}} \ h} - 1} & - \mathsf{G}\pare{\frac{\varepsilon\Lambda}{\tanh\pare{\Lambda}} \ h}
			\end{array}
		}. 
		\end{align*}
		We can now use Lemma \ref{lem:interpolation_inequality}
		\begin{equation*}
		\begin{aligned}
		\av{\widetilde{A}_{1}^2}_{s, \lambda} \leqslant & \ \cK_s \varepsilon \bra{
			\av{\partial_1 h}_{s, \lambda} \pare{\av{\mathsf{G}\pare{\frac{\varepsilon\Lambda}{\tanh\pare{\Lambda}} \ h}}_{0, \lambda} + 1} + 
			\av{\partial_1 h}_{0, \lambda} \av{\mathsf{G}\pare{\frac{\varepsilon\Lambda}{\tanh\pare{\Lambda}} \ h}}_{s, \lambda}
		} , \\
		\av{\widetilde{A}_2^2}_{s, \lambda} \leqslant & \  \av{\mathsf{G}\pare{\frac{\varepsilon\Lambda}{\tanh\pare{\Lambda}} \ h}}_{s, \lambda},  
		\end{aligned}
		\end{equation*}
		and control the contributions provided by $ \mathsf{G} $ with Lemma \ref{lem:composition_Wiener_FG}
		\begin{equation*}
		\begin{aligned}
		\av{\widetilde{A}_{1}^2}_{s, \lambda} \leqslant &  \ \cK_s \varepsilon \bra{
			1+ \varepsilon \av{h}_{1, \lambda} \pare{
				\frac{1}{1-\cK_0 \varepsilon \av{h}_{1, \lambda}} + \frac{1}{1-\cK_s \varepsilon \av{h}_{1, \lambda}}
			}
		}  \av{h}_{s+1, \lambda}, \\
		\av{\widetilde{A}_2^2}_{s, \lambda} \leqslant & \ \frac{\varepsilon}{1-\cK_s\varepsilon\av{h}_{1, \lambda}} \ \av{h}_{s+1, \lambda}. 
		\end{aligned}
		\end{equation*}
		We combine the above estimates and use Lemma \ref{lem:interpolation_inequality} in order to obtain the desired bounds
		\begin{equation*}
		\begin{aligned}
		\av{A_1^k \partial_k w}_{s , \lambda}  \leqslant & \ \av{\partial_1 w}_{s , \lambda} + \pare{1+ \av{\widetilde{A}^2_1}_{0, \lambda}} \av{\partial_2 w}_{s , \lambda} + \av{\widetilde{A}^2_1}_{s, \lambda} \av{\partial_2 w}_{0, \lambda} , \\
		\leqslant & \ \av{\partial_1 w}_{s, \lambda} + \pare{1+ \varepsilon \bra{
				1+ 
				\frac{2 \varepsilon \av{h}_{1, \lambda} }{1-\cK_0 \varepsilon \av{h}_{1, \lambda}}
			}  \av{h}_{1, \lambda}} \av{\partial_2 w}_{s, \lambda} \\
		& \ + \cK_s \varepsilon \bra{
			1+ \varepsilon \av{h}_{1, \lambda} \pare{
				\frac{1}{1-\cK_0 \varepsilon \av{h}_{1, \lambda}} + \frac{1}{1-\cK_s \varepsilon \av{h}_{1, \lambda}}
			}
		}  \av{h}_{s+1, \lambda} \av{\partial_2 w}_{0, \lambda} , \\
		\av{A_2^2 \ \partial_2 w}_{s, \lambda} \leqslant & \ \pare{1+ \av{\widetilde{A}_2^2}_{0 , \lambda}} \av{\partial_2 w}_{s, \lambda} + \av{\widetilde{A}_2^2}_{s , \lambda} \av{\partial_2 w}_{0, \lambda}, \\
		\leqslant & \ \pare{
			1+ \frac{\varepsilon}{1-\cK_0\varepsilon\av{h}_{1, \lambda}} \ \av{h}_{1, \lambda}
		} \av{\partial_2 w}_{s, \lambda} + \frac{\varepsilon}{1-\cK_s\varepsilon\av{h}_{1, \lambda}} \ \av{h}_{s+1, \lambda} \ \av{\partial_2 w}_{0, \lambda}. 
		\end{aligned}
		\end{equation*}
	\end{proof}

	\subsubsection{Bounds for $ N_{\phi_2} $} \label{sec:bounds_Nphi2}

	We start with the nonlinear bounds for the term $ N_{\phi_2} $ since they are the more challenging. Indeed we have
	\begin{equation*}
	\av{N_{\phi_2}}_{1, \lambda} = \av{N_{\phi_2 , 1}}_{1, \lambda} + \av{N_{\phi_2 , 2}}_{1, \lambda}, 
	\end{equation*}
	where
	\begin{align*}
	N_{\phi_2 , 1} & =  \frac{1}{\alpha} \ A_2^2 \partial_2 \phi_2 , \\
	N_{\phi_2 , 2} & =  - \sqrt{\delta} \  A_1^2 \partial_2 \phi_2 \ \partial_1 h . 
	\end{align*}

	We invoke Lemma \ref{lem:pullback_transport_estimates} to obtain
	\begin{equation*}
	\av{N_{\phi_2, 1}}_{1, \lambda} \leqslant \frac{1}{\alpha} \bra{
		\pare{
			1+ \frac{\varepsilon}{1-\cK_0\varepsilon\av{h}_{1, \lambda}} \ \av{h}_{1, \lambda}
		} \av{\partial_2 \phi_2}_{1, \lambda} + \frac{\varepsilon}{1-\cK_1 \varepsilon\av{h}_{1, \lambda}} \ \av{h}_{2, \lambda} \ \av{\partial_2 \phi_2}_{0, \lambda}
	}. 
	\end{equation*}
	
	We can now recall the result provided in Corollary \ref{cor:bounds_nabla_phi2} so that we can express $ \av{\partial_2\phi_2}_{s_0, \lambda}, \ s_0 =0, 1, \lambda \geqslant 0 $ in terms of boundary contributions involving $ h $ and its derivative only. This combined with the fact that $ \cK_s\leqslant 2^{s} $ (see Lemma \ref{lem:interpolation_inequality}), leads to
	\begin{multline*}
	\av{N_{\phi_2, 1}}_{1, \lambda} \leqslant  \frac{48 \ c_0 \sqrt{\delta}}{\varepsilon }    \left[
	\pare{
		1+ \frac{\varepsilon}{1-\varepsilon\av{h}_{1, \lambda}} \ \av{h}_{1, \lambda}
	} \bra{ \Big. \nu\alpha  \pare{1+ 2 \alpha \av{h}_{1, \lambda}} \av{h}_{4, \lambda}   + \varepsilon \av{h}_{2, \lambda}} \right. \\
	\left. + \frac{\varepsilon}{1- \varepsilon\av{h}_{1, \lambda}} \ \av{h}_{2, \lambda} \ \bra{ \Big.  \nu\alpha  \pare{1+ 2 \alpha \av{h}_{1, \lambda}} \av{h}_{3, \lambda}   + \varepsilon \av{h}_{1, \lambda}}
	\right]. 
	\end{multline*}
	We can use now the interpolation inequality \eqref{eq:interpolation_inequality} to get $ \av{h}_{2, \lambda} \av{h}_{3, \lambda} \leqslant \av{h}_{1, \lambda} \av{h}_{4, \lambda} $, which we use in order to deduce
	\begin{multline*}
	\frac{\varepsilon}{1- \varepsilon\av{h}_{1, \lambda}} \ \av{h}_{2, \lambda} \ \bra{ \Big.  \nu\alpha  \pare{1+ 2 \alpha \av{h}_{1, \lambda}} \av{h}_{3, \lambda}   + \varepsilon \av{h}_{1, \lambda}}\\
	\leqslant \frac{\varepsilon}{1- \varepsilon\av{h}_{1, \lambda}} \ \av{h}_{1, \lambda} \ \bra{ \Big.  \nu\alpha  \pare{1+ 2 \alpha \av{h}_{1, \lambda}} \av{h}_{4, \lambda}   + \varepsilon \av{h}_{2, \lambda}}. 
	\end{multline*}
Thus,
	\begin{equation}\label{eq:bound_Nphi21}
	\av{N_{\phi_2, 1}}_{1, \lambda} \leqslant\frac{48 \ c_0 \sqrt{\delta}}{\varepsilon }  \bra{
		\pare{
			1+ \frac{2 \varepsilon}{1- \varepsilon\av{h}_{1, \lambda}} \ \av{h}_{1, \lambda}
		} \bra{ \Big.  \nu\alpha\sqrt{\delta}  \pare{1+ 2 \alpha \av{h}_{1, \lambda}} \av{h}_{4, \lambda}   + \varepsilon \av{h}_{2, \lambda}} } . 
	\end{equation}

	
	Analogously as above we use  Lemma \ref{lem:interpolation_inequality} on the term $ N_{\phi_2, 2} $ to obtain that
	\begin{equation*}
	\av{N_{\phi_2, 2}}_s \leqslant \sqrt{\delta} \pare{\av{h}_{1, \lambda} \av{A_1^2 \partial_2 \phi_2}_{1, \lambda} + \av{h}_{2, \lambda} \av{A_1^2 \partial_2 \phi_2}_{0, \lambda} \Big. }.
	\end{equation*}
We use Lemma \ref{lem:pullback_transport_estimates} to compute that
	\begin{multline*}
	\av{N_{\phi_2, 2}}_{1, \lambda} \leqslant \sqrt{\delta} \Bigg[  \pare{1+ \varepsilon \pare{
			1+ 
			\frac{2 \varepsilon \av{h}_{1, \lambda} }{1- \varepsilon \av{h}_{1, \lambda}}
		}  \av{h}_{1, \lambda}}\av{h}_{1, \lambda}  \av{\partial_2 \phi_2}_{1, \lambda} \\
	+ \pare{ 1+ 3\varepsilon \pare{
			1+ 
			\frac{2 \varepsilon \av{h}_{1, \lambda} }{1- \varepsilon \av{h}_{1, \lambda}}
		}  \av{h}_{1, \lambda}} \av{h}_{2, \lambda} \av{\partial_2 \phi_2}_{0, \lambda} 
	\Bigg]. 
	\end{multline*}
	It remains hence to control the trace  $ \av{ \partial_2 \phi_2}_{s_0}, \ s_0=0, 1 $. To this end we invoke Corollary \ref{cor:bounds_nabla_phi2} and obtain that
	\begin{multline*}
	\av{N_{\phi_2, 2}}_{1, \lambda} \leqslant 48 \ c_0 \delta^{3/2}    \left[
	\pare{1+ \varepsilon \pare{
			1+ 
			\frac{2 \varepsilon \av{h}_{1, \lambda} }{1- \varepsilon \av{h}_{1, \lambda}}
		}  \av{h}_{1, \lambda}}\av{h}_{1, \lambda}  \bra{ \Big. \nu\alpha  \pare{1+ 2 \alpha \av{h}_{1, \lambda}} \av{h}_{4, \lambda}   + \varepsilon \av{h}_{2, \lambda}} \right. \\
	\left.  + \pare{ 1+ 3\varepsilon \pare{
			1+ 
			\frac{2 \varepsilon \av{h}_{1, \lambda} }{1- \varepsilon \av{h}_{1, \lambda}}
		}  \av{h}_{1, \lambda}} \av{h}_{2, \lambda}  \ \bra{ \Big. \nu\alpha  \pare{1+ 2 \alpha \av{h}_{1, \lambda}} \av{h}_{3, \lambda}   + \varepsilon \av{h}_{1, \lambda}}
	\right]. 
	\end{multline*}
Due to an interpolation inequality \eqref{eq:interpolation_inequality}, we find that
	\begin{equation}\label{eq:bound_Nphi22}
	\av{N_{\phi_2, 2}}_{1, \lambda} \leqslant  96 \ c_0 \delta^{3/2}  
	\pare{1+ 2 \varepsilon \pare{
			1+ 
			\frac{2 \varepsilon \av{h}_{1, \lambda} }{1- \varepsilon \av{h}_{1, \lambda}}
		}  \av{h}_{1, \lambda}}\av{h}_{1, \lambda}  \bra{ \Big. \nu\alpha  \pare{1+ 2 \alpha \av{h}_{1, \lambda}} \av{h}_{4, \lambda}   + \varepsilon \av{h}_{2, \lambda}} . 
	\end{equation}
	
	We combine now \eqref{eq:bound_Nphi21} and \eqref{eq:bound_Nphi22} in order to obtain the final bound for $ N_{\phi_2} $:
	
	\begin{multline}\label{eq:bound_Nphi2}
	\av{N_{\phi_2}}_{1, \lambda} \leqslant  48 \ c_0 \sqrt{\delta} 
	\Bigg\{
	\frac{1}{\varepsilon} \pare{1+ \frac{2\varepsilon\av{h}_{1, \lambda}}{1-\varepsilon\av{h}_{1, \lambda}}} \\
	+ 2\delta \av{h}_{1, \lambda} \pare{1+2\varepsilon\av{h}_{1, \lambda}\pare{1+\frac{2\varepsilon\av{h}_{1, \lambda}}{1-\varepsilon\av{h}_{1, \lambda}}}}
	\Bigg\}
	\bra{ \Big. \nu\alpha  \pare{1+ 2 \alpha \av{h}_{1, \lambda}} \av{h}_{4, \lambda}   + \varepsilon \av{h}_{2, \lambda}}. 
	\end{multline}


	\subsubsection{Bounds for $ N_{\phi_1} $} \label{sec:bounds_Nphi1}

	We consider now the term $ N_{\phi_1} $. We can use the product rule stated in Lemma \ref{lem:interpolation_inequality} to find that
	\begin{equation*}
	\av{N_{\phi_1}}_{1, \lambda} \leqslant 
	\sqrt{\delta} \pare{ \av{A_1^k \partial_k \phi_1}_{0, \lambda} \av{ \partial_1 h}_{1, \lambda}  + \av{A_1^k \partial_k \phi_1}_{1, \lambda} \av{ \partial_1 h}_{0, \lambda}  }. 
	\end{equation*}
	Next we apply the estimates proved in Lemma \ref{lem:pullback_transport_estimates} in order to control terms $ \av{A_1^k \partial_k \phi_1}_{s_0}, \ s_0= 0, 1 $. This gives
	\begin{multline}\label{eq:bound_Nphi1_1}
	\av{N_{\phi_1}}_{1, \lambda} \leqslant \sqrt{\delta} \left[
	\pare{
		\av{\partial_1 \phi_1 }_{0, \lambda} + \pare{1+ 2 \varepsilon \bra{ 1+   \frac{2 \varepsilon \av{h}_{1, \lambda} }{1- \varepsilon \av{h}_{1, \lambda}} }  \av{h}_{1, \lambda}} \av{\partial_2 \phi_1 }_{0, \lambda}
	} \av{h}_{2, \lambda}
	\right.  \\
	+ \left. 
	\pare{\av{\partial_1 \phi_1}_{1, \lambda} + \pare{1+ \varepsilon \bra{ 1+   \frac{2 \varepsilon \av{h}_{1, \lambda} }{1- \varepsilon \av{h}_{1, \lambda}} }  \av{h}_{1, \lambda}} \av{\partial_2 \phi_1 }_{1, \lambda} + \varepsilon \bra{ 1+   \frac{2 \varepsilon \av{h}_{1, \lambda} }{1- \varepsilon \av{h}_{1, \lambda}} } \ \av{h}_{2 } \av{\partial_2 \phi_1 }_{0, \lambda}} \av{h}_{1, \lambda}
	\right]. 
	\end{multline}
	In \eqref{eq:bound_Nphi1_1} there still appear terms of the form $ \av{\partial_j \phi_1}_s, \ j=1, 2, \ s=0, 1 $. It is here that we use the elliptic estimates of Section \ref{sec:elliptic_estimates}, so that we can control the trace contributions $ \av{\partial_j \phi_1}_s $ in terms of suitable nonlinear bounds involving $ h $ only. We apply the estimates of Lemma \ref{lem:elliptic_estimates_phi1}, which were performed explicitly for this task, obtaining
	\begin{equation}\label{eq:bound_Nphi1_2}
	\begin{aligned}
	\av{N_{\phi_1}}_{1, \lambda} \leqslant 2 \sqrt{\delta} \Bigg[ \ &
	\pare{\nu\alpha \av{h}_{3, \lambda} \pare{1+ 2 \alpha \av{h}_{1, \lambda}}   + \varepsilon \av{h}_{1, \lambda}}  \av{h}_{2, \lambda} \\
	& + \pare{1+ 2 \varepsilon \bra{ 1+   \frac{2 \varepsilon \av{h}_{1, \lambda} }{1- \varepsilon \av{h}_{1, \lambda}} }  \av{h}_{1, \lambda}}  \pare{\sqrt{\delta}\bra{ \nu\alpha \av{h}_{3, \lambda} \pare{1+ 2 \alpha \av{h}_{1, \lambda}}   + \varepsilon \av{h}_{1, \lambda} } } \av{h}_{2, \lambda}  \\
	& + 
	\pare{\nu\alpha \av{h}_{4, \lambda} \pare{1+ 2 \alpha \av{h}_{1, \lambda}}   + \varepsilon \av{h}_{2, \lambda}}  \av{h}_{1, \lambda} \\
	&+ 4 \pare{1+ \varepsilon \bra{ 1+   \frac{2 \varepsilon \av{h}_{1, \lambda} }{1- \varepsilon \av{h}_{1, \lambda}} }  \av{h}_{1, \lambda}} \pare{\sqrt{\delta}\bra{ \nu\alpha \av{h}_{4, \lambda} \pare{1+ 2 \alpha \av{h}_{1, \lambda}}   + \varepsilon \av{h}_{2, \lambda} } }  \av{h}_{1, \lambda} \\
	& + \varepsilon \bra{ 1+   \frac{2 \varepsilon \av{h}_{1, \lambda} }{1- \varepsilon \av{h}_{1, \lambda}} }  \ \av{h}_{2 } \pare{\sqrt{\delta}\bra{ \nu\alpha \av{h}_{3, \lambda} \pare{1+ 2 \alpha \av{h}_{1, \lambda}}   + \varepsilon \av{h}_{1, \lambda} } }\av{h}_{1, \lambda}
	\Bigg]. 
	\end{aligned}
	\end{equation}
	We simplify now the expression in \eqref{eq:bound_Nphi1_2}. Invoking interpolation inequalities allow us to produce the inequality
	\begin{equation}\label{eq:bound_Nphi1_3}
	\pare{\nu\alpha \av{h}_{3, \lambda} \pare{1+ 2 \alpha \av{h}_{1, \lambda}}   + \varepsilon \av{h}_{1, \lambda}}  \av{h}_{2, \lambda} 
	\leqslant  \av{h}_{1, \lambda} \pare{\nu\alpha \av{h}_{4, \lambda} \pare{1+ 2 \alpha \av{h}_{1, \lambda}}   + \varepsilon \av{h}_{2, \lambda}}  . 
	\end{equation}
	
	We can now apply \eqref{eq:bound_Nphi1_3} repeatedly in \eqref{eq:bound_Nphi1_2} in order to reduce the previous expression. The resulting inequality is 
	\begin{equation}\label{eq:bound_Nphi1}
	\av{N_{\phi_1}} \leqslant 2\sqrt{\delta}
	\bra{
		2 + \sqrt{\delta} \pare{5+ 7 \varepsilon  \bra{ 1+   \frac{2 \varepsilon \av{h}_{1, \lambda} }{1- \varepsilon \av{h}_{1, \lambda}} }  \av{h}_{1, \lambda}  }
	}\av{h}_{1, \lambda} \bra{\nu\alpha \av{h}_{4, \lambda} \pare{1+ 2 \alpha \av{h}_{1, \lambda}}   + \varepsilon \av{h}_{2, \lambda}\Big. }. 
	\end{equation}

	\subsubsection{Bounds for $ N_h $} \label{sec:bounds_Nh}
	
	At last we need to bound the nonlinear term $ N_h $ defined in \eqref{eq:def_Nh}. It is natural to decompose $ N_h $ as 
	\begin{equation*}
	N_h = N_{h, 1} + N_{h, 2} + N_{h, 3}, 
	\end{equation*}
	where
	\begin{equation*}
	\begin{aligned}
	N_{h, 1} & = \sqrt{\delta} 
	\mathsf{G}\pare{ \frac{\varepsilon \Lambda}{\tanh\pare{\Lambda}} h} \tanh\pare{\sqrt{\delta}\Lambda} \pare{\nu \alpha \Lambda^3 h - \varepsilon \Lambda h}, \\
	N_{h, 2} & =   - \sqrt{\delta}  \nu \alpha  \tanh\pare{\sqrt{\delta}\Lambda} \Lambda \pare{\mathsf{F}\pare{\alpha h'} \Lambda^2 h} , \\
	N_{h, 3} & = \sqrt{\delta}  \nu \alpha\mathsf{G}\pare{\frac{\varepsilon \Lambda}{\tanh\pare{\Lambda}} h} \tanh\pare{\sqrt{\delta}\Lambda} \Lambda \pare{\mathsf{F}\pare{\alpha h'} \Lambda^2 h} . 
	\end{aligned}
	\end{equation*}
	
	What remain hence is to deduce some suitable estimate for the quantities $ \av{N_{h, j}}_{1, \lambda}, \ j=1, 2, 3 $. We start analyzing $ N_{h, 1} $. We can apply the product rule of Lemma \ref{lem:interpolation_inequality} and the estimate 
	$$
	 \av{\tanh\pare{\sqrt{\delta}\Lambda} f}_s \leqslant \av{f}_s, \ s\geqslant 0 
	 $$ to find that
	\begin{equation}\label{eq:Nh1_1}
	\av{N_{h, 1}}_{1, \lambda} \leqslant \sqrt{\delta} \pare{
		\av{\mathsf{G}\pare{ \frac{\varepsilon \Lambda}{\tanh\pare{\Lambda}} h}}_{1, \lambda} \av{ \pare{\nu \alpha \Lambda^3 h - \varepsilon \Lambda h}}_ {0, \lambda} + \av{\mathsf{G}\pare{ \frac{\varepsilon \Lambda}{\tanh\pare{\Lambda}} h}}_{0, \lambda} \av{ \pare{\nu \alpha \Lambda^3 h - \varepsilon \Lambda h}}_ {1, \lambda}
	}. 
	\end{equation}
	In order to control the contributions provided by $ \av{\mathsf{G}\pare{ \frac{\varepsilon \Lambda}{\tanh\pare{\Lambda}} h}}_{s_0, \lambda}, \ s_0 =0, 1 $ we use Lemma \ref{lem:composition_Wiener_FG} and combine it with the inequality
	\begin{align*}
	\av{\frac{\varepsilon \Lambda}{\tanh\pare{\Lambda}} \ f}_{s,\lambda} \leqslant \varepsilon \av{f}_{s+1, \lambda}, && s, \lambda \geqslant 0, 
	\end{align*}
	which is valid for zero-mean functions $ f $ in order to obtain that
	\begin{align}\label{eq:Nh1_2}
	\av{\mathsf{G}\pare{ \frac{\varepsilon \Lambda}{\tanh\pare{\Lambda}} h}}_{s_0, \lambda} \leqslant \frac{\varepsilon}{1-\varepsilon \av{h}_{1, \lambda}} \av{h}_{s_0+1, \lambda}, && s_0 = 0, 1.
	\end{align}
	We apply the estimate \eqref{eq:Nh1_2} in \eqref{eq:Nh1_1} and we conclude that
	\begin{equation}\label{eq:Nh1_3}
	\av{N_{h, 1}}_{1, \lambda} \leqslant   \frac{\varepsilon \sqrt{\delta} }{1-\varepsilon \av{h}_{1, \lambda}}  \pare{ \Big. 
		\av{h}_{2, \lambda}\pare{\nu\alpha\av{h}_{3, \lambda} + \varepsilon\av{h}_{1, \lambda}} + \av{h}_{1, \lambda}\pare{\nu\alpha\av{h}_{4, \lambda} + \varepsilon\av{h}_{2, \lambda}}
	}. 
	\end{equation}
	Using interpolation, we find that the final bound for $ N_{h, 1} $ reduces to
	\begin{equation}\label{eq:Nh1_4}
	\av{N_{h, 1}}_{1, \lambda} \leqslant  \varepsilon \sqrt{\delta} \frac{2  \av{h}_{1, \lambda} }{1-\varepsilon \av{h}_{1, \lambda}}  \pare{\nu\alpha\av{h}_{4, \lambda} + \varepsilon\av{h}_{2, \lambda}}
	. 
	\end{equation}

	We consider now the term $ N_{h, 2} $ and we obtain the following estimate
	\begin{equation}\label{eq:Nh2_1}
	\av{N_{h, 2}}_{1, \lambda} \leqslant \sqrt{\delta} \nu \alpha  \av{\mathsf{F}\pare{\alpha h'} \ \Lambda^2 h}_{2, \lambda}, 
	\end{equation}
	where the analytic function $ \mathsf{F} $ is defined in \eqref{eq:def_F}. Let us now study the quantity $  \av{\mathsf{F}\pare{\alpha h'} \ \Lambda^2 h}_{s, \lambda}, \ s, \lambda \geqslant 0 $. Providing a bound for such quantity will allow us to conclude the estimate for the term $ N_{h, 2} $. We use the product rule stated in Lemma \ref{lem:interpolation_inequality} to get
	\begin{equation}\label{eq:Nh2_2}
	\av{\mathsf{F}\pare{\alpha h'} \ \Lambda^2 h}_{s, \lambda} \leqslant \av{ \mathsf{F}\pare{\alpha h'}}_{s, \lambda} \av{h}_{2, \lambda} + \av{ \mathsf{F}\pare{\alpha h'}}_{0, \lambda} \av{h}_{s+2, \lambda }. 
	\end{equation}
	At this stage we use Lemma \ref{lem:composition_Wiener_FG} in order to obtain the bound $ \av{ \mathsf{F}\pare{\alpha h'}}_{s, \lambda} \leqslant \alpha \av{h}_{s+1, \lambda}, \ s, \lambda \geqslant 0 $, which combined with \eqref{eq:Nh2_2} gives
	\begin{equation}
	\label{eq:Nh2_3}
	\begin{aligned}
	\av{\mathsf{F}\pare{\alpha h'} \ \Lambda^2 h}_{s, \lambda} & \leqslant \alpha \pare{\av{h}_{s+1, \lambda}\av{h}_{2, \lambda} + \av{h}_{1, \lambda} \av{h}_{s+2, \lambda}}, \\
	&\leqslant 2\alpha \av{h}_{1, \lambda} \av{h}_{s+2, \lambda }.
	\end{aligned}
	\end{equation}
We now insert the estimate \eqref{eq:Nh2_3} into \eqref{eq:Nh2_1} and we find that
	\begin{equation*}
	\av{N_{h ,  2}}_{1, \lambda} \leqslant 2\sqrt{\delta}\nu \alpha^2 \av{h}_{1, \lambda} \av{h}_{4, \lambda}. 
	\end{equation*}
	
	At last we focus on the term $ N_{h, 3} $, we can apply the product rule of Lemma \ref{lem:interpolation_inequality}
	\begin{equation}\label{eq:Nh3_1}
	\av{N_{h, 3}}_{1, \lambda} \leqslant \sqrt{ \delta } \nu\alpha \pare{ \av{\mathsf{G}\pare{ \frac{\varepsilon \Lambda}{\tanh\pare{\Lambda}} h}}_{0} \av{\mathsf{F}\pare{\alpha h'} \ \Lambda^2 h}_{2, \lambda} + \av{\mathsf{G}\pare{ \frac{\varepsilon \Lambda}{\tanh\pare{\Lambda}} h}}_{1, \lambda} \av{\mathsf{F}\pare{\alpha h'} \ \Lambda^2 h}_{1, \lambda} }. 
	\end{equation}
	We combine the estimates \eqref{eq:Nh1_2} and \eqref{eq:Nh2_2} in \eqref{eq:Nh3_1}, which gives
	\begin{equation*}
	\av{N_{h, 3}}_{1, \lambda}  \leqslant \sqrt{ \delta } \nu\alpha^2 \frac{\varepsilon\av{h}_{1, \lambda}}{1-\varepsilon \av{h}_{1, \lambda} } \pare{\Big. \av{h}_{1, \lambda} \av{h}_{4, \lambda} + \av{h}_{2, \lambda} \av{h}_{3, \lambda} }. 
	\end{equation*}
	Using the usual interpolation inequality \eqref{eq:interpolation_inequality}, we obtain the final bound for $ N_{h, 3} $:
	\begin{equation}
	\label{eq:Nh3_2}
	\av{N_{h, 3}}_{1, \lambda} \leqslant 2\sqrt{ \delta } \nu\alpha^2 \frac{\varepsilon\av{h}_{1, \lambda}^2}{1-\varepsilon \av{h}_{1, \lambda} }  \av{h}_{4, \lambda} .
	\end{equation}

	We can now finally provide the desired bound for $ N_h $ combining \eqref{eq:Nh1_4}, \eqref{eq:Nh2_3} and \eqref{eq:Nh3_2}:
	\begin{equation}\label{eq:bound_Nh}
	\av{N_h}_{1, \lambda} \leqslant 2 \sqrt{\delta}\nu \alpha \av{h}_{1, \lambda} \pare{\frac{\pare{1+\alpha}\varepsilon}{1-\varepsilon\av{h}_{1, \lambda}} + \alpha} \ \av{h}_{4, \lambda}. 
	\end{equation}

	\subsection{Uniform energy bounds for solutions of \eqref{eq:Muskat7}}\label{sec:uniform_bounds}

	In this section we finally prove the uniform energy bounds for smooth solutions of \eqref{eq:Muskat7} stemming for suitably small initial data $ h_0\in \mathbb{A}^1 $.

	Let us combine the energy inequality \eqref{eq:Muskat7} with the nonlinear bounds \eqref{eq:bound_Nphi2}, \eqref{eq:bound_Nphi1} and \eqref{eq:bound_Nh}. This gives the energy inequality
	\begin{multline}\label{eq:EI1}
	\ddt \av{h}_{1, \mu t} + \frac{\sqrt{\delta}}{8} \pare{\nu\sqrt{\delta} - 1 } \av{h}_{4, \mu t} 
	\leqslant  48 \ c_0 \sqrt{\delta} 
	\Bigg\{
	\frac{1}{\varepsilon} \pare{1+ \frac{2\varepsilon\av{h}_{1, \mu t}}{1-\varepsilon\av{h}_{1, \mu t}}} \\
	+ 2\delta \av{h}_{1, \mu t} \pare{1+2\varepsilon\av{h}_{1, \mu t}\pare{1+\frac{2\varepsilon\av{h}_{1, \mu t}}{1-\varepsilon\av{h}_{1, \mu t}}}}
	\Bigg\}
	\bra{ \Big. \nu\alpha  \pare{1+ 2 \alpha \av{h}_{1, \mu t}} \av{h}_{4, \mu t}   + \varepsilon \av{h}_{2, \mu t}} \\
	+ 2\sqrt{\delta}
	\bra{
		2 + \sqrt{\delta} \pare{5+ 7 \varepsilon  \bra{ 1+   \frac{2 \varepsilon \av{h}_{1, \mu t} }{1- \varepsilon \av{h}_{1, \mu t}} }  \av{h}_{1, \mu t}  }
	}\av{h}_{1, \mu t} \bra{\nu\alpha \av{h}_{4, \mu t} \pare{1+ 2 \alpha \av{h}_{1, \mu t}}   + \varepsilon \av{h}_{2, \mu t}\Big. }\\
	+ 2 \sqrt{\delta}\nu \alpha \av{h}_{1, \mu t} \pare{\frac{\pare{1+\alpha}\varepsilon}{1-\varepsilon\av{h}_{1, \mu t}} + \alpha} \ \av{h}_{4, \mu t} .
	\end{multline}
	We focus now in simplifying the expression \eqref{eq:EI1}. We recall that
	\begin{align}\label{eq:EI1.1}
	\av{h}_{1, \mu t} \leqslant \frac{c_0   \sqrt{\delta}  }{260 \varepsilon} , && c_0 \in \pare{0  ,  \frac{1}{2\sqrt{\delta}}}. 
	\end{align}
	We stress the fact that the hypothesis in \eqref{eq:EI1.1} are the ones stated in Proposition \ref{prop:elliptic_est_phi2_strip}.

	Since by hypothesis $ \nu\alpha > \varepsilon $ and $ \mathbb{A}^4_{\mu t}\Subset \mathbb{A}^2_{\mu t} $ we can deduce the first simplification for \eqref{eq:EI1}
	\begin{multline}\label{eq:EI2}
	\ddt \av{h}_{1, \mu t} + \frac{\sqrt{\delta}}{8} \pare{\nu\sqrt{\delta} - 1 } \av{h}_{4, \mu t} \\
	\leqslant  96 \ c_0 \sqrt{\delta} 
	\Bigg\{
	\frac{1}{\varepsilon} \pare{1+ \frac{2\varepsilon\av{h}_{1, \mu t}}{1-\varepsilon\av{h}_{1, \mu t}}} 
	+ 2\delta \av{h}_{1, \mu t} \pare{1+2\varepsilon\av{h}_{1, \mu t}\pare{1+\frac{2\varepsilon\av{h}_{1, \mu t}}{1-\varepsilon\av{h}_{1, \mu t}}}}
	\Bigg\}
	\bra{ \Big. \nu\alpha  \pare{1+  \alpha \av{h}_{1, \mu t}} \av{h}_{4, \mu t}   } \\
	+ 4\sqrt{\delta}
	\bra{
		2 + \sqrt{\delta} \pare{5+ 7 \varepsilon  \bra{ 1+   \frac{2 \varepsilon \av{h}_{1, \mu t} }{1- \varepsilon \av{h}_{1, \mu t}} }  \av{h}_{1, \mu t}  }
	}\av{h}_{1, \mu t} \bra{ \Big. \nu\alpha  \pare{1+  \alpha \av{h}_{1, \mu t}} \av{h}_{4, \mu t}   }\\
	+ 2 \sqrt{\delta}\nu \alpha \av{h}_{1, \mu t} \pare{\frac{\pare{1+\alpha}\varepsilon}{1-\varepsilon\av{h}_{1, \mu t}} + \alpha} \ \av{h}_{4, \mu t} .
	\end{multline}

	We recall now the definitions of $\varepsilon$, $\delta$ and $\alpha$ and we use the above relations combined with \eqref{eq:EI1.1} in order deduce that
	\begin{align*}
	1+ \frac{2\varepsilon \av{h}_{1, \mu t}}{1-\varepsilon\av{h}_{1, \mu t}} & \leqslant \frac{3}{2}, \\
	2\delta \av{h}_{1, \mu t} \pare{1+2\varepsilon\av{h}_{1, \mu t}\pare{1+\frac{2\varepsilon\av{h}_{1, \mu t}}{1-\varepsilon\av{h}_{1, \mu t}}}} & \leqslant \frac{\delta}{2}, \\
	1+\alpha\av{h}_{1, \mu t} & \leqslant 1+\frac{\sqrt{\delta}}{2},  \\
	5+ 7 \varepsilon  \bra{ 1+   \frac{2 \varepsilon \av{h}_{1, \mu t} }{1- \varepsilon \av{h}_{1, \mu t}} }  \av{h}_{1, \mu t}   & \leqslant 5+ \frac{1}{2}=\frac{11}{2},  \\
	\bra{
		2 + \sqrt{\delta} \pare{5+ 7 \varepsilon  \bra{ 1+   \frac{2 \varepsilon \av{h}_{1, \mu t} }{1- \varepsilon \av{h}_{1, \mu t}} }  \av{h}_{1, \mu t}  }
	}\av{h}_{1, \mu t} & \leqslant \frac{\sqrt{\delta}}{2\varepsilon}\pare{1+\sqrt{\delta}}c_0, \\
	2  \av{h}_{1, \mu t} \pare{\frac{\pare{1+\alpha}\varepsilon}{1-\varepsilon\av{h}_{1, \mu t}} + \alpha} &\leqslant \frac{c_0 \sqrt{\delta}}{\varepsilon}.
	\end{align*}
Thus
	\begin{equation*}
	\begin{aligned}
	\ddt \av{h}_{1, \mu t} + \frac{\sqrt{\delta}}{8}  \pare{\nu\sqrt{\delta} - 1 } \av{h}_{4, \mu t} \leqslant & \bra{48 \sqrt{\delta}\pare{\frac{3}{\varepsilon} + \delta} + \frac{\delta}{\varepsilon}\pare{2+\sqrt{\delta}}} \nu \alpha \ c_0 \av{h}_{4, \mu t}, \\
	\leqslant & \bra{192 \delta + 3\delta^{3/2}} c_0\nu \av{h}_{4, \mu t}, \\
	\leqslant & 195 \ \delta c_0\nu \av{h}_{4, \mu t}, 
	\end{aligned}
	\end{equation*}
 where in the last two inequalities we used the bound $ \delta \leqslant 1 $. Hence if
	\begin{equation}\label{eq:EI5}
	c_0 \leqslant \frac{\nu\sqrt{\delta} -1 }{3120 \ \nu\sqrt{\delta}}, 
	\end{equation}
	then 
	\begin{equation}\label{eq:ED1}
	\ddt \av{h}_{1, \mu t} +  \frac{\sqrt{\delta}}{16} \pare{\nu\sqrt{\delta} - 1 } \av{h}_{4, \mu t} \leqslant  0.
	\end{equation}
An integration-in-time in $ \bra{0, t} $, gives that
	\begin{equation*}
	\av{h \pare{t}}_{1, \mu t} + \frac{\sqrt{\delta}}{16}  \pare{\nu\sqrt{\delta} - 1 } \int_0^t \av{h  \pare{t'} }_{4, \mu t'} \dt' \leqslant \av{h_0}_{1} . 
	\end{equation*}
	We need yet to check that under the assumption \eqref{eq:EI1.1} and \eqref{eq:EI5} there is no pinch-off between the interface and the bottom, i.e. \eqref{eq:no_pintch-off} is satisfied. This is assured if
	\begin{equation*}
	\av{h\pare{t}}_{L^\infty} < \frac{1}{\varepsilon}. 
	\end{equation*}
	We use the embedding $ \av{h}_{1, \mu t}\geqslant\av{h}_0 \geqslant\av{h}_{L^\infty} $, thus
	\begin{equation*}
	\av{h}_{1, \mu t} \leqslant \frac{\nu\sqrt{\delta} - 1}{ \mathsf{C}_0 \  \nu\varepsilon   } \leqslant \frac{1}{\mathsf{C}_0 \ \varepsilon} < \frac{1}{\varepsilon} , 
	\end{equation*}
	when $ \delta \leqslant 1 $. As a consequence, in such setting there is no pinch off and the estimates are global. In order to prove the exponential decay we consider \eqref{eq:ED1} and usa a Poincar\'e  inequality in order to deduce that
	\begin{equation*}
	\av{h\pare{t}}_{1, \mu t}\leqslant \av{h_0}_1 \exp\set{-\frac{\sqrt{\delta}}{16}\pare{\nu \sqrt{\delta}- 1} t}. 
\end{equation*}	 
	\hfill$ \Box $

	\section{Proof of Theorem \ref{thm:main_nondim}} \label{sec:pf_main}

	We will divide the proof of Theorem \ref{thm:main_nondim} in several steps in order to dlarify the presentation

\subsubsection*{Step 1: regularization of the system and existence of limit function }	
	
	Let us define, for any $ N\in\bN $ the regularizing operator
	\begin{equation*}
	 \widehat{ \cJ_N u } \pare{k} = 1_{\set{\av{k}\leqslant N}} \hat{u} \pare{k},
	\end{equation*}
	and the regularizations of the full system \eqref{eq:Muskat5}:
	
	\begin{equation}\label{eq:Muskat_regularized}
	\left\lbrace
	\begin{aligned}
	& \begin{multlined}
	\partial_t h_N + \frac{1}{\varepsilon} \tanh\pare{\sqrt{\delta} \Lambda} \pare{\nu \alpha \Lambda^3 h_N - \varepsilon \Lambda h_N} \\
	= - \sqrt{\delta} \cJ_N \pare{\big. A_1^k\pare{\nabla \sigma_N }\partial_k \pare{\phi_{1, N} + \phi_{2, N}}} \partial_1 h_N + \frac{1}{\alpha} \pare{\big. \big. A_2^2 \pare{\nabla \sigma_N } \partial_2  \phi_{2, N}} \\ 
	+\sqrt{\delta} \cJ_N \left[
	\mathsf{G}\pare{ \frac{\varepsilon \Lambda}{\tanh\pare{\Lambda}} h_N} \tanh\pare{\sqrt{\delta}\Lambda} \pare{\nu \alpha \Lambda^3 h_N - \varepsilon \Lambda h_N }\right. \\
	\left. - \nu \alpha \pare{1-\mathsf{G}\pare{\frac{\varepsilon \Lambda}{\tanh\pare{\Lambda}} h_N}} \tanh\pare{\sqrt{\delta}\Lambda} \Lambda  \pare{\mathsf{F}\pare{\alpha h'_N} \Lambda^2 h_N } 
	\right]    , 
	\end{multlined} && \text{ on } \Gamma \\[4mm]
	& \phi_{1, N} \pare{x_1, x_2, t} = \ \frac{\cosh\pare{\sqrt{\delta}\pare{1+x_2}\Lambda}}{\cosh\pare{\sqrt{\delta}\Lambda}}\pare{ \big. \nu \alpha \frac{\partial_1^2 h_N }{\pare{1+  \pare{ \alpha \partial_1 h_N }^2}^{3/2}}+ \varepsilon h_N \pare{x_1, t}}&& \text{ in }\cS,  \\[4mm]
	& \nd \cdot \bra{ \pare{I + \varepsilon Q \pare{\nabla \sigma_N}} \nd \phi_{2, N} } =- \varepsilon \nd \cdot  \bra{ Q \pare{\nabla \sigma_N} \phi_{1, N} },  && \text{ in }\cS,  \\
	& \phi_{2, N}  =  0 ,&& \text{ on }\Gamma , \\
	& \partial_2\phi_{2, N} =0 , && \text{ on }\Gamma_{\textnormal{b}}  \\[4mm]
	& \left. h_{N} \right|_{t=0} = h_{N, 0} = \cJ_N h_0, 
	\end{aligned}
	\right. 
	\end{equation}
	where respectively
	\begin{align*}
	\sigma_N \pare{x_1, x_2, t} & = \frac{\sinh\pare{\pare{1+x_2}\Lambda}}{\sinh\pare{\Lambda}} \  h_N \pare{x_1, t} , \\
	A \pare{\nabla \sigma_N }&  =  \frac{1}{1 + \varepsilon \partial_2 \sigma_N } \pare{
		\begin{array}{cc}
		1 + \varepsilon \partial_2 \sigma_N & 0 \\[2mm]
		-\varepsilon \partial_1 \sigma_N & 1
		\end{array}
	}  ,  \\
	Q \pare{\nabla \sigma_N } & =  \pare{
		\begin{array}{cc}
		\partial_2 \sigma_N  & -\sqrt{\delta} \partial_1 \sigma_N \\[2mm]
		-\sqrt{\delta} \partial_1 \sigma_N & \displaystyle \frac{-\partial_2 \sigma_N + \varepsilon\delta \av{\partial_1 \sigma_N }^2}{1+\varepsilon\partial_2 \sigma_N }
		\end{array}
	}.
	\end{align*}
	
	The functions $ \phi_{1, N}, \sigma_N $ and $ Q_N\pare{\nabla \sigma_N} $ are $ \cC^\infty $ in the strip $ \cS $, so that, for any $ M \gg 1 $ Lax-Milgram Theorem produces a solution $ \phi_{2, N} \in H^M \pare{\cS} $ whose trace on $ \Gamma $ is $ \left. \phi_{2, N}\right|_{x_2 =0}\in H^{M-\frac{1}{2}}\pare{\bT} $. We can hence consider the system \eqref{eq:Muskat_regularized} as an ODE of the form
	\begin{align*}
	\dot{h}_N = \cJ_N \cN \pare{\Lambda h_N, \Lambda^3 h_N, \left. \nabla \phi_{1, N}\right|_{x_2 =0}, \left. \nabla \phi_{2, N}\right|_{x_2 =0} }, &&
	h_N\pare{0} = h_{N, 0},  
	\end{align*}
	for an appropriate nonlinear function $ \cN $. For any $ M > 0 $ we can define the Hilbert space
	\begin{equation*}
	\cH^M_N = \set{w \in H^M \pare{\bT} \left| \ \textnormal{supp}\pare{\hat{w}}\subset B_N \pare{0} \right. }, 
	\end{equation*}
	hence by Picard theorem in Banach spaces we can assert that for any $ N $ there exists a maximal $ T_N > 0 $ such that
	\begin{equation*}
	h_N \in \cC^1 \pare{\left[0, T_N \right); \cH^M_N} . 
	\end{equation*}
	
	The estimates performed in Section \ref{sec:apriori_estimates} can still be applied to the regularized system \eqref{eq:Muskat_regularized}. This implies that, setting $ \mu $ as in \eqref{eq:mu}, the sequence $ \pare{h_N}_N $ is uniformly bounded in the energy space
	\begin{align*}
	L^\infty\pare{\bR_+; \mathbb{A}^1_{\mu t}}\cap L^1\pare{\bR_+ ; \mathbb{A}^4_{\mu t}},  && \mu \in \left[0, \frac{\sqrt{\delta}}{4}\pare{\nu\sqrt{\delta} - 1}\right) , 
	\end{align*}
	and for any $ t>0 $ satisfies the energy inequality
	\begin{equation*}
	\av{h_N \pare{t}}_{1, \mu t} + \frac{\sqrt{\delta}}{16}\pare{\nu\sqrt{\delta} - 1}  \int_0^t \av{h_N \pare{t'}}_{4, \mu t'} \dt'\leqslant \av{\cJ_N h_0}_{1} . 
	\end{equation*}

We can also apply the following version of Aubin-Lions lemma (see \cite[Corollary 6]{Simon87}): 
	\begin{lemma}[\cite{Simon87}]
		Let $ X_0, X $ and $ X_1 $ be Banach spaces  such that
		\begin{equation*}
		X_0 \Subset X \hra X_1 , 
		\end{equation*}
		and let $ p\in \left(1, \infty\right] $. Let us consider a set $ \cG $ such that:
		\begin{itemize}
			\item $ \cG $ is uniformly bounded in $ L^p\pare{\bra{0, T}; X } \cap L^1_{\loc}\pare{\bra{0, T}; X_0} $, 
			
			\item $ \partial_t\cG = \set{\partial_t g , \ g\in \cG} $ is uniformly bounded in $ L^1_{\loc} \pare{\bra{0, T}; X_1 } $, 
			
		\end{itemize}
		then $ \cG $ is relatively compact in $ L^q\pare{\bra{0, T}; X}, \ q \in [1, p) $. 
	\end{lemma}

	We can hence set $ X_0 = \mathbb{A}^4, \ X_1 = \mathbb{A}^0, \ X = \mathbb{A}^s, \ s \in [3, 4) $ and we use the compact embedding $ \mathbb{A}^4\Subset \mathbb{A}^s, \ s\in [3, 4) $ proved in Lemma \ref{lem:interpolation_inequality}. Since, by interpolation, we have that
	\begin{align*}
	h_N\in L^{\frac{3}{s-1}}\pare{\bra{0, T}; \mathbb{A}^s}, && \frac{3}{s-1} > 1 ,
	\end{align*}	 
 uniformly in $ N $. We have obtained that (up to a non-relabeled subsequence)
	\begin{align}\label{eq:convergence1}
	h_N \xrightarrow{N\to \infty} h , && \text{in} && L^q \pare{[0, T] ; \mathbb{A}^s}, \ s\in [3, 4), \ q\in \left[ 1, \frac{3}{s-1}\right). 
	\end{align}

	\subsubsection*{Step 2: The limit function is a solution of the Muskat problem}
Due to the strong convergence of the approximate problems, it is rather straightforward to show that the limit function $h$ satisfies the Muskat problem in the weak form. Namely, for every $ \theta \in \cD\pare{\bT\times\bra{0, T}} $ the following quantity
	\begin{multline}\label{eq:equality_WS2}
	\int _{\bT} h\pare{t}\theta\pare{t}\dx - \int _{\bT} h_0 \ \theta\pare{ 0}\dx -
	\int_0^T \int _{\bT}  \pare{ h \  \partial_t \theta  + \frac{1}{\varepsilon}  \tanh\pare{\sqrt{\delta} \Lambda} \pare{\nu \alpha \Lambda^3 h - \varepsilon \Lambda h}  \   \theta} \ \dx\ \dt \\
	- \int_0^T \int _{\bT} \Bigg\lbrace - \sqrt{\delta}\pare{\big. A_1^k \pare{\nabla\sigma} \partial_k \pare{\phi_1 + \phi_2}} \partial_1 h + \frac{1}{\alpha} \pare{\big. \big. A_2^2 \pare{\nabla\sigma} \partial_2  \phi_2} \\
	+\sqrt{\delta} \left[
	\mathsf{G}\pare{ \frac{\varepsilon \Lambda}{\tanh\pare{\Lambda}} h} \tanh\pare{\sqrt{\delta}\Lambda} \pare{\nu \alpha \Lambda^3 h - \varepsilon \Lambda h}\right. \\
	\left. - \nu \alpha \pare{1-\mathsf{G}\pare{\frac{\varepsilon \Lambda}{\tanh\pare{\Lambda}} h}} \tanh\pare{\sqrt{\delta}\Lambda} \Lambda \pare{\mathsf{F}\pare{\alpha h'} \Lambda^2 h} \right] \Bigg\rbrace \theta \ \dx\ \dt=0. 
	\end{multline}
Furthermore, due to its regularity, the function $h$ is a classical solution of the Muskat problem.  

	\subsubsection*{Step 3: Instantaneous space-analiticity of the solution}
	The convergence proved in \eqref{eq:convergence1} implies, in particular, that $ h_N \to h $ in $ L^1\pare{\bra{0, T}; \mathbb{A}^0} $, thus up to a non-relabeled subsequence we obtain that $ \hat{h}_N\pare{n ,  t} \xrightarrow{N\to\infty} \hat{h} \pare{n ,  t} $, $ \forall n\in\mathbb{Z}$ and almost every $ t $. Let us define for $ \mu $ as in \eqref{eq:mu} and $ t > 0 $
	\begin{equation*}
	\cE_\mu \pare{t}= \esssup_{t'\in\bra{0, t}} \av{h\pare{t'}}_{1, \mu t} + \frac{\sqrt{\delta}}{16}\pare{\nu\sqrt{\delta} - 1}  \int_0^t \av{h\pare{t'}}_{4, \mu t'} \dt' \in \bra{0, \infty} . 
	\end{equation*}
Applying Fatou's lemma we deduce that
	\begin{equation*}
	\cE_\mu \pare{t}\leqslant\liminf_{N\to \infty} \pare{\esssup_{t'\in\bra{0, t}} \av{h_N\pare{t'}}_{1, \mu t} + \frac{\sqrt{\delta}}{16}\pare{\nu\sqrt{\delta} - 1}  \int_0^t \av{h_N\pare{t'}}_{4, \mu t'} \dt'} \leqslant  \av{h_0}_{1} < \infty, 
	\end{equation*}
	so, in addition, we obtain that the limit function $ h $ belongs to the space $ L^\infty\pare{\bra{0, T}; A^1_{\mu t}}\cap L^1\pare{\bra{0, T}; A^4_{\mu t}}, $ and for any $ t \in\bra{0, T} $
	\begin{equation}\label{eq:enest_limit_h}
	\esssup_{t'\in\bra{0, t}} \av{h\pare{t'}}_{1, \mu t} + \frac{\sqrt{\delta}}{16}\pare{\nu\sqrt{\delta} - 1}  \int_0^t \av{h\pare{t'}}_{4, \mu t'} \dt'\leqslant  \av{h_0}_{1} . 
	\end{equation}
Using the regularity of $h$ and the fact that it is a classical solution of the Muskat problem, we obtain that 
	\begin{align}\label{eq:UBpat}
	\norm{\partial_t h}_{L^1\pare{[0, T]; \mathbb{A}^1_{\mu t}}} \leqslant C_T.
	\end{align}
	What remains to prove is that the solution is analytic for all times; let us set $ 0< t \leqslant t_1 < t_2 $ we have that
	\begin{equation}\label{eq:ineq_continuity}
	\begin{aligned}
	\av{h\pare{t_2} - h\pare{t_1}}_{1, \mu t} & = \av{\int _{t_1}^{t_2} \partial_th\pare{t'} \dt'}_{1, \mu t}\leqslant \int _{t_1} ^{t_2} \av{\partial_th\pare{t'}}_{1, \mu t'} \dt', 
	\end{aligned}
	\end{equation}
	thus using \eqref{eq:UBpat} we obtain that for any $ \epsilon > 0 $ 
	\begin{equation*}
	h \in \cC\pare{\bra{\epsilon , T}; \mathbb{A}^1_{\mu \epsilon}}, 
	\end{equation*}
	which in turn implies that the solution becomes instantaneously analytic. The inequality \eqref{eq:ineq_continuity} proves in fact a stronger statement: fixed $ \mu $ as in \eqref{eq:mu} the application $ t \mapsto \av{h\pare{t}}_{1, \mu t} $ is right-continuous. We want to prove that such application is continuous. Let us select a $ \mu' \in \pare{0, \mu} $ and let us consider $ 0< t_1 <t_2 $ such that
	\begin{equation*}
	\mu' t_2 < \mu t_1, 
	\end{equation*}
	which we combine with \eqref{eq:ineq_continuity} in order to obtain that
	\begin{equation*}
	\av{h\pare{t_2} - h\pare{t_1}}_{1, \mu' t_2} \xrightarrow{t_1 \to t_2} 0. 
	\end{equation*}
Thus $ h\in \cC \pare{\bra{0, T}; \mathbb{A}^1_{\mu' t}} $, but $ \mu $ and $ \mu' $ are arbitrary elements of an open interval, thus we get that
	\begin{equation*}
	h\in \cC \pare{\bra{0, T}; \mathbb{A}^1_{\mu t}}. 
	\end{equation*}

	This concludes the proof of Theorem \ref{thm:main_nondim}. 
	\hfill$ \Box $

	\section*{Acknowledgments}
The research of F.G. has been partially supported by the grant MTM2017-89976-P (Spain) and by the ERC through the Starting Grant project H2020-EU.1.1.-639227. R. G-B has been funded by the grant MTM2017-89976-P from the Spanish government. The research of S.S. is supported by the Basque Government through the BERC 2018-2021 program and by Spanish Ministry of Economy and Competitiveness MINECO through BCAM Severo Ochoa excellence accreditation SEV-2017-0718 and through project MTM2017-82184-R funded by (AEI/FEDER, UE) and acronym "DESFLU".

	\begin{footnotesize}
	
	\end{footnotesize}


\begin{thebibliography}{10}

\bibitem{AlazardLazar2019}
Thomas Alazard and Omar Lazar.
\newblock Paralinearization of the Muskat equation and application to the
  Cauchy problem.
\newblock {\em arXiv:1907.02138}, 2019.

\bibitem{AlazardOneFluid2019}
Thomas Alazard, Nicolas Meunier, and Didier Smets.
\newblock Lyapounov functions, identities and the Cauchy problem for the
  Hele-Shaw equation.
\newblock {\em arXiv:1907.03691}, 2019.

\bibitem{BCD}
Hajer Bahouri, Jean-Yves Chemin, and Rapha{\"e}l Danchin.
\newblock {\em Fourier analysis and nonlinear partial differential equations},
  volume 343 of {\em Grundlehren der Mathematischen Wissenschaften [Fundamental
  Principles of Mathematical Sciences]}.
\newblock Springer, Heidelberg, 2011.

\bibitem{BKLST1986}
D.~Bensimon, L.~Kadano, S.~Liang, B~Shraiman, and B.~Tang.
\newblock Viscous {F}lows in {T}wo {D}imensions.
\newblock {\em Reviews of {M}odern {P}hysics}, 58(4):977--999, 1986.

\bibitem{BCG-B2014}
Luigi~C. Berselli, Diego C\'{o}rdoba, and Rafael Granero-Belinch\'{o}n.
\newblock Local solvability and turning for the inhomogeneous {M}uskat problem.
\newblock {\em Interfaces Free Bound.}, 16(2):175--213, 2014.

\bibitem{BGB19}
Gabriele Bruell and Rafael Granero-Belinch\'{o}n.
\newblock On the {T}hin {F}ilm {M}uskat and the {T}hin {F}ilm {S}tokes
  {E}quations.
\newblock {\em J. Math. Fluid Mech.}, 21(2):21:33, 2019.

\bibitem{Cameron2019}
Stephen Cameron.
\newblock Global well-posedness for the two-dimensional {M}uskat problem with
  slope less than 1.
\newblock {\em Anal. PDE}, 12(4):997--1022, 2019.

\bibitem{CCF2016}
A.~Castro, D.~Córdoba, and D.~Faraco.
\newblock Mixing solutions for the Muskat problem.
\newblock {\em arXiv:1605.04822}, 2016.

\bibitem{CFM2019}
\'{A}. Castro, D.~Faraco, and F.~Mengual.
\newblock Degraded mixing solutions for the {M}uskat problem.
\newblock {\em Calc. Var. Partial Differential Equations}, 58(2):Art. 58, 29,
  2019.

\bibitem{CCFG2013}
\'{A}ngel Castro, Diego C\'{o}rdoba, Charles Fefferman, and Francisco Gancedo.
\newblock Breakdown of smoothness for the {M}uskat problem.
\newblock {\em Arch. Ration. Mech. Anal.}, 208(3):805--909, 2013.

\bibitem{CCFG2016}
Angel Castro, Diego C\'{o}rdoba, Charles Fefferman, and Francisco Gancedo.
\newblock Splash singularities for the one-phase {M}uskat problem in stable
  regimes.
\newblock {\em Arch. Ration. Mech. Anal.}, 222(1):213--243, 2016.

\bibitem{CCFGL-F2012}
\'{A}ngel Castro, Diego C\'{o}rdoba, Charles Fefferman, Francisco Gancedo, and
  Mar\'{i}a L\'{o}pez-Fern\'{a}ndez.
\newblock Rayleigh-{T}aylor breakdown for the {M}uskat problem with
  applications to water waves.
\newblock {\em Ann. of Math. (2)}, 175(2):909--948, 2012.

\bibitem{Chen1993}
Xinfu Chen.
\newblock The {H}ele-{S}haw problem and area-preserving curve-shortening
  motions.
\newblock {\em Arch. Rational Mech. Anal.}, 123(2):117--151, 1993.

\bibitem{CG-BS2016}
C.~H.~Arthur Cheng, Rafael Granero-Belinch\'{o}n, and Steve Shkoller.
\newblock Well-posedness of the {M}uskat problem with {$H^2$} initial data.
\newblock {\em Adv. Math.}, 286:32--104, 2016.

\bibitem{CCGR-PS2016}
Peter Constantin, Diego C\'{o}rdoba, Francisco Gancedo, Luis
  Rodr\'{i}guez-Piazza, and Robert~M. Strain.
\newblock On the {M}uskat problem: global in time results in 2{D} and 3{D}.
\newblock {\em Amer. J. Math.}, 138(6):1455--1494, 2016.

\bibitem{CCGS2013}
Peter Constantin, Diego C\'{o}rdoba, Francisco Gancedo, and Robert~M. Strain.
\newblock On the global existence for the {M}uskat problem.
\newblock {\em J. Eur. Math. Soc. (JEMS)}, 15(1):201--227, 2013.

\bibitem{CGSV2017}
Peter Constantin, Francisco Gancedo, Roman Shvydkoy, and Vlad Vicol.
\newblock Global regularity for 2{D} {M}uskat equations with finite slope.
\newblock {\em Ann. Inst. H. Poincar\'{e} Anal. Non Lin\'{e}aire},
  34(4):1041--1074, 2017.

\bibitem{CCG2011}
Antonio C\'{o}rdoba, Diego C\'{o}rdoba, and Francisco Gancedo.
\newblock Interface evolution: the {H}ele-{S}haw and {M}uskat problems.
\newblock {\em Ann. of Math. (2)}, 173(1):477--542, 2011.

\bibitem{CordobaGancedo2007}
Diego C\'{o}rdoba and Francisco Gancedo.
\newblock Contour dynamics of incompressible 3-{D} fluids in a porous medium
  with different densities.
\newblock {\em Comm. Math. Phys.}, 273(2):445--471, 2007.

\bibitem{CordobaGancedo2009}
Diego C\'{o}rdoba and Francisco Gancedo.
\newblock A maximum principle for the {M}uskat problem for fluids with
  different densities.
\newblock {\em Comm. Math. Phys.}, 286(2):681--696, 2009.

\bibitem{CGO2008}
Diego C\'{o}rdoba, Francisco Gancedo, and Rafael Orive.
\newblock A note on interface dynamics for convection in porous media.
\newblock {\em Phys. D}, 237(10-12):1488--1497, 2008.

\bibitem{CG-SZ2017}
Diego C\'{o}rdoba, Javier G\'{o}mez-Serrano, and Andrej Zlato\v{s}.
\newblock A note on stability shifting for the {M}uskat problem, {II}: {F}rom
  stable to unstable and back to stable.
\newblock {\em Anal. PDE}, 10(2):367--378, 2017.

\bibitem{CordobaLazar2018}
Diego {C}ordoba and {O}mar {L}azar.
\newblock Global well-posedness for the 2d stable Muskat problem in
  ${H}^{3/2}$.
\newblock {\em Preprint arXiv:1803.07528}, 2018.

\bibitem{Darcy1856}
Henry {D}arcy.
\newblock Les {F}ontaines {P}ubliques de la {V}ille de {D}ijon.
\newblock {\em Dalmont, {P}aris}, 1856.

\bibitem{DuchonRobert1984}
Jean Duchon and Raoul Robert.
\newblock \'{E}volution d'une interface par capillarit\'{e} et diffusion de
  volume. {I}. {E}xistence locale en temps.
\newblock {\em Ann. Inst. H. Poincar\'{e} Anal. Non Lin\'{e}aire},
  1(5):361--378, 1984.

\bibitem{ELM11}
J.~Escher, Ph. Lauren\c{c}ot, and B.-V. Matioc.
\newblock Existence and stability of weak solutions for a degenerate parabolic
  system modelling two-phase flows in porous media.
\newblock {\em Ann. Inst. H. Poincar\'e Anal. Non Lin\'eaire}, 28(4):583--598,
  2011.

\bibitem{escher2012modelling}
J.~Escher, A.-V. Matioc, and B.-V. Matioc.
\newblock Modelling and analysis of the {M}uskat problem for thin fluid layers.
\newblock {\em Journal of Mathematical Fluid Mechanics}, 14(2):267--277, 2012.

\bibitem{escher2013existence}
J.~Escher and B.-V. Matioc.
\newblock Existence and stability of solutions for a strongly coupled system
  modelling thin fluid films.
\newblock {\em NoDEA: Nonlinear Differential Equations and Applications}, pages
  1--17, 2013.

\bibitem{EscherMatioc2011}
Joachim Escher and Bogdan-Vasile Matioc.
\newblock On the parabolicity of the {M}uskat problem: well-posedness,
  fingering, and stability results.
\newblock {\em Z. Anal. Anwend.}, 30(2):193--218, 2011.

\bibitem{EscherSimonett1997}
Joachim Escher and Gieri Simonett.
\newblock Classical solutions for {H}ele-{S}haw models with surface tension.
\newblock {\em Adv. Differential Equations}, 2(4):619--642, 1997.

\bibitem{GG-JPS2019}
F.~Gancedo, E.~Garc\'{i}a-Ju\'{a}rez, N.~Patel, and R.~M. Strain.
\newblock On the {M}uskat problem with viscosity jump: {G}lobal in time
  results.
\newblock {\em Adv. Math.}, 345:552--597, 2019.

\bibitem{GancedoPatel2018}
Francisco {G}ancedo and {N}eel {P}atel.
\newblock On the local existence and blow-up for generalized {S}{Q}{G} patches.
\newblock {\em Preprint arXiv:1811.00530}, 2018.

\bibitem{G-BG-S2014}
Javier G\'{o}mez-Serrano and Rafael Granero-Belinch\'{o}n.
\newblock On turning waves for the inhomogeneous {M}uskat problem: a
  computer-assisted proof.
\newblock {\em Nonlinearity}, 27(6):1471--1498, 2014.

\bibitem{Granero-Belinchon2014}
Rafael Granero-Belinch\'{o}n.
\newblock Global existence for the confined {M}uskat problem.
\newblock {\em SIAM J. Math. Anal.}, 46(2):1651--1680, 2014.

\bibitem{granero2019growth}
Rafael Granero-Belinch{\'o}n and Omar Lazar.
\newblock Growth in the Muskat problem.
\newblock To appear in \emph{Mathematical Modelling of Natural Phenomena,} {\em arXiv preprint arXiv:1904.00294}, 2019.

\bibitem{GS19}
Rafael Granero-Belinch\'{o}n and Stefano Scrobogna.
\newblock Models for damped water waves.
\newblock To appear in \emph{SIAM Journal on Applied Mathematics}, arXiv preprint at \url{https://arxiv.org/abs/1905.07751}.

\bibitem{GS18}
Rafael Granero-Belinch\'{o}n and Stefano Scrobogna.
\newblock On an asymptotic model for free boundary {D}arcy flow in porous
  media.
\newblock Submitted, \url{https://arxiv.org/abs/1810.11798}.

\bibitem{GS_DDZ_2}
Rafael Granero-Belinch\'on and Stefano Scrobogna.
\newblock Well-posedness of a water wave model with viscous effects.
\newblock Submitted, \url{https://arxiv.org/abs/1911.01912v2}

\bibitem{GS18_2}
Rafael Granero-Belinch\'{o}n and Stefano Scrobogna.
\newblock Asymptotic models for free boundary flow in porous media.
\newblock {\em Phys. D}, 392:1--16, 2019.

\bibitem{granero2019well}
Rafael Granero-Belinch{\'o}n and Steve Shkoller.
\newblock Well-posedness and decay to equilibrium for the Muskat problem with
  discontinuous permeability.
\newblock {\em Transactions of the American Mathematical Society},
  372(4):2255--2286, 2019.

\bibitem{GHS2007}
Yan Guo, Chris Hallstrom, and Daniel Spirn.
\newblock Dynamics near unstable, interfacial fluids.
\newblock {\em Comm. Math. Phys.}, 270(3):635--689, 2007.

\bibitem{hadzic2017local}
Mahir Hadzic, Gustavo Navarro, and Steve Shkoller.
\newblock Local well-posedness and global stability of the two-phase Stefan
  problem.
\newblock {\em SIAM Journal on Mathematical Analysis}, 49(6):4942--5006, 2017.

\bibitem{Homsy1987}
G.~M. Homsy.
\newblock Viscous fingering in porous media.
\newblock {\em Ann. Rev. Fluid Mech.}, 19:271--311, 1987.

\bibitem{HLS1994}
Thomas~Y. Hou, John~S. Lowengrub, and Michael~J. Shelley.
\newblock Removing the stiffness from interfacial flows with surface tension.
\newblock {\em J. Comput. Phys.}, 114(2):312--338, 1994.

\bibitem{Kim2003}
Inwon~C. Kim.
\newblock Uniqueness and existence results on the {H}ele-{S}haw and the
  {S}tefan problems.
\newblock {\em Arch. Ration. Mech. Anal.}, 168(4):299--328, 2003.

\bibitem{KRYZ2016}
Alexander Kiselev, Lenya Ryzhik, Yao Yao, and Andrej Zlato\v{s}.
\newblock Finite time singularity for the modified {SQG} patch equation.
\newblock {\em Ann. of Math. (2)}, 184(3):909--948, 2016.

\bibitem{Lannes2013}
David Lannes.
\newblock {\em The water waves problem}, volume 188 of {\em Mathematical
  Surveys and Monographs}.
\newblock American Mathematical Society, Providence, RI, 2013.
\newblock Mathematical analysis and asymptotics.

\bibitem{laurenccot2013gradient}
Ph. Lauren{\c{c}}ot and B.-V. Matioc.
\newblock A gradient flow approach to a thin film approximation of the {M}uskat
  problem.
\newblock {\em Calculus of Variations and Partial Differential Equations},
  47(1-2):319--341, 2013.

\bibitem{laurencot2014thin}
Ph. Lauren{\c{c}}ot and B.-V. Matioc.
\newblock A thin film approximation of the {M}uskat problem with gravity and
  capillary forces.
\newblock {\em Journal of the Mathematical Society of Japan}, 66(4):1043--1071,
  2014.

\bibitem{laurenccot2017finite}
Ph. Lauren{\c{c}}ot and B.-V. Matioc.
\newblock Finite speed of propagation and waiting time for a thin-film {M}uskat
  problem.
\newblock {\em Proceedings of the Royal Society of Edinburgh Section A:
  Mathematics}, 147(4):813--830, 2017.

\bibitem{laurencot2017self}
Ph. Lauren{\c{c}}ot and B.-V. Matioc.
\newblock Self-similarity in a thin film {M}uskat problem.
\newblock {\em SIAM Journal on Mathematical Analysis}, 49(4):2790--2842, 2017.

\bibitem{LuoHou2014}
Guo Luo and Thomas~Y. Hou.
\newblock Toward the finite-time blowup of the 3{D} axisymmetric {E}uler
  equations: a numerical investigation.
\newblock {\em Multiscale Model. Simul.}, 12(4):1722--1776, 2014.

\bibitem{matioc2012non}
B.-V. Matioc.
\newblock Non-negative global weak solutions for a degenerate parabolic system
  modelling thin films driven by capillarity.
\newblock {\em Proceedings of the Royal Society of Edinburgh Section A:
  Mathematics}, 142(5):1071--1085, 2012.

\bibitem{Matioc2019}
Bogdan-Vasile Matioc.
\newblock The {M}uskat problem in two dimensions: equivalence of formulations,
  well-posedness, and regularity results.
\newblock {\em Anal. PDE}, 12(2):281--332, 2019.

\bibitem{NguenPausader2019}
Huy~Q. Nguyen and Benoît Pausader.
\newblock A paradifferential approach for well-posedness of the Muskat problem.
\newblock {\em arXiv:1907.03304}, 2019.

\bibitem{Otto1997}
Felix Otto.
\newblock Viscous fingering: an optimal bound on the growth rate of the mixing
  zone.
\newblock {\em SIAM J. Appl. Math.}, 57(4):982--990, 1997.

\bibitem{PatelStrain2017}
Neel Patel and Robert~M. Strain.
\newblock Large time decay estimates for the {M}uskat equation.
\newblock {\em Comm. Partial Differential Equations}, 42(6):977--999, 2017.

\bibitem{rayleigh1878instability}
Lord Rayleigh.
\newblock On the instability of jets.
\newblock {\em Proceedings of the London mathematical society}, 1(1):4--13,
  1878.

\bibitem{SaffmanTaylor1958}
P.~G. Saffman and Geoffrey Taylor.
\newblock The penetration of a fluid into a porous medium or {H}ele-{S}haw cell
  containing a more viscous liquid.
\newblock {\em Proc. Roy. Soc. London. Ser. A}, 245:312--329. (2 plates), 1958.

\bibitem{Sharp1984}
DH~Sharp.
\newblock An overview of Rayleigh-Taylor instability.
\newblock {\em Physica D: Nonlinear Phenomena}, 12(1-3):3--18, 1984.

\bibitem{SCH2004}
Michael Siegel, Russel~E. Caflisch, and Sam Howison.
\newblock Global existence, singular solutions, and ill-posedness for the
  {M}uskat problem.
\newblock {\em Comm. Pure Appl. Math.}, 57(10):1374--1411, 2004.

\bibitem{Simon87}
Jacques Simon.
\newblock Compact sets in the space {$L^p(0,T;B)$}.
\newblock {\em Ann. Mat. Pura Appl. (4)}, 146:65--96, 1987.

\bibitem{MR3014484}
L\'{a}szl\'{o} Sz\'{e}kelyhidi, Jr.
\newblock Relaxation of the incompressible porous media equation.
\newblock {\em Ann. Sci. \'{E}c. Norm. Sup\'{e}r. (4)}, 45(3):491--509, 2012.

\bibitem{taylor1950instability}
Geoffrey~Ingram Taylor.
\newblock The instability of liquid surfaces when accelerated in a direction
  perpendicular to their planes. i.
\newblock {\em Proceedings of the Royal Society of London. Series A.
  Mathematical and Physical Sciences}, 201(1065):192--196, 1950.

\end{thebibliography}
\end{document}